\documentclass[a4paper,12pt,leqno]{article}
\usepackage{amsmath}
\usepackage{amsthm}
\usepackage{amssymb}
\usepackage{amscd}
\usepackage{graphicx, color}
\usepackage[all]{xy}

\newtheorem{thm}{Theorem}[section]
\newtheorem{dfn}[thm]{Definition}
\newtheorem{lmm}[thm]{Lemma}%[section]
\newtheorem{crl}[thm]{Corollary}%[section]
\newtheorem{example}[thm]{Example}%[section]
\newtheorem{rmk}[thm]{Remark}%[section]
\numberwithin{equation}{section}

\usepackage[all]{xy}

\newcommand{\Hol}{\mbox{{\rm Hol}}}

\newcommand{\Z}{\Bbb Z}
\newcommand{\C}{\Bbb C}
\newcommand{\R}{\Bbb R}

\newcommand{\F}{\Bbb F}
\newcommand{\N}{\Bbb N}
\renewcommand{\P}{{\rm P}}
\newcommand{\SP}{\mbox{{\rm SP}}}

\newcommand{\Map}{\mbox{{\rm Map}}}

\newcommand{\CP}{\mathbb{C}\mathbb{P}}

\newcommand{\dis}{\displaystyle}
\newcommand{\p}{\prime}

\newcommand{\I}{\mbox{{\rm (i)}}}
\newcommand{\II}{\mbox{{\rm (ii)}}}
\newcommand{\III}{\mbox{{\rm (iii)}}}
\newcommand{\IV}{\mbox{{\rm (iv)}}}
\newcommand{\V}{\mbox{{\rm (v)}}}

\newcommand{\XS}{X_{\Sigma}}

\newcommand{\GS}{G_{\Sigma}}
\newcommand{\n}{{\bf n}}
\newcommand{\dmin}{d_{\rm min}}
\newcommand{\rmin}{r_{\rm min}}
\newcommand{\KS}{\mathcal{K}_{\Sigma}}
\newcommand{\T}{\Bbb T}

\newcommand{\SZ}{{\mathcal{X}}^{D}}

\newcommand{\po}{\mbox{{\rm Poly}}}

\title{\bf
Spaces of non-resultant systems of bounded multiplicity 
determined by 
a toric variety
}
\author{Andrzej Kozlowski\footnote{%
Institute of Applied Mathematics and Mechanics,
University of Warsaw, Banacha 2, 02-097 Warsaw, Poland
(E-mail: akoz@mimuw.edu.pl)
}
\  and \ 
Kohhei Yamaguchi\footnote{%
Department of Mathematics,
University of Electro-Communications,  Chofu, Tokyo 182-8585, Japan
(E-mail: kohhei@im.uec.ac.jp)
}}

\date{}
\begin{document}
\maketitle

%%%(Abstract)%%%%%%%%
\begin{abstract}
%%%%%%%%%%%%%%%%%
The space %$\po^{D,\Sigma}_n$
of non-resultant systems 
of bounded multiplicity for  a toric variety  $X$ is a generalization of  the space of rational curves on 
it.
In our earlier work \cite{KY8} we proved a homotopy stability theorem and determined explicitly the homotopy type of this space 
for the case $X =\CP^m$. 
%when $X$ is the complex projective space $\CP^m$. 
%in the case 
%where the toric variety is the complex projective space $\CP^m$.
In this paper we consider the case of
%the case of 
 a general  non-singular toric variety and prove a 
homotopy stability theorem generalising the one for $\CP^m$.

\end{abstract}
%%

%%
%% Start line numbering here if you want
%%
%\linenumbers

%% main text
%%
%%%%(SECTION 1: Introduction)%%%%
%%%%%
%%%%(SECTION 1: Introduction)%%%%
%%%%%
\section{Introduction}\label{section 1}
%%%%%
%%%

%%%
%%%%%%%%%%%%%%%%%%%%%%%%%%%
%%%
%\par\vspace{2mm}\par
%%%%%
%%%%%
%%%%%
%%%%%
%\paragraph{Historical survey.}
%%%%%%
For a complex manifold $X$, let $\Map^*(S^2,X)=\Omega^2X$ (resp. $\Hol^*(S^2,X)$)
denote the space of all base point preserving  continuous maps
(resp. base point preserving holomorphic maps)
from the Riemann sphere $S^2$ to $X$.
The relationship between the topology of the space $\Hol^*(S^2,X)$ and that of the space
$\Omega^2X$ has played a significant role in several different areas of geometry and 
mathematical physics  (e.g. \cite{AJ}, \cite{A}).
In particular there arose the problem of whether the inclusion
$\Hol^*(S^2,X)\to \Omega^2X$ is a homotopy equivalence (or homology equivalence)
up to a certain dimension,
which we will refer to as the stability dimension.
Since G. Segal \cite{Se} studied this problem for the case $X=\CP^m$,
a number of mathematicians have investigated various closely related ones.
In particular, M. Guest \cite{Gu2} obtained a generalization of Segal's result to the case of compact non-singular toric varieties $X$.
More generally, J. Mostovoy and E. Munguia-Vilanueva \cite{MV}
generalized the result of Guest to the case of spaces of holomorphic maps from $\CP^m$ to a compact
non-singular toric variety $X$ for $m\geq 1$ and they also improved the 
homology stability dimension of Guest for the case $m=1$.
The authors \cite{KY9} also generalized the result of Mosotovy-Vilanueva for
 {\it non-compact} non-singular toric varieties $X$ for $m=1$
(see  Theorem \ref{thm: I:KY9} in detail).
\par
Similar stabilization results appeared in the work of Arnold (\cite{Ar1}, \cite{Ar2}), and Vassiliev 
(\cite{Va}, \cite{Va2}) in connection with singularity theory.  They considered spaces of polynomials without  roots of multiplicity greater than a certain  natural number. These spaces are examples of \lq\lq complement of discriminants\rq\rq\  in Vassiliev's terminology \cite{Va}. 
In fact, a part of Segal's proof in \cite{Se} was  based on an argument due to Arnold. 
The work of  B. Farb and J. Wolfson \cite{FW}  was  inspired by this argument, and 
they introduced a new family of spaces
$\po^{d,m}_n(\mathbb{F})$.
 %which is defined for every field  $\F$, 
They simultaneously generalized the ones studies by Segal, Arnold and Vassiliev, and
they obtained algebro-geometric and arithmetic refinements of their topological results.
Recall the definition of the space $\po^{d,m}_n(\mathbb{F})$  as follows.
%\par\vspace{1mm}\par
%%%%%
%%%
%%%
%%%(Definition 1.1)%%%%%
\begin{dfn}[\cite{FW}]\label{dfn: FW}
%%%%%%%%%%%%%%%%%%%%%%%%
{\rm
Let $\N$ be the set of all positive integers.
For a field $\mathbb{F}$ with its algebraic closure $\overline{\mathbb{F}}$ and 
a pair  
$(m,n)\in \N^2$ with $(m,n)\not= (1,1)$,
let $\po^{d,m}_n(\mathbb{F})$ denote the space of 
all $m$-tuples $(f_1(z),\cdots ,f_m(z))\in \mathbb{F}[z]^m$ of $\mathbb{F}$-coefficients monic polynomials of the same degree $d$ such that
 polynomials $f_1(z),\cdots ,f_m(z)$ have
no common root $\alpha \in \overline{\F}$  of multiplicity $\geq n$.
\qed
}
%%%
\end{dfn}
%%%%(End of Definition 1.1)%%%%
%%
Note that the space
$\po^{d,m}_n(\mathbb{F})$ can be identified with 
$\Hol_d^*(S^2,\CP^{m-1})$ for $(\F,n)=(\C,1)$,
where $\Hol_d^*(S^2,\CP^{m-1})$ denotes the space of base point preserving  holomorphic maps
$f:S^2\to\CP^{m-1}$ of degree $d$.
Thus,  the space $\po^{d,m}_n(\C)$ may be regarded 
as a generalizations of the space $\Hol^*(S^2,\CP^{m-1})$.
%for $X=\CP^{m-1}$.
\par\vspace{1mm}\par
For a monic polynomial $f(z)\in \F [z]$ of degree $d$, let $F_n(f)(z)$
denote the $n$-tuple of monic polynomials in $\F [z]$ of the same degree $d$
defined by
%%(1.1)%%%%%%%%
\begin{equation}\label{eq: Fn}
F_n(f)(z)=(f(z),f(z)+f^{\p}(z),f(z)+f^{\p\p}(z),\cdots ,f(z)+f^{(n-1)}(z)).
\end{equation}
%\par
In an earlier paper \cite{KY8} we determined the homotopy type of the space
$\po^{d,m}_n(\F)$ explicitly for the case $\F =\C$ and obtained the following homotopy stability result.
%%
%%
%%
%%%(Theorem 1.2)%%%
\begin{thm}[\cite{KY8}]\label{thm: KY8}
%%%%
Let $d,m,n\in \N$ be positive integers with $(m,n)\not=(1,1)$, and
let
%%%%(1.2)%%%
%\begin{equation} 
%%%%%%%%
$i^{d,m}_n:\po^{d,m}_n(\C) \to 
\Omega_d^2\CP^{mn-1}\simeq \Omega S^{2mn-1}
$
%\end{equation}
%%%%%%
denote the natural map given by
%%(1.2)%%
\begin{equation}\label{eq: jet map}
%%%%%
i^{d,m}_n(f)(\alpha)
=
\begin{cases}
[F_n(f_1)(\alpha):F_n(f_2)(\alpha):\cdots :F_n(f_m)(\alpha)] & \mbox{ if }\alpha \in \C
\\
[1:1:\cdots :1] & \mbox{ if }\alpha=\infty
\end{cases}
\end{equation}
%%%%%
for $(f,\alpha )=((f_1(z),\cdots ,f_m(z)),\alpha)\in \po^{d,m}_n(\C)\times S^2$,
where
we identify $S^2=\C\cup \infty$. 
%and
%$F_n(f_i(z))$ denotes the $n$-tuple of monic polynomials of the same degree $d$ defined by
%%%%(1.3)%%%%
%\begin{equation}\label{eq: jet}
%%%%%
%F_n(f_i(z))=(f_i(z),f_i(z)+f_i^{\p}(z),f_i(z)+f_i^{\p\p}(z),\cdots
%,f_i(z)+f_i^{(n-1)}(z)).
%\end{equation}
%%%
\par
Then the map $i^{d,m}_n$ is a homotopy equivalence through dimension
$D(d;m,n)$ if $(m,n)\not=(1,2)$ and it is a homology equivalence through dimension
$\lfloor \frac{d}{n}\rfloor$ if $(m,n)=(1,2)$,
where $\lfloor x\rfloor$ denotes the integer part of a real number $x$ and
the positive integer $D(d;m,n)$ is given by
%%()%%
%\begin{equation}\label{eq: D(d;m,n)}
%%
$D(d;m,n)=(2mn-3)(\lfloor \frac{d}{n}\rfloor +1)-1.$
\qed
%%%%
\end{thm}
%%%%%%(End of Theorem 1.2)%%%
%%%%
%%(Remark 1.3)%%
\begin{rmk}
{\rm
A map $f:X\to Y$ is called 
{\it a homotopy equivalence through dimension} $N$
(resp. 
{\it a homology equivalence through dimension} $N$)
if the induced homomorphism
$
f_*:\pi_k(X)\to \pi_k(Y)$
$(\mbox{resp. }f_*:H_k(X;\Z) \to H_k(Y;\Z))
$
is an isomorphism for any $k\leq N$.
\qed
}
\end{rmk}
%%%%(End of Remark 1.3)%%%
%%%%
%%%%
%%%%
%%%%%%%%%%%%%%%%%%
Our aim of this paper is to further generalize the above result to the case where the conditions on the roots are given in terms of the combinatorial information contained in a non-singular
toric variety $\XS$,
where $\Sigma$ denotes a fan in $\R^m$ and let $\XS$ be the toric variety
associated to $\Sigma$.
%which is a more general affine variety defined by the systems of 
%polynomial equations (the generalized resultants) of integer coefficients. 
%For this purpose,  we 
%consider the affine variety 
%$\po^{D,\Sigma}_n(\mathbb{F})$
%as one of generalizations of the space $\Hol_D^*(S^2,\XS)$ for a 
%simply connected non-singular toric variety $\XS$, 
%%
%For this purpose, let us consider the following space.
%%
%%
%%
%%(Definition 1.4)%%%
\begin{dfn}\label{def: poly}
%%%%%%%%%%%%%
{\rm
Let $\F$ be a field with its algebraic closure $\overline{\F}$, and
let $\Sigma$ be a fan in $\R^m$ such that 
$\Sigma (1)=\{\rho_1,\cdots ,\rho_r\}$, where $\Sigma (1)$ denotes
the set of all one dimensional cones in $\Sigma$.\footnote{%
%%%(Footnote 1)%%%%
The precise definition and notation concerning fans and toric varieties
are explained in \S \ref{section: main results}.}
%%%(End of Footnote 1)%%%%%%%%
\par
For each $r$-tuple $D=(d_1,\cdots ,d_r)\in \N^r$,
let $\po^{D,\Sigma}_n(\mathbb{F})$ denote the space of all
$r$-tuples $(f_1(z),\cdots ,f_r(z))\in \mathbb{F}[z]^r$ of 
$\F$-coefficients monic polynomials satisfying the following
two conditions:
\begin{enumerate}
%%%(1.2.1)%%
\item[(\ref{eq: jet map}.1)]
%%%%%%%%%%%%
$f_i(z)\in \F [z]$ is an $\mathbb{F}$-coefficients monic polynomial of the degree $d_i$
for each $1\leq i\leq r$.
%%(1.2.2)%%
\item[(\ref{eq: jet map}.2)]
%%%%%%%%%%%
For each $\sigma =\{i_1,\cdots  ,i_s\}\in I(\KS)$,
polynomials $f_{i_1}(z),\cdots ,f_{i_s}(z)$ have no common root 
$\alpha\in \overline{\F}$ of multiplicity
$\geq n$,
where
$\KS$ denotes the underlying simplicial complex of $\XS$ on the index set
$[r]=\{1,2,\cdots ,r\}$ defined by (\ref{eq: KS}) and
$I(\KS)$ is the set
$I(\KS)=\{\sigma\subset [r]:\sigma \not\in \KS\}$
defined by (\ref{eq: IK}).
\qed
%%%%%
\end{enumerate}
%%%%
}
%%%%%%
\end{dfn}
%%End of Definition 1.4)%%%%%
%%
%%
%%%%%%%%%
%%%%%(Remark 1.5)%%%
\begin{rmk}
%%%%%%%%%%%%%%%%%%%
{\rm
(i)
By using the classical theory of resultants, one can show that 
$\po^{D,\Sigma}_n(\F)$ is an affine variety over $\F$ which is the complement of the set of solutions of a system of 
 polynomial equations (called a generalised resultant) with integer coefficients. This is why we call it
{\it the space of non-resultant systems of bounded multiplicity of type}  $(\Sigma,n)$.
\par
(ii)
When $\XS$ is a simply connected non-singular toric variety
(over $\C$)
satisfying the condition
(\ref{equ: homogenous}.1),
one can show that
$\po^{D,\Sigma}_n(\C)=\Hol_D^*(S^2,\XS)$
if $n=1$
%%%%%%%%%%%%%%%%%%%
(see Definition \ref{dfn: holomorphic} for the details).
%%%(End of Footnote 1)%%%%
%%%%%%%%%%%%%%%%%%%%%%%%%%
Moreover,  note that $\po^{d,m}_n(\mathbb{F})=\po^{D,\Sigma}_n(\mathbb{F})$ 
when $\XS =\CP^{m-1}$ and $D=(d,d,\cdots ,d)$.
So we may regard the space $\po^{D,\Sigma}_n(\F)$ as a
generalization of the  spaces $\po^{d,m}_n(\F)$ and
$\Hol_D^*(S^2,\XS)$.
\qed
}
%%%%%%%%
\end{rmk}
%%%%%%%(End of Remark 1.5)%%%%
%%%%
%%%%
%%%%
%%%%%%%%%%%%%%%%%
The  principal motivation of this paper
is to generalize the above result (Theorem \ref{thm: KY8}) for the space
$\po^{D,\Sigma}_n(\C)$.
From now on, we write
%%(1.4)%%
\begin{equation}
%%%%
\po^{D,\Sigma}_n=\po^{D,\Sigma}_n(\C)
\quad
\mbox{ for }\F =\C .
%%%
\end{equation}
%%%%%%%%%
Then the main result of this paper is the following
 
%%%%%
%%
%%
%%
%%
%%
%%%%(Theorem 1.6)%%%
\begin{thm}[Theorem \ref{thm: I}]
\label{thm: main result} 
%%%%%%%%%%%%%%%%%%%%%%
Let $D=(d_1,\cdots ,d_r)\in \N^r$,
$n\geq 2$,
and let $\XS$ be an $m$ dimensional simply connected non-singular toric variety 
such that 
the condition $($\ref{equ: homogenous}.1$)$ holds.
%%%
\par
%%(i)%%
$\I$
If $\sum_{k=1}^rd_k\textbf{\textit{n}}_k= {\bf 0}_m$, then
the natural map (given by (\ref{eq; def iD}))
$$
i_D:\po^{D,\Sigma}_n \to 
\Omega^2_D\XS (n)\simeq
\Omega^2_0\XS (n)\simeq 
\Omega^2\mathcal{Z}_{\KS}(D^{2n},S^{2n-1})
$$ 
is a homotopy equivalence
through dimension $d(D;\Sigma ,n)$, where
$d(D;\Sigma,n)$ denotes the positive integer defined in
(\ref{eq: dDSigma}), and
the spaces
$\XS (n)$ and $\mathcal{Z}_K(X,A)$
are 
the orbit space and the polyhedral product of a pair
$(X,A)$ given by (\ref{eq: XSigma(n)}) and
Definition \ref{dfn: polyhedral product}, respectively.
%%%
\par
%%(ii)%%
$\II$
If
$\sum_{k=1}^rd_k\textbf{\textit{n}}_k\not= {\bf 0}_m$,
there is a map
$$
j_D:\po^{D,\Sigma}_n \to 
%\Omega^2_0\XS (n)\simeq
%\Omega^2_D\XS (n)\simeq 
\Omega^2\mathcal{Z}_{\KS}(D^{2n},S^{2n-1})
$$
which
is a homotopy equivalence
through dimension $d(D;\Sigma ,n)$.
%\par
%Here,
%%$\lfloor x\rfloor$ denotes the integer part of a real number $x$ and
%$d(D;\Sigma,n)$ denotes the positive integer given as in (\ref{eq: dDSigma}),
%%%(1.5)%%
%\begin{equation}
%%%%%%%%
%d(D;\Sigma ,n)=(2n\rmin (\Sigma)-3)\lfloor \frac{d_{min}}{n}\rfloor -2.
%\end{equation}
%%%%
\end{thm}
%%%(End of Theorem 1.5)%%%%%
%%%
%%%
%%%
This paper is organized as follows.
%%(SECTION 2)%%
In \S \ref{section: main results} we recall several basic  definitions 
and facts about toric varieties and 
holomorphic curves on toric varieties, which will be in the statements of the results of this paper.
Precise statements of the main results  are
stated after these basic definitions and facts.
%\par
%%(SECTION 3)%%%
In \S \ref{section: polyhedral products} we recall several basic facts related polyhedral
products and toric varieties.
%%
%%%(SECTION 4)%%%%%%%%
In \S \ref{section: simplicial resolution}, 
we summarize  the definitions of the non-degenerate simplicial resolution
and the associated truncated simplicial resolution. 
Then
we construct the Vassiliev spectral sequence converging to  $H_*(\po^{D,\Sigma}_n;\Z)$ by using them,
and prove the homotopy stability result 
(Theorem \ref{thm: III}, Corollary \ref{crl: III}).
%%(Section 6)%%
In \S \ref{section: scanning maps} we consider the configuration model 
%$E^{\Sigma}_n(X,A)$
for $\po^{D,\Sigma}_n$ and
recall the stabilized scanning map.
Furthermore, we investigate
the space
$E^{\Sigma}_n(\overline{U},\partial \overline{U})$
%(\overline{U},\partial \overline{U})$
and show that it is homotopy equivalent to the Davis-Januszkiewicz space 
$DJ(\KS (n))$.
%%
%\par
%%(section 7)%%
In \S \ref{section: stability} we give the proof of stability result
(Theorem \ref{thm: IV}) by using the stabilized scanning map,
and
%%%(Section 8)%%
%\par
finally in \S \ref{section: proofs} we give the proofs of the main results
(Theorem \ref{thm: I}, 
Corollary \ref{crl: II}).
%Finally in \S \ref{section: Appendix} we study about the homotopy type of the space
%$E^{\Sigma}_n(\overline{U},\partial \overline{U})$
%and show that it is homotopy equivalent to the Davis-Januszkiewicz space of certain
%simplicial complex $K(\Sigma,n)$.
%Although this result is not need to proving the main results, it seems interesting from the point of view of toric topology.
%%%%%%%%%%%%%%%%%%%%%%%%%%%%%%
%%%
%%%(End of SECTION 1)%%%
%%
%%%%%%%%%%%%%%%%%

%%%%%%%%%%%%%%%%%%%%%%%%%%%%%%
%%
%%
%%
%%%%%%%%%%%%%%%%%%%%%%%%%%%%%%
%%%(Section 2: 2: Basic related facts and the main results)%%%
\section{Toric varieties and the main results}
\label{section: main results}
%%%%%%%%%%%%%%%%%%%%%%%%%%%%%%%%%%%
%%%
%%%
%%%
%%%
%%%
In this section we recall several basic definitions and facts related to toric varieties
(convex rational polyhedral cones, toric varieties, a fan of toric variety,
polyhedral products,homogenous coordinate, rational curves on a toric variety etc)
and give precise statements of the main results of this paper.
%%%%%%%%%%%%%%%%%%%%%%%%%%%%%%%%%%%
\paragraph{Fans and toric varieties}
%%%%%%%%%%%%%%%%%%%%%%%%%%%%%%%%%%%
%%%
%%%
%%%
%%%%%%%%%%%%%%%%%%%%%%%%%%%%%%%%%%%%%%%%%%%%%%%%%%%%%%%%%%
{\it A convexex rational polyhedral cone}
in $\R^m$
is a subset of $\R^m$ of the form
%%%%%%%%
%%%(2.1)%%
\begin{equation}\label{eq: cone}
%%%%%%
\sigma = \mbox{Cone}(S)=
\mbox{Cone}(\textit{\textbf{m}}_1,\cdots,\textit{\textbf{m}}_s)
%=
%\sum_{k=1}^s\R_{\geq 0}\cdot \textbf{\textit{m}}_k
=
\{\sum_{k=1}^s\lambda_k\textit{\textbf{m}}_k:\lambda_k\geq 0
%\mbox{ for any }1\leq k\leq s\}
\}
\end{equation}
%%%
for a finite set $S=\{\textit{\textbf{m}}_1, \cdots ,
\textit{\textbf{m}}_s\} \subset\Z^m$.
The dimension of $\sigma$ is the dimension of
the smallest subspace of $\R^m$ which contains $\sigma$.
A convex rational polyhedral cone $\sigma$
is called
{\it  strongly convex} if
$\sigma \cap (-\sigma)=\{{\bf 0}_m\}$,
where
we set ${\bf 0}_m={\bf 0}=(0,0,\cdots, 0)\in \R^m.$
%%
%%%
{\it A face} $\tau$ of 
a convex rational polyhedral cone
$\sigma$ is a subset $\tau\subset \sigma$ of the form
$\tau =\sigma\cap \{\textit{\textbf{x}}\in\R^m:L(\textbf{\textit{x}})=0\}$ 
for some linear form $L$ on $\R^m$, such that
$\sigma \subset \{\textit{\textbf{x}}\in\R^m:
L(\textit{\textbf{x}})\geq 0\}.$
If we set 
$\{k:1\leq k\leq s, L(\textit{\textbf{m}}_k)=0\}=\{i_1,\cdots ,i_t\},$
we easily see that $\tau =\mbox{Cone}(\textit{\textbf{m}}_{i_1},
\cdots , \textit{\textbf{m}}_{i_t})$.
Hence, if $\sigma$ is a strongly convex rational polyhedral cone,
so is any of its faces.\footnote{%
%%%(Footnote 2)%%%
When $S$ is the emptyset $\emptyset$,
we set $\mbox{Cone}(\emptyset)=\{{\bf 0}_m\}$
and we may also regard it as one of strongly convex rational polyhedral cones
in $\R^m$.
}
%%(End of Footnote 2)%%%%%%
\par
Let $\Sigma$ be a finite collection of strongly convex rational polyhedral cones
in $\R^m$.
Then it 
is called {\it a fan} (in $\R^m$) if
the following two conditions (\ref{eq: cone}.1) and (\ref{eq: cone}.2) 
%(\ref{eq: cone}.1) and (\ref{eq: cone}.2) 
are satisfied:
%%()%%%
\begin{enumerate}
%%%%(2.1.1)%%%%%%
\item[(\ref{eq: cone}.1)]
Every face $\tau$ of $\sigma\in \Sigma$ belongs to $\Sigma$.
%%%%(2.1.2)%%%%
\item[(\ref{eq: cone}.2)]
%%%%
If $\sigma_1,\sigma_2\in \Sigma$, $\sigma_1\cap \sigma_2$ is
a common face of each $\sigma_k$ and
$\sigma_1\cap\sigma_2\in \Sigma$.
\end{enumerate}
%%%%%%%%%%%%%
%%%%%
\par
An $m$ dimensional irreducible normal  variety
$X$ (over $\C$) is called {\it a toric variety}
if it has a Zariski open subset
 $\T^m_{\C}=(\C^*)^m$ and the action of $\T^m_{\C}$ on itself
extends to an action of $\T^m_{\C}$ on $X$.
The most significant property of a toric variety
is that it is characterized up to isomorphism entirely by its 
associated fan 
$\Sigma$. 
We denote by $\XS$ the toric variety associated to a fan $\Sigma$
(see \cite{CLS} for the details).
\par
It is well known that
there are no holomorphic maps
$\CP^1=S^2\to \T^m_{\C}$ except the constant maps, and that
the fan $\Sigma$ of $\T^m_{\C}$ is $\Sigma =\{{\bf 0}_m\}$.
%${\bf 0}_m=(0,0,\cdots ,0,0)\in \C^m.$
Hence, without loss of generality
we always assume that $\XS\not=\T^m_{\C}$ and that
any fan $\Sigma$ in $\R^m$
satisfies the condition
%%%()%%
%\begin{equation}\label{eq: assumption0}
$\{{\bf 0}_m\}\subsetneqq \Sigma .$
%\end{equation}
%%%%
%%%%%%%%%%%%%%%%%%%%%%%%%%%%%%%%%%%%%%%%%%%%%%%%%%%%%%%%%%%%%%
\paragraph{Polyhedral products}
%%%%%%%%%%%%%%%%%%%%%%%%%%%%%%%%%%%%%%%%%%%%%%%%%%%%%%%%%%%%%
Now recall the basic definitions concerning polyhedral products and related spaces.
%%%
%%%
%%%
%%%%%
%%%
%%%
%%%
%%%(Definition 2.1)%%%
\begin{dfn}\label{dfn: polyhedral product}
%%%%%%%%%%%%%%%%%%%%%
{\rm
Let $K$ be a simplicial complex on the index set $[r]=\{1,2,\cdots ,r\}$,%
%%%%%%%%%%%%%%%%%%
%%%(Footnote 2)%%%%%%
\footnote{%
Let $K$ be some set of subsets of $[r]$.
Then the set $K$ is called {\it an abstract simplicial complex} on the index set $[r]$ if
the following condition holds:
 if $\tau \subset \sigma$ and $\sigma\in K$, then $\tau\in K$.
In this paper by a simplicial complex $K$ we always mean an  \textit{an abstract simplicial complex}, 
and we always assume that a simplicial complex $K$  contains the empty set 
$\emptyset$.
}
%%%%%%%%%%%
%%%(End of FootNote 2)%%%%
%%%%%
and let $(X,A)$
%$(\underline{X},\underline{A})=\{(X_1,A_1),\cdots ,(X_r,A_r)\}$ 
be a  pairs of based spaces.
%such that $A\subset X$.
\par\vspace{1mm}\par
(i) Let $I(K)$ denote the some collection of subsets $\sigma\subset [r]$ defined by
%%(2.2)%%
\begin{equation}\label{eq: IK}
I(K)=\{\sigma \subset [r]:\sigma\notin K\}.
\end{equation}
\par
(ii)
Define
 {\it the polyhedral product} $\mathcal{Z}_K
(X,A)$ 
%of an $r$-tuple of pairs of spaces
%$(\underline{X},\underline{A})$ 
with respect to $K$
by
%%%%(2.3)%%%%%%%%
\begin{align}\label{eq: polyhedral p}
%%%%%%%%%%%%%%%%
\mathcal{Z}_K(X,A)
&=
\bigcup_{\sigma\in K}(X,A)^{\sigma},
\qquad
\mbox{where}
\\
(X,A)^{\sigma}
&=
\{(x_1,\cdots ,x_r)\in X^r:
x_k\in A\mbox{ if }k\notin \sigma\}.
\nonumber
%%%%%%%%%%%
\end{align}
%%%%%%%%%%%%%%
%%
%When $(X_i,A_i)=(X,A)$ for each $1\leq i\leq r$, we write
%$\mathcal{Z}_K(X,A)$ for $\mathcal{Z}_K(\underline{X},\underline{A}).$
%%%%
\par
(iii)
For each subset $\sigma =\{i_1,\cdots ,i_s\}\subset [r]$, let
$L_{\sigma}(\C^n)$ denote the subspace of $\C^{nr}$ defined by
%%(2.4)%%
\begin{align}
%%%%%%%
L_{\sigma}(\C^n)
&=
\{(\textbf{\textit{x}}_1,\cdots ,\textbf{\textit{x}}_r)\in\C^{nr}:
\textbf{\textit{x}}_i\in \C^n,\ 
\textbf{\textit{x}}_{i_1}=\cdots =\textbf{\textit{x}}_{i_s}=
{\bf 0}_n\}
%\\
%&=
%\{(\textbf{\textit{x}}_1,\cdots ,\textbf{\textit{x}}_r)\in\C^{nr}:
%\textbf{\textit{x}}_i\in \C^n,\ 
%\textbf{\textit{x}}_j={\bf 0}_n \mbox{ if }j\in\sigma\}
%\nonumber
\end{align}
%%%%%%
and let $L_n(\Sigma)$ denote the subspace of $\C^{nr}$ defined by
%%(2.5)%%
\begin{equation}\label{eq: L(Sigma)}
%%%%%
L_n(\Sigma )=\bigcup_{\sigma\in I(K)}L_{\sigma}(\C^n)
=\bigcup_{\sigma\subset [r],\sigma\notin K}L_{\sigma}(\C^n).
%%%
\end{equation}
%%%%%
%where we set
%%%%(2.5)%%
%\begin{equation}\label{eq: IK}
%%%%
%I(K)=\{\sigma\subset [r]:\sigma\notin K\}.
%\end{equation}
%%%
Then it is easy to see that}
%%(2.6)%%
\begin{equation}\label{eq: LSigma}
%%%%%%%
\mathcal{Z}_K(\C^n,(\C^{n})^*)
=
\C^{nr}\setminus
L_n(\Sigma),
\ \ 
\mbox{\rm where }(\C^{n})^*=\C^n\setminus \{{\bf 0}_n\}.
\quad
\qed
%=
%\C^{nr}\setminus \bigcup_{\sigma\in I(K)}L_{\sigma}(\C^n),
%%%%%%%%%%% 
\end{equation}
%%%%%%%%%%
%where
%we set}
%%%
%$(\C^{n})^*=\C^n\setminus \{{\bf 0}_n\}.
%$
%%%%%
%%%
%%%
\end{dfn}
%%%(End of Definition 2.1)%%%%%%
%%
%%
%%
%%%(Homogenous coordinates)%%%%%
\paragraph{Homogenous coordinates of toric varieties}
%%%%%%%%%%%%%%%%%%%%%%%%%%%%%%%%
%%%
%%%
%%%
%%%
Next we recall the basic facts about homogenous coordinates on
toric varieties.
\par\vspace{1mm}\par
%%%
%%
%%%(Definition 2.2)%%%%%%%%
\begin{dfn}\label{dfn: fan}
%%%%%%%%%%%%%%%%%%%%%%%%%%%
{\rm
Let $\Sigma$ be a fan in $\R^m$
such that $\{{\bf 0}_m\}\subsetneqq \Sigma$, and let
%%%%%%
%(2.7)%%
\begin{equation}\label{eq: one dim cone}
%%%%%%%
\Sigma (1)=\{\rho_1,\cdots ,\rho_r\}
\end{equation}
%%%%%
denote the set of all
one dimensional cones in $\Sigma$. 
\par
%%(i)%%
(i)
For each $1\leq k\leq r$,
we denote by $\textbf{\textit{n}}_k\in\Z^m$ 
\textit{the primitive generator} of
$\rho_k$, such that %that is, we have
%%()%%
%\begin{equation}\label{eq: primitive}
%%%
$\rho_k \cap \Z^m=\Z_{\geq 0}\cdot \textbf{\textit{n}}_k.$
%\end{equation}
%%%%%%%%
%
Note that
$\rho_k %\R_{\geq 0}\cdot \textit{\textbf{n}}_k
=\mbox{Cone}(\textit{\textbf{n}}_k).$
\par
%%%%%%%%%%%
(ii)
%%%%%%%%%%%
Let $\mathcal{K}_{\Sigma}$ denote {\it the underlying simplicial complex of}  $\Sigma$ 
%on the index set $[r]$
defined by
%%%%%
%%%(2.8)%%%%%%
\begin{equation}\label{eq: KS}
%%%%%%%%%%%%%
\KS =
\Big\{\{i_1,\cdots ,i_s\}\subset [r]:
\textbf{\textit{n}}_{i_1},\textbf{\textit{n}}_{i_2},\cdots
,\textbf{\textit{n}}_{i_s}
\mbox{ span a cone in }\Sigma\Big\}.%\cup \{\emptyset\}.
\end{equation}
%%%%%%
It is easy to see that $\KS$ is a simplicial complex on the index set $[r]$.
\par
%%(iii)%%
(iii)
Define  the subgroup
$\GS\subset \T^r_{\C}=(\C^*)^r$  by
%%%%%%%%
%%(2.9)%%
\begin{equation}
%%%%%%%%
\GS =\{(\mu_1,\cdots ,\mu_r)\in \T^r_{\C}:
\prod_{k=1}^r(\mu_k)^{\langle \textbf{\textit{n}}_k,\textbf{\textit{m}}
\rangle}=1
\mbox{ for all }\textbf{\textit{m}}\in\Z^m\},
%%%%%%%%%%%%
\end{equation}
%%%
where
$\langle \textbf{\textit{u}}, \textbf{\textit{v}}\rangle=
\sum_{k=1}^m u_kv_k$ for
$\textbf{\textit{u}}=(u_1,\cdots ,u_m)$
and  $\textbf{\textit{v}}=(v_1,\cdots ,v_m)\in\R^m$.
%%%%
\par
%%%(iv)%%%%%
(iv)
Now 
consider the natural $\GS$-action on
$\mathcal{Z}_{\mathcal{K}_{\Sigma}}(\C^n,(\C^{n})^*)$ given by 
coordinate-wise multiplication, i.e.
%%%%%%%%%
%%%(2.10)%%%%%
\begin{equation}
%%%%%%%%%%%%
(\mu_1,\cdots ,\mu_r)
\cdot(\textbf{\textit{x}}_1,\cdots ,\textbf{\textit{x}}_r)=
(\mu_1\textbf{\textit{x}}_1,\cdots ,\mu_r\textbf{\textit{x}}_r)
\end{equation}
%%%%%
for
$((\mu_1,\cdots ,\mu_r),
(\textbf{\textit{x}}_1,\cdots ,\textbf{\textit{x}}_r))
\in \GS \times \mathcal{Z}_{\mathcal{K}_{\Sigma}}(\C^n,(\C^{n})^*),$
where we set
%%%(2.11)%%
\begin{equation}
%%%
\mu \textbf{\textit{x}}=
(\mu x_1,\cdots ,\mu x_r).
\quad
\mbox{if }(\mu ,
\textbf{\textit{x}})=(\mu, (x_1,\cdots ,x_r))\in \C\times \C^n.
%%%
\end{equation}
%%%%%%%%%%%%%
Then define the space $\XS (n)$ by
the corresponding orbit space
%%%(2.12)%%
\begin{equation}\label{eq: XSigma(n)}
%%%%%%
\XS (n)=
\mathcal{Z}_{\mathcal{K}_{\Sigma}}(\C^n,(\C^{n} )^*)/\GS.
\qquad\qquad\qquad
\qed
\end{equation}
%%%%%%%
%%%%
}
%%%%%
\end{dfn}
%%%(end of Definition 2.2)%%%%%
%%%
%%%
%%%
%%%
%%%
%%%(Remark 2.3)%%%
\begin{rmk}\label{rmk: fan}
%%%%%%%%%%%%%%%%%%%%%%%%
%%
%%
%%
%%%%%%%%%%%%%%%%%%
{\rm
(i)
Let $\Sigma$ be a fan in $\R^m$ as in (\ref{eq: one dim cone}).
Then the fan $\Sigma$ is completely determined by
the pair
$(\mathcal{K}_{\Sigma},\{\textit{\textbf{n}}_k\}_{k=1}^r)$ 
(see \cite[Remark 2.3]{KY9} in detail).
%In fact, if we set 
%%%%%%%%%%
%%%(2.13)%%
%\begin{equation}
%\mbox{Cone}_{\sigma}=
%\begin{cases}
%\mbox{Cone}
%(\textit{\textbf{n}}_{i_1},\cdots ,\textit{\textbf{n}}_{i_s})
%&
%\mbox{if
%$\sigma =\{i_1,\cdots ,i_s\}$}
%\\
%\{{\bf 0}_m\}
%&
%\mbox{if }\sigma =\emptyset
%\end{cases}
%\end{equation}
%%%%%
%for a subset $\sigma\subset [r]$,
%then
%it is easy to see that
%%%()%%
%%\begin{equation}
%$\Sigma =
%\{\mbox{Cone}_{\sigma}:\sigma \in \mathcal{K}_{\Sigma}\}.$
%%\end{equation}
\par
(ii) 
Note that the group $\GS$ acts on 
$\mathcal{Z}_{\KS}(\C^n,(\C^n)^*)$ freely
(see Corollary \ref{crl: principal}).
Moreover, one can show that
$\XS (n)$ is a toric variety
(see Remark \ref{rmk: toric-variety-remark}).
\qed
%%%%
}
%%%%%%%
\end{rmk}
%%%%%%%(End of Remark 2.3)%%%%%%%%
%%%
%%%
%%%
%%%
%%%
%%%
%%%%%%%%%%%%%%%%%%%%%%%%
 The following theorem 
 plays a crucial role in the proof of the main result of this paper.
 %states that in a toric variety, under certain mild conditions, 
%%% 
%%%%%%%%%%%%%%%%%%%%%%%
%%
%%
%%%(Theorem 2.4: Theorem of Cox)%%
\begin{thm}
[\cite{Cox1}, Theorem 2.1; \cite{Cox2}, Theorem 3.1]\label{prp: Cox}
%%%%%%%%%
Let $\Sigma$ be a fan in $\R^m$ as in Definition
\ref{dfn: fan} and 
suppose that the set $\{\textit{\textbf{n}}_k\}_{k=1}^r$ 
of all primitive generators  spans $\R^m$
$($i.e. $\sum_{k=1}^r\R\cdot \textit{\textbf{n}}_k =
\{\sum_{k=1}^r\lambda_k \textit{\textbf{n}}_k:\lambda_k\in \R\}=
\R^m).$
\par
%%(i)%%%
$\I$
Then
there is a natural isomorphism
%%(2.13)%%
\begin{equation}\label{equ: homogenous1}
\XS\cong
\mathcal{Z}_{\mathcal{K}_{\Sigma}}(\C,\C^*)/\GS =\XS (1). 
\end{equation}
%%%%
%%%
\par
%%(ii)%%%%%%
$\II$
%%%
If
$f:\CP^s\to \XS$ is a holomorphic map, then
there exists an $r$-tuple $D=(d_1,\cdots ,d_r)\in (\Z_{\geq 0})^r$
of non-negative integers
satisfying the condition
$\sum_{k=1}^rd_k\textit{\textbf{n}}_k={\bf 0}_m$,
and homogenous polynomials $f_i\in \C [z_0,\cdots ,z_s]$ of degree $d_i$
$(i=1,2,\cdots, r)$
such that the
polynomials $\{f_{i}\}_{i\in\sigma}$ have no common root except
${\bf 0}_{s+1}\in\C^{s+1}$ for each $\sigma \in I(\mathcal{K}_{\Sigma})$
and that
the diagram
%%(2.14)%%
\begin{equation}\label{eq: homogenous-CD}
%%%%
\begin{CD}
\C^{s+1}\setminus \{{\bf 0}_{s+1}\} @>(f_1,\cdots ,f_r)>> \mathcal{Z}_{\KS}(\C,\C^*)
\\
@V{\gamma_s}VV @V{q_{\Sigma}}VV
\\
\CP^s @>f>> \mathcal{Z}_{\KS}(\C,\C^*)/\GS =\XS 
\end{CD}
\end{equation}
%%%%%%%%
is commutative,
where we identify $\XS =\XS (1)$ as in (\ref{equ: homogenous1}) and the
two map
$\gamma_s:\C^{s+1}\setminus \{{\bf 0}_{s+1}\}\to \CP^s$
and $q_{\Sigma}:\mathcal{Z}_{\KS}(\C,\C^*)\to \XS =\XS (1)$
denote
the canonical Hopf fibering and the canonical projection
induced from the identification $($\ref{equ: homogenous1}$)$, respectively.
In this case, we call this holomorphic map $f$  a holomorphic map of degree $D=(d_1,\cdots ,d_r)$ and we
 represent it  as
%%%%%%%%%
%%(2.15)%%%
\begin{equation}\label{equ: homogenous}
%%%%%%%%
f=[f_1,\cdots ,f_r].
\end{equation}
%%%%%%%%%%%%
\par
%%%(iii)%%%%
$\III$
If $g_i\in \C [z_0,\cdots ,z_s]$ is a homogenous polynomial
of degree $d_i$
$(1\leq i\leq r)$ such that $f=[f_1,\cdots ,f_r]=[g_1,\cdots ,g_r]$,
there  exists some element $(\mu_1,\cdots ,\mu_r)\in \GS$ such that
$f_i=\mu_i\cdot g_i$ for each $1\leq i\leq r$.
Thus, the $r$-tuple
$(f_1,\cdots ,f_r)$ of homogenous polynomials representing a holomorphic map $f$
is  determined uniquely up to
$\GS$-action.
\qed
%%%
\end{thm}
%%%%%%(End of Theorem 2.4)%%%%%
%%%
%%%
%%%
%%%
%%%
%%%%(Assumptions)%%%%%%%
\paragraph{Assumptions}
%%%%%%%%%%%%%%%%%%%%%%%
%%
%%
%%
%%%%%%%%%%%%%%%%%%%%%%%
From now on, let $\Sigma$ be a fan in $\R^m$
satisfying the condition (\ref{eq: one dim cone})
as in Definition \ref{dfn: fan}, 
and assume that $\XS$ is simply connected and non-singular.
Moreover,
we shall assume  the following condition holds.
%%(2.15.1)%%
\begin{enumerate}
%%%%%%%%%%
\item[$($\ref{equ: homogenous}.1$)$]
There is an $r$-tuple
$D_*=(d_1^*,\cdots ,d_r^*)\in \N^r$ such that
$\sum_{k=1}^rd_k^*\textit{\textbf{n}}_k={\bf 0}_m.$
%%%
\end{enumerate}
%%%%%%%%%%%%%%%%%%%%%%%%
%%
%%
%%
%%
%%
%%
%%%%%%%%%%%%%%%%%%%%%%%%%
%%%(Remark 2.5)%%%
\begin{rmk}\label{rmk: assumption}
%%%%%%%%%%%%%%%%%%%%%%%%%
%%
%%
%%%%%%%%%%%%%%%%%%%%%%
{\rm
%(i) 
It follows from \cite[Theorem 12.1.10]{CLS} that
$\XS$ is simply connected if and only if the fan
$\Sigma$ satisfies the following condition $(*)$:
%%%%%%%%%%
\begin{enumerate}
%%%%%%%%%%%
\item[($*$)]
%%%%%%%%%%%
The set $\{\textbf{\textit{n}}_k\}_{k=1}^r$ of all primitive generators spans
$\Z^m$ over $\Z$, i.e.
\newline
$\sum_{k=1}^r\Z\cdot \textbf{\textit{n}}_k=\Z^m$.
\end{enumerate}
%%%%%%%%%%
%%
%%
%%
%%%%%%%%%%%%%%%%%
Thus, one can easily see that
the set $\{\textit{\textbf{n}}_k\}_{k=1}^r$ 
of all primitive generators  spans $\R^m$
if $\XS$ is simply connected.
In particular, we can see that $\XS$ is simply connected
if $\XS$ is a compact smooth toric variety
(see Lemma \ref{lmm: toric}).
\qed
}
%%%%%%%%%%%%%%%%%%%%%%%
\end{rmk}
%%(End of Remark 2.5)%%
%%%%%
%%
%%
%%
%%%%%%%%%%%%%%%%%%%%%%%
\paragraph{Spaces of holomorphic maps}
%%%%%%%%%%%%%%%%%%%%%%%%
%%
%%
%%%%%%%%%%%%%
We let $\XS$ be a smooth toric variety and
 make the identification $\XS =\mathcal{Z}_{\KS}(\C,\C^*)/\GS =\XS (1)$. 
Now
consider a base point preserving holomorphic map $f=[f_1,\cdots ,f_r]:
\CP^s\to \XS$ for the
case $s=1$.
In this case,  we  make the identification $\CP^1=S^2=\C \cup \infty$ and choose the points $\infty$ and $[1,1,\cdots ,1]$ as the base points of
$\CP^1$ and $\XS$ respectively.
Then, by setting $z=\frac{z_0}{z_1}$,  for each $1\leq k \leq r$,
we can view $f_k$ as a monic polynomial  $f_k(z)\in\C [z]$
of degree $d_k$ in the complex variable  $z$.
Now we can define the space of holomorphic maps as follows.
%%%%%%%%%%%%
%%
%%
%%
%%%%%%%%%%%%%%%%%%%%%%%%%%%
%%%%(Definition 2.6)%%%%%%%
\begin{dfn}\label{dfn: holomorphic}
%%%%%%%%%%%%%%%%%%%%%%%%%%%
%%
%%
%%%%%%%%%%%
{\rm
%%(i)%%
(i)
%%%%%%
Let $\P^d$ denote the space of all monic polynomials 
$g(z)=z^d+a_1z^{d-1}+\cdots +a_{d-1}z+a_d\in \C [z]$ of degree
$d$, and
we set
%%(2.16)%%%
\begin{equation}\label{eq: Pd}
%%%%%%%%
\P^D=\P^{d_1}\times \P^{d_2}\times\cdots \times \P^{d_r}
\quad
\mbox{if }D=(d_1,\cdots ,d_r)\in \N^r.
\end{equation}
%%%
Note that there is a natural homeomorphism
$\phi:\P^d\cong \C^d$ given by
$\phi (z^d+\sum_{k=1}^d a_kz^{d-k})
=(a_1,\cdots ,a_d)\in\C^d.$
\par
%%%(ii)%%%
(ii)
%%%%%%%%%
For any $r$-tuple $D=(d_1,\cdots ,d_r)\in \N^r$ 
satisfying the condition (\ref{equ: homogenous}.1),
%%% 
let $\Hol_D^*(S^2,\XS)$ denote the
space consisting of all $r$-tuples
$f=(f_1(z),\cdots ,f_r(z))\in \P^D$ 
%of monic polynomials
satisfying the following condition $(\dagger_{\Sigma})$:
%%%%%
%%
%%
%%
%%%%%%%%%%%%%%%
\begin{enumerate}
%%%%%%%%%%%%%%%
%%
%%
%%(2.16.1)%%
\item[$(\dagger_{\Sigma})$]
%%%%%%%
For any $\sigma =\{i_1,\cdots ,i_s\}\in I(\KS)$,
the polynomials
$f_{i_1}(z),\cdots ,f_{i_s}(z)$ have no common root,
i.e.
$(f_{i_1}(\alpha),\cdots ,f_{i_s}(\alpha))\not= {\bf 0}_s=(0,\cdots ,0)$
for any $\alpha \in  \C$.
%%%%%%%%%%
\end{enumerate}
%%%%%%%%%%
By identifying 
$\XS =\mathcal{Z}_{\KS}(\C,\C^*)/\GS$ and $\CP^1=S^2=\C\cup\infty$,
one can define the natural inclusion map
%%%%%%%%
$i_D:\Hol_D^*(S^2,\XS)\to \Map^*(S^2,\XS)=\Omega^2\XS$ 
by
%%
%%
%%(2.17)%%
\begin{equation}
%%%%%%%%%
i_D(f)(\alpha )=
\begin{cases}
[f_1(\alpha),f_2(\alpha),\cdots ,f_r(\alpha)] & \mbox{ if }\alpha \in\C
\\
[1,1,\cdots ,1] & \mbox{ if }\alpha =\infty
\end{cases}
%%%%%%%%
\end{equation}
%%%%%%%%%%
%%
%%
%%%%%%%%
for $f=(f_1(z),\cdots ,f_r(z))\in \Hol_D^*(S^2,\XS)$,
where
we choose the points $\infty$ and $[1,1,\cdots ,1]$
as the base points of $S^2$ and $\XS$. %respectively.
%%%%
%%%
%%%
\par\vspace{1mm}\par
%%%%
Since the representation of polynomials in $\P^D$ representing a 
base point preserving holomorphic map of degree $D$
is uniquely determined, 
the space $\Hol_D^*(S^2\XS)$ can be identified with 
{\it the space of base point preserving holomorphic maps of 
degree} $D$.
%\par
Moreover,
since $\Hol_D^*(S^2,\XS)$ is path-connected,\footnote{%
%%%%%%%%%%%%%%%%%%%%%%%%%%
%%(FootNote 3)%%%
By \cite[Remark 2.10]{KY9}, we see that $\Hol_D^*(S^2,\XS)$ is path-connected.
}
%%%%%%%%%%%%
%%%(End of FootNote 3)%%%
%%%%
%%%
 the image of
$i_D$
is contained in a certain path-component of $\Omega^2\XS$,
which is denoted by $\Omega^2_D\XS$.
Thus we have a natural inclusion
%%%(2.18)%%%%%
\begin{equation}\label{eq: map-iD} 
%%%%%%%%
i_D:\Hol_D^*(S^2,\XS)\to \Map^*_D(S^2,\XS)=\Omega^2_D\XS .
\quad
\qed
\end{equation}
%%%%%%%%%%
}
\end{dfn}
%%%

%%%(Spaces of resultants of bounded multiplicity)%%
\paragraph{Spaces of non-resultant systems of bounded multiplicity}
%%%%%%%%%%%%%%%%%%%%%%%%%%%%%%%%%%%%%%%%%%%%%%%%%%%
Now 
%recall the definition of the space $\Hol_D^*(S^2,\XS)$ and
consider the space 
$\po^{D,\Sigma}_n(\F)$ for $\F =\C$.
For this purpose, we need the following notation.
%%(Definition 2.7)%%
\begin{dfn}\label{dfn: poly}
{\rm
%%%
For a monic polynomial $f(z)\in \P^d$
of degree $d$, let $F_n(f)(z)$ denote the $n$-tuple of monic polynomials 
of the same degree $d$ given by
%%(2.19)%%%
\begin{equation}
F_n(f)(z)=(f(z),f(z)+f^{\p}(z),f(z)+f^{\p\p}(z),\cdots ,f(z)+f^{(n-1)}(z))
\end{equation}
as in (\ref{eq: Fn}).
%%%%%%
Note that a monic polynomial $f(z)\in \P^d$ has a root $\alpha\in \C$
of multiplicity $\geq n$ iff
$F_n(f)(\alpha)={\bf 0}_n\in \C^n.$
\qed
}
%%%%
\end{dfn}
%%%(End of Dfinition 2.7)%%%%%
Then the space $\po^{D,\Sigma}_n=\po^{D,\Sigma}_n(\C)$ can be redefined as follows.

%%Definition 2.8)%%
\begin{dfn}
%%%%
{\rm
%%%%%%
(i)
For each $D=(d_1,\cdots ,d_r)\in \N^r$, $n\in \N$
and a fan $\Sigma$ in $\R^m$,
let $\po^{D,\Sigma}_n$ denote the space of
$r$-tuples
$
(f_1(z),\cdots ,f_r(z))\in \P^D
$
of $\C$-coefficients
monic polynomials satisfying the following condition $(\dagger_{\Sigma, n})$:
%%%%%
\begin{enumerate}
%%%%%%%%()%%%%
\item[$(\dagger_{\Sigma,n})$]
%%%%%%%%%%%%%%%%%
For any $\sigma =\{i_1,\cdots ,i_s\}\in I(\KS)$,
polynomials $f_{i_1}(z),\cdots ,f_{i_s}(z)$ have no common root of
multiplicity $\geq n$
(but they may have common roots of multiplicity $<n$).
%Thus, for any $\sigma\in I(\KS)$, there exists some
%$(j,\alpha)\in \sigma\times \C$ such that 
%$F_n(f_j)(\alpha)\not={\bf 0}_n.$
%%%%%%%%%%%%%%%%%
\end{enumerate}
%%%%%%%%%%%%%%
Note that the condition $(\dagger_{\Sigma})$ coincides with the condition
$(\dagger_{\Sigma,n})$ if $n=1$.
%We call the space $\po^{D,\Sigma}_n$ as 
%{\it the space of resultant of bounded
%multiplicity $n$ of type} $\Sigma$.
%%
%%
%%
%%
\par\vspace{1mm}\par
%%%(ii)%%%%%
(ii)
%%%%%%%%%%%
When $\sum_{k=1}^rd_k\textbf{\textit{n}}_k={\bf 0}_n,$
define the map
$i_{D}:\po^{D,\Sigma}_n\to \Omega^2\XS (n)$ by
%%(2.20)%%
\begin{equation}\label{eq; def iD}
i_{D}(f)(\alpha)
=
\begin{cases}
[F_n(f_1)(\alpha),F_n(f_2)(\alpha),\cdots ,F_n(f_r)(\alpha)]
& \mbox{if }\alpha \in \C
\\
[\textbf{\textit{e}},\textbf{\textit{e}},
\cdots ,\textbf{\textit{e}}]
& \mbox{if }\alpha =\infty
\end{cases}
\end{equation}
%%%
for $f=(f_1(z),\cdots ,f_r(z))\in \po^{D,\Sigma}_n$ and
$\alpha \in \C \cup \infty =S^2$,
where the space $\XS (n)$ is the space defined as in
(\ref{eq: XSigma(n)}) and 
we set
$\textbf{\textit{e}}=(1,1,\cdots ,1)\in \C^n$.
%%%%
\par
%%%%%
Since $\po^{D,\Sigma}_n$ is connected,\footnote{%
%%%(FotNote 4)%%%%
%This can be proved by using the same way as
%\cite[Remark 2.10]{KY9}.}
See Remark \ref{rmk: ESigma}.}
%%'End of FootNote)%%%
 the image of
$i_D$ is contained some path-component of
$\Omega^2\XS (n)$, which is denoted by
$\Omega^2_D\XS (n).$
%%%
Thus we have the map
%%(2.21)%%
\begin{equation}
%%%
i_D:\po^{D,\Sigma}_n\to \Omega^2_D\XS (n).
\qquad\qquad
\qed
\end{equation}
%%%
}
%%%%%%%%%%%%%
\end{dfn}
%%%%%(End of Definition 2.8)%%
%%
%%
%%
%%
%%
%%%%%%(The previous result)%%
\paragraph{The numbers $\rmin (\Sigma)$ and
$d(D;\Sigma,n)$}
%%%%%%%%%%%%%%%%%%%%%%%%%%%%%
%%
%%
%%%%%%%%%%%%%%
Before stating the main results of this paper, we need to define the positive integers
$\rmin (\Sigma)$ and
$d(D;\Sigma,n)$.
%%
%%
%%%%%%%%%%%%%%%%%%%%%%%%%
%%(Definition 2.9)%%%%%%%
\begin{dfn}\label{dfn: rnumber}
%%%%%%%%%%%%%%%%%%%%%%%%%
{\rm
%%%%%%%%%%
We say that a set $S=\{\textbf{\textit{n}}_{i_1},\cdots ,\textbf{\textit{n}}_{i_s}\}$
%of some primitive generators
is {\it a primitive} 
 if  $\mbox{Cone}(S)\notin\Sigma$ and $\mbox{Cone}(T)\in \Sigma$ for any
 proper subset $T\subsetneqq S$.
%%()%%
%%%%%%%
Then we define  $d(D,\Sigma,n)$ to be  the positive integer given by
%%(2.22)%%%%%%%}
\begin{align}\label{eq: dDSigma}
%%%%%%%%%%%%%%%
d(D;\Sigma ,n)&=(2n\rmin (\Sigma)  -3)\lfloor \frac{d_{\rm min}}{n}\rfloor -2,
\end{align}
%%%%%%
where  $\rmin (\Sigma)$ and $d_{min}=d_{min}(D)$ are the positive integers given by
\begin{align}\label{eq: rmin}
%%(2.24)%%
\rmin (\Sigma)
&=\min\{s\in\N :
\{\textbf{\textit{n}}_{i_1},\cdots ,\textbf{\textit{n}}_{i_s}\}
\mbox{ is primitive}\},
\\
%%(2.24)%%
d_{min}&=d_{min}(D)=
\min\{d_1,d_2,\cdots ,d_r\}.
\qquad\qquad\qquad\qquad
\qed
\end{align}
%%%%%
%%%%%%%
}
%%%%%%
\end{dfn}
%%%%%%%%(End of Definition 2.9)%%%
%%%%
%%%%
%%%%
For connected space $X$, let $\Omega^2_0X$ denote the path-component of
$\Omega^2X$ which contains null-homotopic maps and
recall the following result.
%%%
%%%
%%%
%%%%%%%%%%%%%%%%%%%%%%%%%%%%%%%%%%%%%%%%%%%%%%
%%%(Theorem 2.10: The main Theorem of KY9)%%%%%
\begin{thm}[\cite{KY9}]\label{thm: I:KY9}
%%%%%%%%%%%%%%%%%%%%%%%%%%%%%%%%%%%%%%%%%%%%%%
Let $\XS$ be an $m$ dimensional simply connected non-singular toric variety
 such that 
the condition $($\ref{equ: homogenous}.1$)$ holds.
Then if $D=(d_1,\cdots ,d_r)\in \N^r$ and
$\sum_{k=1}^rd_k\textbf{\textit{n}}_k={\bf 0}_m$, 
the inclusion map
$$
i_D:\Hol_D^*(S^2,\XS) \to \Omega^2_D\XS
\simeq \Omega^2_0\XS
\simeq
\Omega^2\mathcal{Z}_{\KS}
$$ 
is a homotopy equivalence
through dimension $d(D;\Sigma ,1)=
(2\rmin (\Sigma)-3)d_{\min}-2$ if $\rmin (\Sigma)\geq 3$
and  a homology equivalence through dimension 
$d(D;\Sigma ,1)=d_{min}-2$ if $\rmin (\Sigma)=2$,
where
$\mathcal{Z}_K$ denotes the moment-angle complex of a 
simplicial complex $K$.\footnote{%
%%%%(Footnote 6)%%%%% 
See Definition \ref{def: moment}.}
%%%(end of Footnote 6)%
\qed
%%%%
\end{thm}
%%%%%%%%(End od Theorem 2.10)%%
%%%
%%
%%
%%
%%
%%%%%%%%%%%%%%%%%%%%%%%%%%%%%
\paragraph{The main results}
%%%%%%%%%%%%%%%%%%%%%%%%%%%%%
The main result of this paper is a generalization of the above theorem
(Theorem \ref{thm: I:KY9}) to spaces of non-resultant systems of bounded multiplicity.
%From now on, we shall only consider the case $n\geq 2$.
%\par
%%%%
%%%%
%%%%
%%%%
%%%%%%%%%%%%%%%%%%%%%%%%%%%%%%%%%%%%%%%
%%%(Theorem 2.11: The main result :No.1)%%%%%
%%%%%%%%%%%%%%%%%%%%%%%%%%%%%%%%%%%%%%%
\begin{thm}\label{thm: I} 
%%%%%%%%%%%%%%%%%%%%%%
Let $D=(d_1,\cdots ,d_r)\in \N^r$,
$n\geq 2$ and let $\XS$ be an $m$ dimensional simply connected non-singular toric variety 
such that 
the condition $($\ref{equ: homogenous}.1$)$ holds.
\par
%%%(i)%%
$\I$ If
$\sum_{k=1}^rd_k\textbf{\textit{n}}_k={\bf 0}_m$, then 
the natural map
$$
i_D:\po^{D,\Sigma}_n \to \Omega^2_D\XS (n)\simeq \Omega^2_0\XS (n)\simeq
\Omega^2\mathcal{Z}_{\KS}(D^{2n},S^{2n-1})
$$ 
is a homotopy equivalence
through dimension $d(D;\Sigma ,n)$.
%%%
\par
%%(ii)%%
$\II$
%%%%%%
If
$\sum_{k=1}^rd_k\textbf{\textit{n}}_k\not= {\bf 0}_m$,  
there is a map
$$
j_D:\po^{D,\Sigma}_n \to  
%\Omega^2_0\XS (n)\simeq
\Omega^2\mathcal{Z}_{\KS}(D^{2n},S^{2n-1})
$$ 
which
is a homotopy equivalence
through dimension $d(D;\Sigma ,n)$.
%n\rmin (\Sigma)-3)\lfloor \frac{d_{min}}{n}\rfloor -2.
%\end{equation}
%%%%
\end{thm}
%%%%%%%%(End od Theorem 2.11)%%
%%
%%
%%
%%
%%
%%%%%%%%%%%%%%%%%%%%%
%%(Corollary 2.12)%%
\begin{crl}\label{crl: II}
%%%%%%%%%%%%%%%%%%%%%
Let $n\geq 2$, $D=(d_1,\cdots ,d_r)\in \N^r$, and
let $\XS$ be an $m$ dimensional compact smooth toric variety 
over $\C$ such that the condition
(\ref{equ: homogenous}.1) holds.
Let $\Sigma (1)$ denote the set of all one dimensional cones in $\Sigma$, and
let $\Sigma_1$ be any fan in $\R^m$ such that
%%%%()%%
%\begin{equation}\label{eq: Sigma}
%%%%%
$\Sigma (1)\subset \Sigma_1\subsetneqq \Sigma.$
%\end{equation}
%%%%%%%
\par
Then  $X_{\Sigma_1}$ is a non-compact smooth 
toric subvariety of $\XS$ and the following two statements hold:
\par
\begin{enumerate}
\item[$\I$]
 If $\sum_{k=1}^rd_k\textbf{\textit{n}}_k={\bf 0}_m$, the map
$$
i_D:\po^{D,\Sigma_1}_n\to \Omega^2_DX_{\Sigma_1}(n)
\simeq \Omega^2\mathcal{Z}_{\mathcal{K}_{\Sigma_1}}(D^{2n},S^{2n-1})
$$ 
 is a
homotopy equivalence through the dimension $d(D;\Sigma_1,n)$.
%%%
\item[$\II$] 
If $\sum_{k=1}^rd_k\textbf{\textit{n}}_k\not={\bf 0}_m$, there is a map
$$
j_D:\po^{D,\Sigma_1}_n\to 
\Omega^2\mathcal{Z}_{\mathcal{K}_{\Sigma_1}}(D^{2n},S^{2n-1})
$$ 
which is a
homotopy equivalence through dimension $d(D;\Sigma_1,n)$.
\end{enumerate}
%%%%%%%%%
\end{crl}
%%%%%%%%%%%%(End of Corollary 2.12)%%%%
%%
%%
%%
Since the case $\XS=\CP^n$ 
was treated in \cite{KY9} and \cite{KY6}, we will take as an example
the case  where $\XS$  is the Hirzerbruch surface $H(k)$.
%%%%%%%%%%%
%%%
%%%
%%%
%%%%%%%%%%%%%%%%%%%%%%%%%%%%%%%%%%
%%%(Definition of Hirzerbruch surface)%%%%%%%%
%%%%(Definition 2.13)%%
\begin{dfn}\label{dfn: Hirzerbruch}
%%%%%%%%%%%%%%%%%%%%%%%%%%%%%%%%%%%%%%%%%%%%%%
{\rm
For an integer $k\in \Z$, let
$H(k)$ be {\it the Hirzerbruch surface} defined by
%()%%
\begin{equation*}
H(k)=\{([x_0:x_1:x_2],[y_1:y_2])\in\CP^2\times \CP^1:x_1y_1^k=x_2y_2^k\}
\subset \CP^2\times \CP^1.
\end{equation*}
%%%
Since there are isomorphisms
$H(-k)\cong H(k)$ for $k \not=0$ and $H(0)\cong\CP^1\times \CP^1$,
without loss of generality we can assume that $k\geq 1$.
Let $\Sigma_k$ denote the fan
 in $\R^2$ given by
%%%%%%%
$$
\Sigma_k=
\big\{\mbox{Cone}({\bf n}_i,{\bf n}_{i+1})\ (1\leq i\leq 3),
\mbox{Cone}({\bf n}_4,{\bf n}_1),
\mbox{Cone}(\textit{\textbf{n}}_j)\ (1\leq j\leq 4),\ \{{\bf 0}\}\big\},
$$
where we set
$
\textit{\textbf{n}}_1%=\textit{\textbf{e}}_1
=(1,0),\  
\textit{\textbf{n}}_2%=\textit{\textbf{e}}_2
=(0,1),\  
\textit{\textbf{n}}_3%=-\textit{\textbf{e}}_1+k\textit{\textbf{e}}_2
=(-1,k), 
\ 
\textit{\textbf{n}}_4%=-\textit{\textbf{e}}_2
=(0,-1).
$
\par\vspace{1mm}\par
It is well-known that $\Sigma_k$ is the fan of the toric variety $H(k)$
and that the set of all one dimensional cones in $\Sigma_k$
is
$\Sigma_k(1)=\{\mbox{Cone}(\textit{\textbf{n}}_i):1\leq i\leq 4\}$.
Since $\{\textit{\textbf{n}}_1,\textit{\textbf{n}}_3\}$ and
$\{\textit{\textbf{n}}_2,\textit{\textbf{n}}_4\}$ are the 
only primitive
collections, $\rmin (\Sigma_k)=2$.
%%%
%%
%\par
%%
Moreover, for a $4$-tuple $D=(d_1,d_2,d_3,d_4)\in \N^4$,
the equality
$\sum_{k=1}^4d_k\textit{\textbf{n}}_k=\textbf{0}$ holds 
if and only if
$(d_1,d_2,d_3,d_4)=(d_1,d_2,d_1,kd_1+d_2)$, and
$
\dmin =\min \{d_1,d_2,d_3,d_4\}
=\min\{d_1,d_2\}.
$
\qed
%%%%%%%%%%%%%%%%%%%%%%%%
}
\end{dfn}
%%(End of Definition 2.13)%%
%%%%%
%%
%%
%%
%%
Hence, by Corollary \ref{crl: II}, we have the following:
%%%%%%%%%
%%%%(Example 2.14)%%%
\begin{example}\label{example: H(k)}
%%%%%%%%%%%%%%%%%%%%
Let $k\geq 1$ and $n\geq 2$ be  positive integers.
Let 
$\Sigma$ be a fan in $\R^2$ such that
$
\Sigma_k(1)=\{\mbox{\rm Cone}(\textit{\textbf{n}}_i):1\leq i\leq 4\}
\subset \Sigma \subsetneqq \Sigma_k$, where
$\Sigma_k$ is the fan given in
Example \ref{dfn: Hirzerbruch}.
%%
%\par
%%
Then
$\XS$ is a non-compact non-singular subvariety of $H(k)$
and the following two statements hold.
%%%
\begin{enumerate}
\item[$\I$]
If $D=(d_1,d_2,d_1,kd_1+d_2)$,
the  map
$$
i_D:\po^{D,\Sigma}_n\to \Omega^2_D\XS (n)\simeq
\Omega^2\mathcal{Z}_{\KS}(D^{2n},S^{2n-1})
$$ is a homotopy equivalence
through dimension 
$(4n-3)\lfloor\frac{\min \{d_1,d_2\}}{n}\rfloor-2$.
\item[$\II$]
If $D=(d_1,d_2,d_3,d_4)\in \N^4$, 
%is any $4$-tuple of
%positive integers, 
there is a map
$$
j_D:\po^{D,\Sigma}_n\to 
\Omega^2 \mathcal{Z}_{\KS}(D^{2n},S^{2n-1})
$$ 
which is a homotopy equivalence
through dimension 
$(4n-3)\lfloor \frac{d_{min}}{n}\rfloor-2$,
where we set
$d_{min}=\min \{d_1,d_2,d_3,d_4\}$.
\qed
\end{enumerate}
%%%%
\end{example}
%%(End of Example 2.14)%%
%
%%%(End of section 2)%%%%%%%%%%%

%%%%%%%%%%%%%%
%%%(SECTION 3)%%%%
\section{Basic facts about polyhedral products}
\label{section: polyhedral products}
%%%%%%%%%%%%%%%

First, we recall some definitions and known results.

%%%(Definition 3.1)%%%
\begin{dfn}[\cite{BP}, Definition 6.27, Example 6.39]\label{def: moment}
%%%
{\rm
Let $K$ be a simplicial complex on the index set $[r]$.
We denote by $\mathcal{Z}_K$ and $DJ(K)$
{\it the moment-angle complex} of $K$ and 
{\it the Davis-Januszkiewicz space} of $K$, 
respectively,
which are
defined by}
%%(2.1)%%%
\begin{equation}\label{DJ}
%%%%%%%%
\mathcal{Z}_K=\mathcal{Z}_K(D^2,S^1),\quad
DJ(K)=\mathcal{Z}_K(\CP^{\infty},*).
\qquad\qquad
\qquad
\qed
%%%
\end{equation}
%%%%%
\end{dfn}
%%(End of Definition 3.1)%%%
%%
%%
%%%
%%%%%(Lemma 3.2)%%%
\begin{lmm}[\cite{BP}; Corollary 6.30, Theorem 6.33,
Theorem 8.9]\label{Lemma: BP}
%%%%%%%%%%%%%%%%%%%
Let $K$ be a simplicial complex on the index set $[r]$.
\par
$\I$
The space $\mathcal{Z}_K$ is $2$-connected, and
there is a fibration sequence
%%(3.2)%%
\begin{equation}
\mathcal{Z}_K \stackrel{}{\longrightarrow} DJ(K)
\stackrel{\subset}{\longrightarrow} (\CP^{\infty})^r.
\end{equation}
\par
$\II$
There is an $(S^1)^r$-equivariant deformation retraction
%%(3.3)%%
\begin{equation}\label{eq: retr}
ret:\mathcal{Z}_K(\C^n,(\C^n)^*)\stackrel{\simeq}{\longrightarrow}
\mathcal{Z}_K(D^{2n},S^{2n-1}).
\qquad
\qed
\end{equation}
\end{lmm}
%%%%(End of Lemma 3.2)%%
%%
%%
%%
%%
%%(Lemma 3.3)%%%%%%
\begin{lmm}[\cite{Pa1}; (6.2) and Proposition 6.7]
\label{lmm: principal}
%%%%%%%%%%%%%%%
Let $\Sigma$ be a fan in $\R^m$ and let
$\XS$ be a smooth toric variety such that the condition 
(\ref{equ: homogenous}.1) holds.
Then
there is an isomorphism 
%%(3.4)%%
\begin{equation}\label{eq: GS-eq}
\GS\cong \T^{r-m}_{\C}=(\C^*)^{r-m},
\end{equation}
%%%%%
Moreover,
the group $\GS$ acts on the space $\mathcal{Z}_{\mathcal{K}_{\Sigma}}(\C,\C^*)$
freely by the coordinate-wise action
and there is a principal
$\GS$-bundle sequence
%%%
%%%(3.5)%%
\begin{equation}\label{eq: principal}
%%%%%%%%%
\GS \stackrel{}{\longrightarrow}
\mathcal{Z}_{\KS}(\C^n,(\C^n)^*)
\stackrel{q_{\Sigma}}{\longrightarrow}
\XS .
\qquad\quad \qed
%%%%%%%%%%%%%%
\end{equation}
%%%%%%%%%%%%%%}
%%%%%
\end{lmm}
%%%(End of Lemma 3.3)%%
%%
%%
%%
%%
%%
%%(Corollary 3.4)%%
\begin{crl}\label{crl: principal}
%%%%%%%%%%%%
The group  $\GS$ acts on the space 
$\mathcal{Z}_{\KS}(\C^n,(\C^n)^*)$ freely and
there is a principal $\GS$-bundle sequence
%%(3.6)%%
\begin{equation}
\GS \stackrel{}{\longrightarrow}
\mathcal{Z}_{\KS}(\C^n,(\C^n)^*)
\stackrel{q_{\Sigma}}{\longrightarrow}
\XS (n).
\end{equation}
%%%
\end{crl}
%%%%(Proof of Corollary 3.4)%%%
\begin{proof}
%%%%
If $n=1$, the assertion follows from
Lemma \ref{lmm: principal} and assume that $n\geq 2$.
Since the action of $\GS$ on $\mathcal{Z}_{\KS}(\C^n,(\C^n)^*)$
is the diagonal one of the case $n=1$, this action is also free
and we obtain the desired assertion.
\end{proof}
%%(End of proof of Corollary 3.4)%%%
%%
%%
%%
%%
%%
%%%(Lemma 3.5)%%%
\begin{lmm}\label{lmm: XS}
%%%%%%%%%%%%%%%%%
If the condition (\ref{equ: homogenous}.1) is
satisfied, 
the space $\XS$ is simply connected and $\pi_2(\XS)=\Z^{r-m}$. 
%%%
\end{lmm}
%%%%%%%%%%%%%%%%%
\begin{proof}
%%%
This follows from \cite[Lemma 3.4]{KY9}.
%By  (\ref{equ: homogenous}.1) and \cite[Theorem 12.1.10]{CLS}
%we easily see that
%the space $\XS$ is simply connected.
%%%%
%By Lemma \ref{Lemma: BP} and (\ref{eq: retr}),  $\mathcal{Z}_{\KS}(\C,\C^*)$ is $2$-connected.
%Now by using the homotopy exact sequence of the 
%principal $\GS$-bundle (\ref{eq: principal}) and the isomorphism 
%(\ref{eq: GS-eq}), we   
%see that  $\pi_2(\XS)= \Z^{r-m}$.
%%%%%%%
\end{proof}
%%%%%%%%%(End of Proof of Lemma 3.5)%%
%%
%%
%%
%%%%%%%%
%Next, recall the basic facts concerning the relation between a fan and a toric variety.
%%%%%%%%
%%%%(Definition 3.6)%%
%\begin{dfn}
%%%%%%%%
%{\rm
%Let $\Sigma$ be a fan in $\R^m$.
%Then a cone $\sigma\in \Sigma$ is called {\it smooth} if it is generated by
%a subset of a basis of $\Z^m$.
%\qed
%%Let $\vert\Sigma\vert$ denote the support of $\Sigma$ defined by
%%the subset
%%$\vert\Sigma\vert =\bigcup_{\sigma\in\Sigma}\sigma$.
%%%
%}
%%%%%%
%\end{dfn}
%%(End of Definition 3.6)%%%
%%
%%
%%
%%%(Lemma 3.6)%%
\begin{lmm}[\cite{CLS}]\label{lmm: toric}
%%%%%%%%%%%%%%%%
Let $\XS$ be a toric variety determined by a fan $\Sigma$ in $\R^m$.
Then $\XS$ is compact if and only if $\R^m=\bigcup_{\sigma\in \Sigma}\sigma$.
\qed
%\begin{enumerate}
%%%
%\item[$\I$]
%$\XS$ is non-singular
%%smooth 
%if and only if every cone $\sigma\in\Sigma$ is smooth.
%\item[$\II$]
%$\XS$ is compact if and only if $\R^m=\bigcup_{\sigma\in \Sigma}\sigma$.
%\qed
%\end{enumerate}
%%%%%
\end{lmm}
%%%%%%%(End of Lemma 3.7)%%%
%%%
%%%
%%%
%%
%%
%%%%%%%%%%%%%%%%%%%%%%%%%%%%%%%%%%%%
%%(SECTION 4: Vassiliev spectral sequence)%%%
%%%%%%%%%%%%%%%%%%%%%%%%%%%%%%%%%
\section{The
Vassiliev spectral sequence}\label{section: simplicial resolution}
%%%%%%%%%%%%%%%%
%%%
%%%
%%%
First, recall  the definitions of the non-degenerate simplicial resolution
and the associated truncated simplicial resolution (\cite{Mo2},  \cite{Va}).
%%%%%
%%
%%
%%(Definition 4.1)%%
\begin{dfn}\label{def: def}
%%%%%%%%%%%%%%%%%%%
{\rm
%%(i)%%
(i)
%%%%%%% 
For a finite set $\textbf{\textit{v}} =\{v_1,\cdots ,v_l\}\subset \R^N$,
let $\sigma (\textbf{\textit{v}})$ denote the convex hull spanned by 
$\textbf{\textit{v}}.$
%%%
Let $h:X\to Y$ be a surjective map such that
$h^{-1}(y)$ is a finite set for any $y\in Y$, and let
$i:X\to \R^N$ be an embedding.
Let  $\mathcal{X}^{\Delta}$  and $h^{\Delta}:{\mathcal{X}}^{\Delta}\to Y$ 
denote the space and the map
defined by
%%%
%%(4.1)%%%%%%%%
\begin{equation}
%%%%%%%%%%%%%%%
\mathcal{X}^{\Delta}=
\big\{(y,u)\in Y\times \R^N:
u\in \sigma (i(h^{-1}(y)))
\big\}\subset Y\times \R^N,
\ h^{\Delta}(y,u)=y.
%%%%%%%%%%%%%%
\end{equation}
%%%%%%%%%%%%%%%%
The pair $(\mathcal{X}^{\Delta},h^{\Delta})$ is called
{\it the simplicial resolution of }$(h,i)$.
In particular, it %$(\mathcal{X}^{\Delta},h^{\Delta})$
is called {\it a non-degenerate simplicial resolution} if for each $y\in Y$
any $k$ points of $i(h^{-1}(y))$ span $(k-1)$-dimensional simplex of $\R^N$.
%%%%
\par
%%(ii)%%
(ii)
%%%%%%%%
For each $k\geq 0$, let $\mathcal{X}^{\Delta}_k\subset \mathcal{X}^{\Delta}$ be the subspace
of
the union of the $(k-1)$-skeletons of the simplices over all the points $y$ in $Y$
given by
%%(4.2)%%
\begin{equation}
%%%%%%%%
\mathcal{X}_k^{\Delta}=\big\{(y,u)\in \mathcal{X}^{\Delta}:
u \in\sigma (\textbf{\textit{v}}),
\textbf{\textit{v}}=\{v_1,\cdots ,v_l\}\subset i(h^{-1}(y)),l\leq k\big\}.
\end{equation}
%%%%%%%%%%%
We make the identification $X=\mathcal{X}^{\Delta}_1$ by identifying 
 $x\in X$ with the pair
$(h(x),i(x))\in \mathcal{X}^{\Delta}_1$,
and we note that  there is an increasing filtration
%%%(4.3)%%
\begin{equation}\label{equ: filtration}
%%%
\emptyset =
\mathcal{X}^{\Delta}_0\subset X=\mathcal{X}^{\Delta}_1\subset \mathcal{X}^{\Delta}_2\subset
\cdots \subset \mathcal{X}^{\Delta}_k\subset 
%\mathcal{X}^{\Delta}_{k+1}\subset
\cdots \subset \bigcup_{k= 0}^{\infty}\mathcal{X}^{\Delta}_k=\mathcal{X}^{\Delta}.
%%%%%%%%%%%
\end{equation}
%%%%%%%%%%
%%
%%
%%
%%%%%%%%%
Since the map $h^{\Delta}:\mathcal{X}^{\Delta}\stackrel{}{\rightarrow}Y$
 is a proper map, it extends to the map
 ${h}_+^{\Delta}:\mathcal{X}^{\Delta}_+\stackrel{}{\rightarrow}Y_+$
 between the one-point compactifications,
 where $X_+$ denotes the one-point compactification of a locally compact space $X$.
%%%
%\par
%%%%(iii)%%
%(iii)
%%%%%%%%%%
%A space $X\subset \R^n$ is called  {\it semi-algebraic} if it is a subspace
%of the form
%$X=\bigcup_{i=1}^s\bigcap_{j=1}^{r_i}\{(x_1,\cdots ,x_n)\in\R^n:f_{ij}\ *_{ij}\  0\}$,
%where
%$f_{ij}\in \R [X_1,\cdots ,x_n]$ and
%$*_{ij}$ is either $<$ or $=$, for $i=1,\cdots s$ and
%$j=1,\cdots ,r_i$.
%\par
%Similarly, when $X\subset \R^n$ and $Y\subset \R^m$ are semi-algebraic spaces,
%a map $f:X\to Y$ is {\it a semi-algebraic map} if
%the graph of $f$ is semi-algebraic.
\qed
%%%%%%%%%%%%%%%%
}
%%%%%%%%%%%%%%%%
\end{dfn}
%%%%%(End of Definition 4.1)%%%%%%%%
%%
%%
%%
%%%(Lemma 4.2)%%%%%%%%%%%%%%%%%
\begin{lmm}[\cite{Va}, \cite{Va2} (cf. Lemma 3.3 in \cite{KY7})]
\label{lemma: simp}
%%%%%%%%%%%%%%%%%%%%%%%%%%%%%%%%
%%
%%
Let $h:X\to Y$ be a surjective map such that
$h^{-1}(y)$ is a finite set for any $y\in Y,$ and let
$i:X\to \R^N$ be an embedding.
%%%%%%%
%\par
%%(i)%%
%%%%%%
Then
if $X$ and $Y$ are semi-algebraic spaces and the
two maps $h$, $i$ are semi-algebraic maps, then the map
${h}^{\Delta}_+:\mathcal{X}^{\Delta}_+\stackrel{\simeq}{\rightarrow}Y_+$
is a homology equivalence.\footnote{%
%%%(Footnote 7)%%%%
It is known that $h^{\Delta}_+$ is actually a homotopy equivalence
\cite[page 156]{Va2}.
However, in this paper we do not need this stronger assertion.}
%%(End of Footnote 7)%%%
%%%%%%%%%%%%%%%%%%%%%
\end{lmm}
%%%
\begin{proof}
The assertion follows from \cite[Lemma 1 (page 90)]{Va}.
\end{proof}
%%%(End of Proof of Lemma 4.2)%%
%%
%%
%%
%%
%%
%%%%%(Remark 4.3)%%%
\begin{rmk}\label{Remark: homotopy equ}
%%%%%%%%%%%%%%%%%%%%
%%
%%
%%
{\rm
Under the same assumption as Lemma \ref{lemma: simp},
there exists always a non-degenerate simplicial resolution of the map $h$.
In fact,
even for a  surjective map $h:X\to Y$ which is not finite to one,  
it is still possible to construct an associated non-degenerate simplicial resolution.
See \cite[Remark 6.4]{KY9} in detail.
%%%
}
\qed
\end{rmk}
%%%%%(End of Remark 4.3)%%%%%%%%
%%%
%%
%%
%%
%%
%%%(Definition 4.4)%%%
\begin{dfn}\label{def: 2.3}
%%%%%%%%%%%%%%%%%%%%%%
{\rm
Let $h:X\to Y$ be a surjective semi-algebraic map between semi-algebraic spaces, 
$j:X\to \R^N$ be a semi-algebraic embedding, and let
$(\mathcal{X}^{\Delta},h^{\Delta}:\mathcal{X}^{\Delta}\to Y)$
denote the associated non-degenerate  simplicial resolution of $(h,j)$. 
Then
for each positive integer $k\geq 1$, we denote by
%%%%%%%%%%%%
%Following  \cite{Mo3}, we call the map 
$h^{\Delta}_k:X^{\Delta}(k)\to Y$
{\it the truncated $($after the $k$-th term$)$  simplicial resolution} of $Y$ as in \cite{Mo3}.
Note that 
that there is a natural filtration
$$
 X^{\Delta}_0\subset
 X^{\Delta}_1\subset
%X^{\Delta}_2\subset
\cdots 
\subset X^{\Delta}_l\subset X^{\Delta}_{l+1}\subset \cdots
\subset  X^{\Delta}_k\subset X^{\Delta}_{k+1}
=X^{\Delta}_{k+2}
%=X^{\Delta}_{k+3}
=\cdots =X^{\Delta}(k),
$$
where $X^{\Delta}_0=\emptyset$,
$X^{\Delta}_l=\mathcal{X}^{\Delta}_l$ if $l\leq k$ and
$X^{\Delta}_l=X^{\Delta}(k)$ if $l>k$.
\qed
}
%%%%%%%%
\end{dfn}
%%%%(End of Definition 4.4)%%%
%%
%%(Lemma 4.5)%%%%%%%%%
\begin{lmm}[\cite{Mo3}, cf. Remark 2.4 and Lemma 2.5 in \cite{KY4}]\label{Lemma: truncated}
%%%%%%%%%%%%%%%%%%%%%%%
%%
Under the same assumptions as in Definition \ref{def: 2.3}, the map
$h^{\Delta}_k:X^{\Delta}(k)\stackrel{\simeq}{\longrightarrow} Y$ 
is a homotopy equivalence.
\qed
\end{lmm}
%%%%%%%%%%%
%%(End of Lemma 4.5)%%%
%%%%%%%%%%%%%%%%%%%%%
%%%%%%%%%%%%%%%%%%%%%%%%%%%%%%%%%%%%%%%%%%%%%%%%%%%
%%%%(Vassiliev spectral sequences)%%%

Next, 
we construct the Vassiliev spectral sequence.
From now on, we always assume that $\Sigma$ is a fan in $\R^m$ such that
$\XS$ is simply connected and that 
the condition  (\ref{equ: homogenous}.1) is satisfied.
Moreover,  $D=(d_1,\cdots ,d_r)\in \N^r$ will always denote a fixed
$r$-tuple of positive integers.
%%%%%%%%%%%%%%%%%%%
%%%
%%
%%(Definition 4.6)%%%%%%%
\begin{dfn}\label{Def: 3.1}
{\rm
%%(i)%%%
(i)
%%%%%
Let $\Sigma_D$ denote {\it the discriminant} of $\po^{D,\Sigma}_n$ in $\P^D$ 
given by the complement
%%%%
\begin{align*}
%%%%
\Sigma_D&=
\P^D \setminus \po^{D,\Sigma}_n
\\
&=
\{(f_1(z),\cdots ,f_{r}(z))\in \P^D :
(f_1(x),\cdots ,f_{r}(x))\in L_n(\Sigma)
\mbox{ for some }x\in \C\},
%%%%%%%
\end{align*}
%%%%%%%
where
%%%
$\dis L_n(\Sigma) =
\bigcup_{\sigma\in I(\mathcal{K}_{\Sigma})}L_{\sigma}(\C^n)
=
~\bigcup_{\sigma\subset[r],\sigma\notin K_{\Sigma}}
L_{\sigma}(\C^n)
$
as in (\ref{eq: L(Sigma)}).
%%%%%%
\par
%%(ii)%%
(ii)
%%%%%%%
Let  $Z_D\subset \Sigma_D\times \C$
denote %the subspace
{\it the tautological normalization} of 
 $\Sigma_D$ consisting of 
 all pairs 
$(G,x)=((f_1(z),\ldots ,f_{r}(z)),
x)\in \Sigma_D\times\C$
satisfying the condition
$F(x)=(F_n(f_1)(x),\cdots ,F_n(f_{r})(x))\in L_n(\Sigma)$.
%%%%
Projection on the first factor  gives a surjective map
$\pi_D :Z_D\to \Sigma_D$.
\qed
%%%%%%
}
%%%%%%
\end{dfn}
%%%%%%
%%(End of Definition 4.6)%%%%
%%
%%
%%
%%
%%
%%
%%(Remark 4.7)%%
\begin{rmk}
%%%%%%%%%%%%%%%
{\rm
Let $\sigma_k\in [r]$ for $k=1,2$.
It is easy to see that
$L_{\sigma_1}(\C^n)\subset L_{\sigma_2}(\C^n)$ if
$\sigma_1\supset \sigma_2$.
Letting
%%()%%
\begin{equation*}
%%% 
Pr(\Sigma)=\{\sigma =\{i_1,\cdots ,i_s\} \subset [r]:
\{\textit{\textbf{n}}_{i_1},\cdots ,\textit{\textbf{n}}_{i_s}\}
\mbox{ is a primitive collection}\},
\end{equation*}
%%%
we see that
%%%
%%(4.4)%%
\begin{equation}
%%%%%%%%%
L_n(\Sigma)=\bigcup_{\sigma\in Pr(\Sigma)}L_{\sigma}(\C^n)
%%%%%%%%
\end{equation}
%%%%%%%% 
and by using (\ref{eq: rmin}) we obtain the equality 
%%%(4.5)%%
\begin{equation}\label{eq: dim rmin}
%%%
\dim L_n(\Sigma)=2n(r-\rmin (\Sigma)).
\qquad\qquad\qquad
\qed
%%%%%%%%%%%%%%
\end{equation}
}
%%%
\end{rmk}
%%%(End of Remark 4.7)%%%%%
%%
%%
%%
Our goal in this section is to construct, by means of the
{\it non-degenerate} simplicial resolution  of the discriminant, 
a spectral sequence converging to the homology of
$\po^{D,\Sigma}_n$.
%%%%%
%%%%%
%%%%%
%%(Definition 4.8)%%%%
\begin{dfn}\label{non-degenerate simp.}
%%%%%%%%%%%%%%%%%%%%%%
{\rm
%%%'i)%%%
(i)
%%%%%%%%%%%%%%%%%%
For an $r$-tuple $E=(e_1,\cdots ,e_r)\in (\Z_{\geq 0})^r$ of non-negative integers,
let $N(E)$ denote the non-negative integer given by
%%%%
%%%%
%%(4.6)%%
\begin{equation}\label{eq: ND}
%%%%%%%%
N(E)=\sum_{k=1}^r e_k.
\end{equation}
%%%%%
\par
%%%
%%%
%%%(ii)%%
(ii)
%%%%%%%%
For each based space $X$, let $F(X,d)$ denote the ordered configuration space
of distinct $d$ points in $X$
defined by
%%(4.7)%%
\begin{equation}
%%%%%%%%%
F(X,d)=\{(x_1,\cdots ,x_d)\in X^d:x_i\not= x_j\mbox{ if }i\not= j\}.
%%%%%%%%%%
\end{equation}
%%%%%%%%%
%%
%%
%%%%%%%%
Since the symmetric group $S_d$ of $d$-letters acts on $F(X,d)$ freely by 
permuting coordinates
and let $C_d(X)$ denote the unordered configuration space of $d$-distinct
points in $X$ given by the orbit space
%%%(4.8)%%
\begin{equation}
%%%%%%%%%%
C_d(X)=F(X,d)/S_d.
\end{equation}
%%%
\par
%%%(iii)%%
(iii)
%%%%%%%%%
Let
 $L_{k;\Sigma}\subset (\C\times L_n(\Sigma))^k$ denote the subspace
defined by
%%%%%()
%\begin{equation*}\label{l}
$$
L_{k;\Sigma}=\{((x_1,s_1),\cdots ,(x_k,s_k)): 
x_j\in \C,s_j\in L_n(\Sigma),
x_i\not= x_j\mbox{ if }i\not= j\}.
$$
%\end{equation*}
%%%%%%%%%
The symmetric group $S_k$ on $k$ letters  acts on $L_{k;\Sigma}$ by permuting coordinates. Let
$C_{k;\Sigma}$ denote the orbit space
%%%%%
%%%%(4.9)%%%%
\begin{equation}\label{Ck}
%%%%%%%%%%%%%
C_{k;\Sigma}=L_{k;\Sigma}/S_k.
%%%%%%%%%%%%%
\end{equation}
%%%%%%%%%%%%%%%
%%
%%
%%
Note that $C_{k;\Sigma}$ 
is a cell-complex of  dimension (by (\ref{eq: dim rmin}))
%%%(4.10)%%
\begin{equation}\label{eq: dim CSigma}
%%%%%%%%%%
\dim C_{k;\Sigma}
=
2k(1+nr-nr_{\rm min}(\Sigma)).
%%%%%%%%%%
\end{equation}
%%%%
%%
%%%
\par
%%%(iv)%%%
(iv)
%%%%%%%%%
Let 
$(\mathcal{X}^D,{\pi}^{\Delta}_D:\mathcal{X}^D\to\Sigma_D)$ 
%and
%$(\tilde{\SZ}(d),\ ^{\p}\tilde{\pi}_d^{\Delta}:\SZ(d)\to\Sigma_d^*)$ 
be the non-degenerate simplicial resolution associated to the surjective map
$\pi_D:Z_D\to \Sigma_D$ 
with the natural increasing filtration as in Definition \ref{def: def},
$$
\emptyset =
\SZ_0
\subset \SZ_1\subset 
\SZ_2\subset \cdots
\subset 
\SZ=\bigcup_{k= 0}^{\infty}\SZ_k.
\quad
\qed
$$
}
\end{dfn}
%%%(End of Definition 4.8)%%%%%%
%%%%%%%%
%%
%%
%%
%%
%%%%%%%%%%%%%%%%% 
By Lemma \ref{lemma: simp},
the map
$\pi_D^{\Delta}$
%\SZ\stackrel{\simeq}{\rightarrow}\Sigma_D$
%is a homotopy equivalence which
extends to  a homology equivalence
%%%%%%%%%
$\pi_{D+}^{\Delta}:\SZ_+\stackrel{\simeq}{\rightarrow}{\Sigma_{D+}}.$
%%%%%
%where $X_+$ denotes the one-point compactification of a
%locally compact space $X$.
%%
%\par
%%%
Since
${\mathcal{X}_k^{D}}_+/{\SZ_{k-1}}_+
\cong (\SZ_k\setminus \SZ_{k-1})_+$,
we have a spectral sequence 
%%%%%%%%%%%
%%(4.11)%%%
\begin{equation}
%%%%%%%
\big\{E_{t;D}^{k,s},
d_t:E_{t;D}^{k,s}\to E_{t;D}^{k+t,s+1-t}
\big\}
\Rightarrow
H^{k+s}_c(\Sigma_D;\Z),
%%%
\end{equation}
%%%%%%%
where
$E_{1;D}^{k,s}=H^{k+s}_c(\SZ_k\setminus\SZ_{k-1};\Z)$ and
$H_c^k(X;\Z)$ denotes the cohomology group with compact supports given by 
$
H_c^k(X;\Z)= \tilde{H}^k(X_+;\Z).
$
%%%%%%%%%
\par
%%%
Since there is a homeomorphism
$\P^D\cong \C^{N(D)}$,
by Alexander duality  there is a natural
isomorphism
%%
%%%(4.12)%%%
\begin{equation}\label{Al}
%%%%%%%%%
\tilde{H}_k( \po^{D,\Sigma}_n;\Z)\cong
H_c^{2N(D)-k-1}(\Sigma_D;\Z)
\quad
\mbox{for any }k.
%%%%%%%%%%
\end{equation}
%%%%%%%%%%
%%%
%%%
%%%
By reindexing we obtain a
spectral sequence
%%
%%%(4.13)%%%
\begin{eqnarray}\label{SS}
%%%%%%%%%%%%%%%%%%%
&&
\big\{E^{t;D}_{k,s}, \tilde{d}^{t}:E^{t;D}_{k,s}\to E^{t;D}_{k+t,s+t-1}
\big\}
\Rightarrow H_{s-k}
( \po^{D,\Sigma}_n;\Z),
%%%%%%%%%%%%%%%%%%
\end{eqnarray}
%%%%%%%%%%%%%%%%
%%
%%
where
$E^{1;D}_{k,s}=
H^{2N(D)+k-s-1}_c(\SZ_k\setminus\SZ_{k-1};\Z).$
%%%%%%
%%%%
%%%%
%%%%
%%%%
%%%%
%%%%(Lemma 4.9)%%%%
\begin{lmm}\label{lemma: vector bundle*}
%%%%%%%%%%%%%%%%%%%%%
Let 
$d_{\rm min}=\min\{d_1,\cdots ,d_r\}$
and suppose that $1\leq k\leq \lfloor \frac{d_{min}}{n}\rfloor$.
Then the space
$\SZ_k\setminus\SZ_{k-1}$
is homeomorphic to the total space of a real affine
bundle $\xi_{D,k,n}$ over $C_{k;\Sigma}$ with rank 
$l_{D,k,n}=2N(D)-2nrk+k-1$.
%%%%%%%%%%%%%%%%%%
\end{lmm}
%%%%%%%%%%%%%%%%%%%%%%%%%%%%%%%
%%%%%%%%(Proof of Lemma 4.9)%%%
\begin{proof}
%%%%%%%%%%%%%%%%%%%%%%%%%%%%%%%
%%
%%
%%
Suppose that $1\leq k\leq \lfloor \frac{d_{min}}{n}\rfloor$.
The argument is exactly analogous to the one in the proof of  
\cite[Lemma 4.4]{AKY1}.
%%%
Namely, an element of $\SZ_k\setminus\SZ_{k-1}$ is represented by 
$(f,u)=((f_1(z),\cdots ,f_{r}(z)),u)$, where 
$f=(f_1(z),\cdots ,f_{r}(z))$ is an 
$r$-tuple of monic polynomials in $\Sigma_D$
satisfying the condition
%%%(4.14)%%%
\begin{equation}
%%%%%%%%
{\bf F} (x_j)=
(F_n(f_1)(x_j),\cdots ,F_n(f_{r})(x_j))\in L_n(\Sigma)
\quad
\mbox{for each }1\leq j\leq k
%%%%%%%%%% 
\end{equation}
%%%%%%
and $u$ is an element of the interior of
the span of the images of some $k$ distinct points 
$\{x_1,\cdots, x_k\}\in C_k(\C)$ 
under a suitable embedding. %satisfying the condition (\ref{Remark: non-degenerate}.i),
%such that
%%%%(4.14)%%%
%\begin{equation}
%%%%%%%%%
%{\bf F} (x_j)=
%(F_n(f_1)(x_j),\cdots ,F_n(f_{r})(x_j))\in L_n(\Sigma)
%\quad
%\mbox{for each }1\leq j\leq k.
%%%%%%%%%%%
%\end{equation}
%%%
%for each $1\leq j\leq k$, 
%%%%
%of $\C$ into Euclidean space.% satisfying the condition ({\ref{equ: filtration}}$)_k$.
%%%%
By using \cite[Lemma 2.5]{KY10}
one can show that 
the $k$ distinct points $\{x_j\}_{j=1}^k$ 
are uniquely determined by $u$.
%
%Since the $k$ distinct points $\{x_j\}_{j=1}^k$ 
%are uniquely determined by $u$
%by \cite[Lemma 2.5]{KY10}, 
%
%by the definition of the non-degenerate simplicial resolution, %(cf.  ({\ref{equ: filtration}}$)_k$),
Thus there are projection maps
%%%%%%%%%%%% 
$
\pi_{k,D} :\mathcal{X}^{D}_k\setminus
\mathcal{X}^{D}_{k-1}\to C_{k;\Sigma}
$
%%%%%%%%%%%%
defined by
$((f_1(z),\cdots ,f_{r}(z)),u) \mapsto 
\{(x_1,{\bf F}  (x_1)),\dots, (x_k,{\bf F} (x_k))\}$. 
%%%
\par
%%%
Let $c=\{(x_j,s_j)\}_{j=1}^k\in C_{k;\Sigma}$
$(x_j\in \C$, $s_j\in L_n(\Sigma))$ be any fixed element and consider the fibre  $\pi_{k,D}^{-1}(c)$.
%%%
For this purpose,
we write $s_j=(s_{1,j},\cdots ,s_{r,j})$
for each $1\leq j\leq k$
 with $s_{i,j}\in \C^n$ and
consider the $r$-tuple $f=(f_1(z),\cdots ,f_r(z))
\in \P^D$
of monic polynomials satisfying
the condition 
%%
%%
%% 
%%%(4.15)%%%
\begin{align}\label{equ: pik}
%%%%%%%%%%%
{\bf F} (x_j)&=(F_n(f_1)(x_j),\cdots ,F_n(f_{r})(x_j))=s_j
\quad
\mbox{for  each } 1\leq j\leq k
%%%
\\
%%%
\nonumber 
&\Leftrightarrow
F_n(f_t)(x_j)=s_{t,j}
\quad
\mbox{for each }1\leq j\leq k\mbox{ and } 1\leq t\leq r.
%%%%%%
\end{align}
%%%%%%%%%%%
%%
%%
%%
If we set
$s_{t,j}=(s_{t,j}^{(0)},s_{t,j}^{(0)}+s_{t,j}^{(1)},s_{t,j}^{(0)}+s_{t,j}^{(2)},\cdots ,s_{t,j}^{(0)}+s_{t,j}^{(n-1)})\in \C^n$
with
$s_{t,j}^{(l)}\in \C$,
%%
%%
%%
%%(4.16)%%%
\begin{equation}\label{eq: condition Fn}
%%%%%%%%%%
F_n(f_t)(x_j)=s_{t,j}
\ \Leftrightarrow \
f_t^{(l)}(x_j)=s_{t,j}^{(l)}
\quad
\mbox{for each }0\leq l\leq n-1.
%%%%%%%%%%%%
\end{equation}
%%%%%%
%%
%%
%%
%%
In general, 
the condition $f_t^{(l)}(x_j)=s_{t,j}^{(l)}$ gives
one  linear condition on the coefficients of $f_t$,
and determines an affine hyperplanes in $\P^{d_t}(\C)$. 
%%%
Indeed,
if we set $f_t(z)=z^{d_t}+\sum_{i=0}^{d_t-1}a_{i}z^{i}$,
then
$f_t(x_j)=s_{t,j}^{(0)}$ for any $1\leq j\leq k$
if and only if
$A_1\textbf{\textit{x}}=\textbf{\textit{b}}_1$, where
%%%%%%%%%%%%%%%%%%%%%%
{\small
%%()%%%(matrix equation)%%
\begin{equation*}\label{equ: matrix equation}
%%%%%%%%%%%%%%%%%%%%%%
A_1=
\begin{bmatrix}
1 & x_1 & x_1^2 & \cdots & x_1^{d_t-1}
\\
1 & x_2 & x_2^2 & \cdots & x_2^{d_t-1}
\\
\vdots & \ddots & \ddots & \ddots & \vdots
%\\
%1 & x_{k-1} & x_{k-1}^2 & \cdots & x_{k-1}^{d-1}
\\
1 & x_k & x_k^2 & \cdots & x_k^{d_t-1}
\end{bmatrix}
%%%%
,\ \textbf{\textit{x}}=
\begin{bmatrix}
a_{0}\\ a_{1} \\ \vdots %\\ a_{2,t} 
\\ a_{d_t-1}
\end{bmatrix}
,\ 
\textbf{\textit{b}}_1=
\begin{bmatrix}
s_{t,1}^{(0)}-x_1^{d_t}\\ s_{t,2}^{(0)}-x_2^{d_t} \\ \vdots %\\ s_{t,k-1}-x_{k-1}^d 
\\ s_{t,k}^{(0)}-x_k^{d_t}
\end{bmatrix}
\end{equation*}
}
%%%%
%%%()%%%(matrix equation)%%
%\begin{equation*}\label{equ: matrix equation}
%%%%%%%%%%%%%%%%%%%%%%%
%\begin{bmatrix}
%1 & x_1 & x_1^2 & \cdots & x_1^{d_t-1}
%\\
%1 & x_2 & x_2^2 & \cdots & x_2^{d_t-1}
%\\
%\vdots & \ddots & \ddots & \ddots & \vdots
%%\\
%%1 & x_{k-1} & x_{k-1}^2 & \cdots & x_{k-1}^{d-1}
%\\
%1 & x_k & x_k^2 & \cdots & x_k^{d_t-1}
%\end{bmatrix}
%%%%%
%\cdot
%\begin{bmatrix}
%a_{0}\\ a_{1} \\ \vdots %\\ a_{2,t} 
%\\ a_{d_t-1}
%\end{bmatrix}
%=
%\begin{bmatrix}
%s_{t,1}^{(0)}-x_1^{d_t}\\ s_{t,2}^{(0)}-x_2^{d_t} \\ \vdots %\\ s_{t,k-1}-x_{k-1}^d 
%\\ s_{t,k}^{(0)}-x_k^{d_t}
%\end{bmatrix}
%\end{equation*}
%%%%
%%%%
\newline
Similarly,
$f_t^{\p}(x_j)=s_{t,j}^{(1)}$ for any $1\leq j\leq k$
if and only if
$A_2\textbf{\textit{x}}=\textbf{\textit{b}}_2$, where
{\small
%%()%%%(matrix equation)%%
\begin{equation*}\label{equ: matrix equation}
%%%%%%%%%
A_2=
\begin{bmatrix}
0 & 1 & 2x_1 & 3x_1^2 & \cdots & (d_t-1)x_1^{d_t-2}
\\
0 &1 & 2x_2 & 3x_2^2 & \cdots & (d_t-1)x_2^{d_t-2}
\\
 \vdots & \vdots & \ddots & \ddots & \ddots & \vdots
%\\
%1 & x_{k-1} & x_{k-1}^2 & \cdots & x_{k-1}^{d-1}
\\
0 &1 & 2x_k & 3x_k^2 & \cdots & (d_t-1)x_k^{d_t-2}
\end{bmatrix}
%%%%
, \ 
%\begin{bmatrix}
%a_{0}
%\\ 
%a_{1} \\ \vdots %\\ a_{2,t} 
%\\ 
%a_{d_t-1}
%\end{bmatrix}
%=
\textbf{\textit{b}}_2=
\begin{bmatrix}
s_{t,1}^{(1)}-d_tx_1^{d_t-1}\\ s_{t,2}^{(1)}-d_tx_2^{d_t-1} \\ \vdots %\\ s_{t,k-1}-x_{k-1}^d 
\\ s_{t,k}^{(1)}-d_tx_k^{d_t-1}
\end{bmatrix}
\end{equation*}
}
\noindent
%%%%%%%%%%%%
and
$f_t^{\p\p}(x_j)=s_{t,j}^{(2)}$ for any $1\leq j\leq k$
if and only if
$A_3\textbf{\textit{x}}=\textbf{\textit{b}}_3$, where
%%%
%%%
{\small
%%()%%%(matrix equation)%%
\begin{equation*}\label{equ: matrix equation}
%%%%%%%%%%%%%%%%%
A_3=
\begin{bmatrix}
0 &0 & 2 & 6x_1 &  \cdots & (d_t-1)(d_t-2)x_1^{d_t-3}
\\
0 & 0 &2 & 6x_2 &  \cdots & (d_t-1)(d_t-2)x_2^{d_t-3}
\\
\vdots & \vdots &  \ddots & \ddots & \ddots & \vdots
%\\
%1 & x_{k-1} & x_{k-1}^2 & \cdots & x_{k-1}^{d-1}
\\ 
0 & 0 &2 & 6x_k  & \cdots & (d_t-1)(d_t-2)x_k^{d_t-3} 
\end{bmatrix}
%%%%
, \ 
%\begin{bmatrix}
%a_{0}\\ a_{1} \\ \vdots %\\ a_{2,t} 
%\\ a_{d_t-1}
%\end{bmatrix}
\textbf{\textit{b}}_3
=
\begin{bmatrix}
s_{t,1}^{(2)}-d_t(d_t-1)x_1^{d_t-1}\\ s_{t,2}^{(2)}-d_t(d_t-1)x_2^{d_t-1} \\ \vdots %\\ s_{t,k-1}-x_{k-1}^d 
\\ s_{t,k}^{(2)}-d_t(d_t-1)x_k^{d_t-1}
\end{bmatrix}
\end{equation*}
%%%%%%%%%%%%%%
}
\par
\noindent{and so on.}
%%%%%%%
%%%%%%
%%%%%
%%%%%
%%%%%
Since $\{x_i\}_{i=1}^k\in C_k(\C)$, 
by Gaussian elimination of rows of
matrices, 
the matrix $A_1$ reduces to the matrix $B_1$, where
$\dis s_i(l)=\sum_{i_1+\cdots +i_l=i}x_1^{i_1}x_2^{i_2}\cdots x_l^{i_l}$
and  
$B_1$ is the matrix given by
%%%(B_1)%%%%%
{\small
$$
\begin{bmatrix}
1 \ & x_1 \ & x_1^2\  & x_1^3 \ & x_1^4 \  & x_1^5 \  & \cdots \ &\cdots \ & x_1^{d_t-2} & x_1^{d_t-1} \ 
\\
0 \ & 1 \ & s_1(2)\  & s_2(2)\  & s_3(2)\  & s_4(2)\  & \cdots \ & \cdots \ & s_{d_t-3}(2)\  & s_{d_t-2}(2)
\\
0 \ & 0 \ & 1 \ & s_1(3) \ & s_2(3) \  & s_3(3) \ & \cdots & \cdots & s_{d_t-4}(3)\  & s_{d_t-3}(3) \ 
\\
0 & 0 & 0 & 1 & s_1(4) & s_2(4) & \cdots & \cdots & s_{d_t-5}(4) & s_{d_t-4}(4)
\\
\vdots  &\vdots & \ddots &\ddots  &\ddots & \ddots &\ddots  &\ddots &\vdots   &\vdots
\\
0 & \cdots &\cdots  & \cdots & 0 & 1 & s_1(k)\mbox{ } &s_2(k) & \cdots & s_{d_t-k}(k)
\end{bmatrix}
$$
}
\noindent
Similarly, 
by easy Gaussian elimination of rows of
matrices, 
the matrix $A_2$ 
%and $A_3$ 
reduces to the matrix $B_2$,
%and $B_3$ , 
where $B_2$ is the matrix of the following form
{\small
\begin{equation*}
%B_2=
\begin{bmatrix}
 0 & 1 & 2x_1 & 3x_1^2 & 4x_1^3 & 5x_1^4 & \cdots & \cdots 
 %& \cdots 
 &   (d_t-1)x_1^{d_t-2}
 \\
0 & 0 & 2 &3s_1(2)\  & 4s_2(2)\  & 5s_3(2)\  & \cdots & \cdots 
%&\cdots 
&
  (d_t-1)s_{d_t-3}(2)
 \\
 0 & 0 & 0 & 3 & 4s_1(3)\  & 5s_2(3)\  & \cdots & \cdots 
 %& \cdots 
 &
 (d_t-1)s_{d_t-4}(3)
 \\
 \vdots & \vdots &\ddots & \ddots & \ddots & \ddots & \ddots &\ddots & \ddots 
 %& \vdots
 \\
0 & \cdots & \cdots &\cdots & 0 & k-1\mbox{ } 
%&ks_1(k)\mbox{ }
&
ks_1(k)\mbox{ } & \cdots &
(d_t-1)s_{d_t-k-1}(k)
\end{bmatrix}
%%%%%
\end{equation*}
}
\noindent
Analogously, the matrix $A_3$ 
reduces to the matrix $B_3$,
where $B_3$ is the matrix of the following form
{\small
\begin{equation*}
%B_3=
\begin{bmatrix}
0 & 0 & 2 & 6x_1 & 12 x_1^2 & 20x_1^3 &\cdots %&\cdots 
%&
%&
%\cdots & 
%\cdots 
&
d_t(d_t-1)x_1^{d_t-3}
\\
0 & 0& 0 &6 & 12 s_1(2)\  &20 s_2(2)\  &\cdots &
%&
%&
d_t(d_t-1)s_{d_t-k-2}(2)
\\
\vdots& \vdots &\ddots&\ddots &\ddots &\ddots & \ddots & \vdots 
%&
%\ddots&
%\vdots
\\
0& \cdots &\cdots &\cdots & 0 & (k-2)(k-3)\mbox{ }\mbox{ }  %&k(k-1)s_1(k)
\mbox{ }&\cdots
%&\cdots 
&\  d_t(d_t-1)s_{d_t-k-3}(k)
\end{bmatrix}
\end{equation*}
}
\noindent
%%%%%
If we repeat this process, we finally obtain the reduced $(k\times d_t)$
matrices $\{B_l\}_{l=1}^n$
such that each $A_l$ reduces to the matrix $B_l$.
Now 
define the $(lk\times d_t)$ matrix $C_l$ (for $1\leq l\leq n)$
inductively
by $C_1=B_1$ and
$C_l=
\begin{bmatrix}
C_{l-1}
\\
B_l
\end{bmatrix}$
for $2\leq l\leq n$.
Then
by induction on $t$ and some tedious calculations, we see that
each matrix $C_l$ has rank $kl$ for each $1\leq l\leq n$.
%%%
%%%
%%%
%%%
%%%%%%
%Since 
% $\{x_j\}_{j=1}^k\in C_k(\C)$,  
%it follows from the properties of Vandermonde matrices that 
%the above condition (\ref{eq: condition Fn}) 
%%given by $n$ matrix represntations 
%gives exactly $nk$ independent conditions on the coefficients of $f_t(z)$
%for each $1\leq t\leq r$.
Thus  the space of monic polynomials $f_t(z)\in \P^{d_t}$ which satisfies
(\ref{eq: condition Fn})
is the intersection of $nk$ affine hyperplanes in general position
and it is an affine subspace of $\P^{d_t}\cong \C^{d_t}$ with
 codimension $nk$.
%%%%%
Hence,
the fibre $\pi_{k,D}^{-1}(c)$ is homeomorphic  to the product of an open $(k-1)$-simplex
 with the real affine space of dimension
 $2\sum_{i=1}^r(d_i-nk)=2N(D)-2nrk$.
It is now easy to show that  $\pi_{k,D}$ satisfies the local triviality.
%is a (locally trivial) real affine bundle over $C_{k;\Sigma}$.
Thus, it is the real affine bundle over $C_{k;\Sigma}$ with rank $l_{D,k,n}
=2N(D)-2nrk+k-1$.
%%%%%%%%%%%%%%%%%%%%
\end{proof}
%%(End of proof of Lemma 4.9)%%%
%%%
%%%
%%%
%%%
%%%
%%%%%%(Lemma 4.10)%%
\begin{lmm}\label{lemma: E11}
%%%%%%%%%%%%%%%%%%%
If $1\leq k\leq  \lfloor \frac{d_{\rm min}}{n}\rfloor$, there is a natural isomorphism
$$
E^{1;D}_{k,s}\cong
H^{2nrk-s}_c(C_{k;\Sigma};\pm \Z),
$$
where 
the twisted coefficients system $\pm \Z$  comes from
the Thom isomorphism.
\end{lmm}
%%%%
\begin{proof}
%%%(Proof of Lemma 4.10)%%
Suppose that $1\leq k\leq \lfloor \frac{d_{\rm min}}{n}\rfloor$.
%%%
By Lemma \ref{lemma: vector bundle*}, there is a
homeomorphism
$
(\SZ_k\setminus\SZ_{k-1})_+\cong T(\xi_{D,k}),
$
where $T(\xi_{D,k,n})$ denotes the Thom space of
%one-point compactification of
$\xi_{D,k,n}$.
%%%%%%%
Since $(2N(D)+k-s-1)-l_{D,k,n}
=
2nrk-s,$
%$$
%%%%%
the Thom isomorphism gives a natural isomorphism 
%%%
$
E^{1;d}_{k,s}
\cong 
\tilde{H}^{2N(D)+k-s-1}(T(\xi_{d,k,n});\Z)
\cong
H^{2nrk-s}_c(C_{k;\Sigma};\pm \Z).
$
\end{proof}
%%%(End of proof of Lemma 4.10)%%%%%%
%%
%%
%%
%%
%%
%%%(Definition 4.11)%%%
\begin{dfn}\label{dfn: stab}
%%%%%%%%%%%%%%%%%%%%%%
{\rm
For  an $r$-tuple 
$D=(d_1,\cdots ,d_r)\in \N^r$,
let $U_D\subset \C$ denote the subspace defined by
%%(4.17)%%
\begin{equation}
%%%%%%%%%
U_D=\{w\in \C:\mbox{Re}(w)<N(D)\},
%%%%%%%%
\end{equation}
%%%%%%%%
%%
%%
%%
%%
and let $\varphi_D:\C\stackrel{\cong}{\longrightarrow}
U_D$ be any fixed homeomorphism.
Moreover, we choose any
mutually distinct $r$ points 
$x_1,\cdots ,x_r\in \C\setminus U_D$ freely and fix them.
%%%%
\par
%%%(i)%%
(i)
%%%%%%%%
For each monic polynomial $f(z)=\prod_{k=1}^d(z-\alpha_k)\in \P^d$
of degree $d$, let $\varphi_D(f)$
denote the monic polynomial of the same degree $d$ given by
%%(4.18)%%
\begin{equation}
%%%%
\varphi_D(f)=\prod_{k=1}^d(z-\varphi_D(\alpha_k))\in \P^d.
%%%
\end{equation}
%%%
\par
%%%%%%
%%%(ii)%%
(ii)
%%%%%%
For each $r$-tuple
$\textbf{\textit{a}}=(a_1,\cdots ,a_r)\in (\Z_{\geq 0})^r$
with $\textbf{\textit{a}}\not= {\bf 0}_r$, define
the stabilization map
%%(4.19)%%
\begin{equation}
%%%%%%%%%
s_{D,D+\textbf{\textit{a}}}:
\po^{D,\Sigma}_n \to \po^{D+\textbf{\textit{a}},\Sigma}_n
\quad
\mbox{ by }
%%%%%%%%
\end{equation}
%%%%
%%(4.20)%%
\begin{equation}
%%%%%%%%%%%
s_{D,D+\textbf{\textit{a}}}(f)=
(\varphi_D(f_1)(z-x_1)^{a_1},\cdots ,\varphi_D(f_r)(z-x_r)^{a_r})
%%%%%%%%%%%
\end{equation}
%%%%%%%%%%%
for $f=(f_1(z),\cdots ,f_r(z))\in \po^{D,\Sigma}_n$.
\qed
}
%%%%%%
\end{dfn}
%%%(End of Definition 4.11)%%
%%
%%
%%
%%
%%(Remark 4.12)%%
\begin{rmk}
%%%%%%%%
{\rm
%%(i)%%
(i)
%%%%%%
Note that the definition of the map $s_{D,D+\textbf{\textit{a}}}$ depends
on the choice of the homeomorphism
$\varphi_D$ and the points $\{x_k:1\leq k\leq r\}$, but one show that
the homotopy type of it does not depend on these choices.
%%%%%
%%%%%
%%%%%
\par
%%(ii)%%
(ii)
%%%%%%%
Let $\textbf{\textit{a}},\textbf{\textit{b}}\in (\Z_{\geq 0})^r$
be any two $r$-tuples such that
$\textbf{\textit{a}},\textbf{\textit{b}}\not={\bf 0}_r$.
Then it is easy to see that the equality
%%(4.21)%%
\begin{equation}\label{eq: stab-compo}
%%%%%%%%%%%%%%%%%%%
(s_{D+\textbf{\textit{a}},D+\textbf{\textit{a}}+\textbf{\textit{b}}})
\circ (s_{D,D+\textbf{\textit{a}}})
=s_{D,D+\textbf{\textit{a}}+\textbf{\textit{b}}}
\quad
(\mbox{up to homotopy)}
%%%%%%%
\end{equation}
%%%%%%%
holds.
Thus we mainly only consider the stabilization map $s_{D,D+\textbf{\textit{e}}_i}$
for each $1\leq i\leq r$, where
$\textbf{\textit{e}}_1=(1,0,\cdots ,0),
\textbf{\textit{e}}_2=(0,1,0,\cdots ,0),\cdots ,
\textbf{\textit{e}}_r=(0,0,\cdots ,0,1)\in \R^r$
denote the standard basis of $\R^r$.
\qed
%%%%%
}
%%%%
%%%%%%%%
\end{rmk}
%%(End of Remark 4.12)%%%
%%
%%
%%
%%%%%%%%%%%%%%%%
\par
%%%%%%%%%%%%%%%%
Now let $1\leq i\leq r$ and consider
the stabilization map
%%%(4.22)%%
\begin{equation}
%%%
s_{D,D+\textbf{\textit{e}}_i}:\po^{D,\Sigma}_n\to 
\po^{D+\textbf{\textit{e}}_i,\Sigma}_n.
\end{equation}
%%%
%%%
%%
It is easy to see that it extends to an open embedding
%%
%%%(4.23)%%
\begin{equation}\label{equ: sssd}
%%%%%%%%%
s_{D,i}:\C \times 
\po^{d,\Sigma}_n
\to
\po^{D+\textbf{\textit{e}}_i,\Sigma}_n
\end{equation}
%%%%%%%% 
%%%
It also naturally extends to an open embedding
%%()%%
%%%
$
\tilde{s}_{D,i}:\P^D\to \P^{D+\textit{\textbf{e}}_i}
$
%\end{equation}
and  by  restriction  we obtain an open embedding
%%%
%%%(4.24)%%
\begin{equation}\label{equ: open embedding}
%%%%%%%
\tilde{s}_{D,i}:\C\times \Sigma_D\to 
\Sigma_{D+\textit{\textbf{e}}_i}.
\end{equation}
%%%
Since one-point compactification is contravariant for open embeddings,
this map induces a map
$\tilde{s}_{D,i +}:(\Sigma_{D+\textit{\textbf{e}}_i})_+
\to
(\C \times \Sigma_D)_+=S^{2}\wedge \Sigma_{D+}$ and
we can easily see that the following diagram
%%%
%%%%%(4.25)%%
\begin{equation}\label{eq: diagram}
%%%%%%%%%%%%%
\begin{CD}
\tilde{H}_k(\po^{D,\Sigma}_n;\Z) @>{s_{D,D+\textbf{\textit{e}}_i}}_*>>
\tilde{H}_k(\po^{D+\textit{\textbf{e}}_i,\Sigma}_n;\Z)
\\
@V{AD}V{\cong}V @V{AD}V{\cong}V
\\
H^{2N(D)-k-1}_c(\Sigma_D;\Z)
@>{\tilde{s}_{D,i+}}^{\ *}>>
H^{2N(D)-k+1}_c(\Sigma_{D+\textit{\textbf{e}}_i};\Z)
\end{CD}
\end{equation}
%%%%%%%%%%%%% 
is commutative,
where $AD$ is the Alexander duality isomorphism and
 ${\tilde{s}_{D,i+}}^{\ *}$ denotes the composite of 
%homomorphisms 
the suspension isomorphism with the homomorphism
${(\tilde{s}_{D+})^*}$,
$$
%\begin{CD}
H^{M}_c(\Sigma_D;\Z)
\stackrel{\cong}{\rightarrow}
H^{M+2}_c
(\C\times \Sigma_D;\Z)
%\stackrel{(\tilde{s}_{d+})^*}{\longrightarrow}
\stackrel{(\tilde{s}_{D,i+})^*}{\longrightarrow}
H^{M+2}_c(\Sigma_{D+\textit{\textbf{e}}_i};\Z),
%\end{CD}
$$
where $M=2N(D)-k-1$.
%%%%%
By the universality of the non-degenerate simplicial resolution
\cite{Mo2}, 
the map $\tilde{s}_{D,i}$ also naturally extends to a filtration preserving open embedding
%%%()%%%%%
%\begin{equation}\label{equ: flitr-preserve map}
$\tilde{s}_{D,i}:\C \times \SZ \to \mathcal{X}^{D+\textbf{\textit{e}}_i}$
%\end{equation}
%%%%%%%%%%%%%%
between non-degenerate simplicial resolutions.
This  induces a filtration preserving map
$(\tilde{s}_{D,i})_+:
\mathcal{X}^{D+\textbf{\textit{e}}_i}_+\to 
(\C \times \SZ)_+
=S^{2}\wedge \SZ_+$,
and thus a homomorphism of spectral sequences
%%%
%%%(4.26)%%%%
\begin{align}\label{equ: theta1}
%%%%%%%%%%%%
&
\{ \tilde{\theta}_{k,s}^t:E^{t;D}_{k,s}\to 
E^{t;D+\textit{\textbf{a}}}_{k,s}\},
\quad \mbox{where}
\\
%%%%%%%
%
\nonumber
&
\begin{cases}
\big\{E^{t;D}_{k,s}, \tilde{d}^{t}:
E^{t;D}_{k,s}\to E^{t;D}_{k+t,s+t-1}
\big\}
&\Rightarrow 
 H_{s-k}(\po^{D,\Sigma}_n;\Z),
\\
%%
%\nonumber
%&
\big\{E^{t;D+\textit{\textbf{e}}_i}_{k,s}, \tilde{d}^{t}:
E^{t;D+\textit{\textbf{e}}_i}_{k,s}\to 
E^{t;D+\textit{\textbf{e}}_i}_{k+t,s+t-1}
\big\}
&\Rightarrow 
 H_{s-k}(\po^{D+\textit{\textbf{e}}_i}_n;\Z),
 \end{cases}
%%
%\\
%\nonumber
%\mbox{and }&
%\ \mbox{ $E^1$-terms are given by}
\\
\nonumber
&
\begin{cases}
E^{1;D}_{k,s} &=
H_c^{2N(D)+k-1-s}(\SZ_k\setminus \SZ_{k-1};\Z),
\\
E^{1;D+\textit{\textbf{e}}_i}_{k,s}
&=
H_c^{2N(D)+k+1-s}
(\mathcal{X}_k^{D+\textit{\textbf{e}}_i}\setminus 
\mathcal{X}_{k-1}^{D+\textit{\textbf{e}}_i};\Z).
\end{cases}
\end{align}
%%%%%%%%%%%%%%%
%%%(Lemma 4.13)%%%
\begin{lmm}\label{lmm: E1}
%%%%%%%%%%%%%%%%%
If $1\leq i\leq r$ and
$0\leq k\leq \lfloor \frac{d_{\rm min}}{n}\rfloor$, 
$\tilde{\theta}^1_{k,s}:E^{1;D}_{k,s}\to 
E^{1;D+\textit{\textbf{e}}_i}_{k,s}$ is
an isomorphism for any $s$.
\end{lmm}
%%%%%%%%%%%%%%%%
\begin{proof}
%%%%%%%%%%%%%%%%
Since the case $k=0$ is clear,
suppose that $1\leq k\leq \lfloor \frac{d_{\rm min}}{n}\rfloor$.
It follows from the proof of Lemma \ref{lemma: vector bundle*}
that there is a homotopy commutative diagram of affine vector bundles
$$
\begin{CD}
\C \times 
(\SZ_k\setminus\SZ_{k-1}) @>>> C_{k;\Sigma}
\\
@V{\tilde{s}_{D,i}}VV \Vert @.
\\
\mathcal{X}^{D+\textbf{\textit{e}}_i}_k\setminus 
\mathcal{X}^{D+\textbf{\textit{e}}_i}_{k-1} @>>> C_{k;\Sigma}
\end{CD}
$$
%%%%%
%%%
Since one-point compactification is contravariant for open embeddings,
the map $\tilde{s}_{D,i+}$ induces the map
$$
\tilde{s}_{D,i+}: 
(\mathcal{X}^{D+\textbf{\textit{e}}_i}_k\setminus 
\mathcal{X}^{D+\textbf{\textit{e}}_i}_{k-1})_+
\to 
(\C\times (\SZ_k\setminus\SZ_{k-1}))_+
=S^2\wedge (\SZ_k\setminus\SZ_{k-1})_+
$$
between one-point compactfications.
Recall from  Lemma \ref{lemma: vector bundle*} that
 $\xi_{D,k,n}$ (resp. $\xi_{D+\textbf{\textit{e}}_i,k,n}$) is a  real affine bundle over $C_{k;\Sigma}$
 with rank $l_{D,k,n}$ (resp. $l_{D+\textbf{\textit{e}}_i,k,n}$).
 Moreover, note that
\begin{align*}
(2N(D)+k-s+1)-l_{D,k,n}-2&=
(2N(D)+k-s+1)-l_{D+\textbf{\textit{e}}_i,k,n}
\\
&=2nrk-s.
\end{align*}
By the above commutative diagram and  Alexander duality,
%and Lemma \ref{lemma: E1},
we obtain the following commutative diagram:
%Note that the homomorphism $\theta^1_{k,s}$ is given by the composite of homomorphisms
{\small
$$
\begin{CD}
E^{1;D}_{k,s} @>\tilde{\theta}^1_{k,s}>> E^{1;D+\textbf{\textit{e}}_i}_{k,s}
\\
\Vert @. \Vert @.
\\
H^{2N(D)+k-s-1}_c
(\SZ_k\setminus\SZ_{k-1};\Z)
 @.
H^{2N(D+\textbf{\textit{e}}_i)+k-s-1}_c(
\mathcal{X}^{D+\textbf{\textit{e}}_i}_k\setminus 
\mathcal{X}^{D+\textbf{\textit{e}}_i}_{k-1};\Z)
\\
@V{suspension}V{\cong}V \Vert @.
\\
H^{2N(D)+k-s+1}_c
(\C\times 
(\mathcal{X}^{D}_k\setminus \mathcal{X}^{D}_{k-1});\Z) 
@>(\tilde{s}_{D,i})_+^*>>
H^{2N(D)+k-s+1}_c(
\mathcal{X}^{D+\textbf{\textit{e}}_i}_k\setminus 
\mathcal{X}^{D+\textbf{\textit{e}}_i}_{k-1};\Z)
\\
@V{AD}V{\cong}V @V{AD}V{\cong}V
\\
H^{2nrk-s}_c(C_{k;\Sigma};\pm\Z) @>>=> 
H^{2nrk-s}_c(C_{k;\Sigma};\pm\Z)
\end{CD}
$$
}
\newline
Hence, $\tilde{\theta}^1_{k,s}$ is an isomorphism for any $s$, and the assertion follows.
%%%%%%%%
\end{proof}
%%(End of proof of Lemma 4.13)%%%%
%%%%%%%%%%%%%%%
%%

Now we consider the spectral sequences induced by 
truncated simplicial resolutions.

%%%(Definition 4.14)%%%
\begin{dfn}
%%%%%%%%%%%%%%%%%%%%%%
{\rm
Let $X^{\Delta}$ denote the truncated 
(after the $\lfloor \frac{d_{\rm min}}{n}\rfloor$-th term) simplicial resolution of $\Sigma_D$
with the natural filtration
$$
\emptyset =X^{\Delta}_0\subset
X^{\Delta}_1\subset \cdots\subset
X^{\Delta}_{\lfloor d_{\rm min}/n\rfloor}\subset 
X^{\Delta}_{\lfloor d_{\rm min}/n\rfloor+1}=
X^{\Delta}_{\lfloor d_{\rm min}/n\rfloor +2}=\cdots =X^{\Delta},
$$
where $X^{\Delta}_k=\SZ_k$ if $k\leq \lfloor \frac{d_{\rm min}}{n}\rfloor$ and 
$X^{\Delta}_k=X^{\Delta}$ if $k\geq \lfloor \frac{d_{\rm min}}{n}\rfloor +1$.
%%%%%%
\par\vspace{1mm}\par
%%%%
Similarly,
let $Y^{\Delta}$ denote the truncated (after the $\lfloor \frac{d_{\rm min}}{n}\rfloor$-th term) simplicial resolution of 
$\Sigma_{D+\textit{\textbf{e}}_i}$
with the natural filtration
$$
\emptyset =Y^{\Delta}_0\subset
Y^{\Delta}_1\subset \cdots\subset
Y^{\Delta}_{\lfloor d_{\rm min}/n\rfloor}\subset 
Y^{\Delta}_{\lfloor d_{\rm min}/n\rfloor+1}=
Y^{\Delta}_{\lfloor d_{\rm min}/n\rfloor +2}=\cdots =Y^{\Delta},
$$
%%%%%%%
where $Y^{\Delta}_k=\mathcal{X}^{D+\textbf{\textit{e}}_i}_k$ if 
$k\leq \lfloor \frac{d_{\rm min}}{n}\rfloor$ and 
$Y^{\Delta}_k=Y^{\Delta}$ if $k\geq \lfloor \frac{d_{\rm min}}{n}\rfloor +1$.
\qed
%%%%%%%
}
%%%%%%%
\end{dfn}
%%%%%%(End of Definition 4.14)%%%%

%%%%
By using Lemma \ref{Lemma: truncated} and the same method
as in \cite[\S 2 and \S 3]{Mo3} (cf. \cite[Lemma 2.2]{KY4}), 
we obtain the following {\it  truncated spectral sequences}
%%%%()%%%%%%%%
\begin{eqnarray*}\label{equ: spectral sequ2}
%%%%%%%%%%%%%%%%
\big\{E^{t}_{k,s}, d^{t}:E^{t}_{k,s}\to 
E^{t}_{k+t,s+t-1}
\big\}
&\Rightarrow& H_{s-k}(\po^{D,\Sigma}_n;\Z),
%%%
\\
%%%
\big\{\ 
^{\p}E^{t}_{k,s},\  d^{t}:\ ^{\p}E^{t}_{k,s}\to 
\  ^{\p}E^{t}_{k+t,s+t-1}
\big\}
&\Rightarrow& H_{s-k}(\po^{D+\textit{\textbf{e}}_i}_n;\Z),
%%%%
\end{eqnarray*}
%%%
$$
E^{1}_{k,s}= H_c^{2N(D)+k-1-s}(X^{\Delta}_k\setminus X^{\Delta}_{k-1};\Z),
\ 
^{\p}E^{1}_{k,s}= H_c^{2N(D)+k+1-s}(Y^{\Delta}_k\setminus Y^{\Delta}_{k-1};\Z).
$$
%%%
By the naturality of truncated simplicial resolutions,
the filtration preserving map
$\tilde{s}_{D,i}:\C \times \SZ \to \mathcal{X}^{D+\textbf{\textit{e}}_i}$
  gives rise to a natural filtration preserving map
%%()%%
%\begin{equation}
%%%%%%%%
$\tilde{s}_{D,i}^{\p}:\C \times X^{\Delta} \to Y^{\Delta}$
%\end{equation}
%%%%%%%% 
which, in a way analogous to  (\ref{equ: theta1}), induces
a homomorphism of spectral sequences 
%%%%%%%%%%%
%%%(4.27)%%%%
\begin{equation}\label{equ: theta2}
%%%%%%%%%%%%
\{ \theta_{k,s}^t:E^{t}_{k,s}\to \ ^{\p}E^{t}_{k,s}\}.
%%%%%%%%%%%
\end{equation}
%%%%%%%%%%
%%%%(Lemma 4.15)%%
\begin{lmm}\label{lmm: Ed}
%%%%%%%%%%%%%%%%%
%%%%%%%%
\begin{enumerate}
%%(i)%%%
\item[$\I$]
%%%%%%%
If $k<0$ or $k\geq \lfloor \frac{\dmin}{n}\rfloor +2$,
$E^1_{k,s}=\ ^{\p}E^1_{k,s}=0$ for any $s$.
%%(ii)%%
\item[$\II$]
%%%%%%%%
$E^1_{0,0}=\ ^{\p}E^1_{0,0}=\Z$ and $E^1_{0,s}=\ ^{\p}E^1_{0,s}=0$ if $s\not= 0$.
%%(iii)%%%
\item[$\III$]
%%%%%%%%%
If $1\leq k\leq \lfloor \frac{d_{\rm min}}{n}\rfloor$, there are  
isomorphisms
$$
E^1_{k,s}\cong \ ^{\p}E^1_{k,s}\cong H^{2nrk-s}_c(C_{k;\Sigma};\pm \Z).
$$
%%(iv)%%
\item[$\IV$]
%%%%%%%%
If $1\leq k\leq \lfloor \frac{d_{\rm min}}{n}\rfloor$, 
$E^1_{k,s}=\ ^{\p}E^1_{k,s}=0$ for any 
$s\leq (2n\rmin (\Sigma)-2)k-1.$
%%(v)%%%
\item[$\V$]	
%%%%%%%%
$E^1_{\lfloor d_{\rm min}/n\rfloor +1,s}=
\ ^{\p}E^1_{\lfloor d_{\rm min}/n\rfloor +1,s} =0$ 
for any 
$s\leq (2n\rmin (\Sigma)-2)\lfloor \frac{d_{\rm min}}{n}\rfloor-1$.
%%%%%%%%%%%
\end{enumerate}
%%%%%%
%%%%%
\end{lmm}
%%%%%%
\begin{proof}
%%%%%(Proof of Lemma 4.15)%%%%
Let us write $\rmin =\rmin (\Sigma)$
and $d_{\rm min}^{\p}=\lfloor \frac{\dmin}{n}\rfloor$.
Since the proofs of both cases are identical,  it suffices to prove the assertions for $E^1_{k,s}$.
\par
(i), (ii), (iii)
Since $X^{\Delta}_k=\SZ_k$ for any $k\geq d_{\rm min}^{\p}+2$,
the assertions (i) and (ii) are clearly true.
Since $X^{\Delta}_k=\SZ_k$ for any $k\leq d_{\rm min}^{\p}$,
the assertion
(iii) easily follows from Lemma \ref{lemma: E11}.
\par
%%%(iv)%%%%
(iv)
Suppose that $1\leq k\leq
d_{\rm min}^{\p}$. Since
$\dim C_{k;\Sigma}=2k(1+nr-n\rmin )$
by (\ref{eq: dim CSigma}),
%%%
$2nrk > \dim C_{k;\Sigma}
\
\Leftrightarrow \
s\leq (2n\rmin -2)k-1.$
Thus, the assertion (iv) follows from (iii).
%%%(v)%%%%%%
\par
(v)
It remains to prove (v).
%%%%
By Lemma \cite[Lemma 2.1]{Mo3}, we see that
%%%%%%%
\begin{align*}
%%%%%%
\dim (X^{\Delta}_{d_{\rm min}^{\p}+1}\setminus 
X^{\Delta}_{d_{\rm min}^{\p}})
&=
\dim (\SZ_{d_{\rm min}^{\p}}\setminus \SZ_{d_{\rm min}^{\p}-1})+1
=l_{D,d_{\rm min}^{\p},n}+\dim C_{d_{\rm min}^{\p};\Sigma}+1
\\
&=2N(D)+3d_{\rm min}^{\p}-2n\rmin d_{\rm min}^{\p}.
\end{align*}
%%%
Since 
$E^1_{d_{\rm min}^{\p}+1,s}=
H_c^{2N(D)+d_{\rm min}^{\p}-s}
(X^{\Delta}_{d_{\rm min}^{\p}+1}\setminus X^{\Delta}_{d_{\rm min}^{\p}};\Z)$
and
%%%
\begin{align*}
%%%%%
2N(D)+d_{\rm min}^{\p}-s
&>
\dim (X^{\Delta}_{d_{\rm min}^{\p}+1}\setminus X^{\Delta}_{d_{\rm min}^{\p}})
=2N(D)+3d_{\rm min}^{\p}-2n\rmin d_{\rm min}^{\p}
\\
%\Leftrightarrow
%2N(D)+d_{\rm min}^{\p}-s
%&>
%2N(D)+3d_{\rm min}^{\p}-2n\rmin d_{\rm min}^{\p}
\Leftrightarrow
s&<(2nr_{\rm min}-2)d_{\rm min}^{\p}
\Leftrightarrow
s\leqq (2nr_{\rm min}-2)d_{\rm min}^{\p}-1,
%%%%%%%
\end{align*}
%%%%%%
%$2N(D)+d_{\rm min}-s>\dim (X^{\Delta}_{d+1}\setminus X^{\Delta}_d)$
%$\Leftrightarrow$
%$s\leq (2\rmin -2)d_{\rm min}-1$,
we see that
$E^1_{d_{\rm min}^{\p}+1,s}=0$ for any $s\leq (2n\rmin -2)d_{\rm min}^{\p}-1$
and the assertion (iv) follows.
%%%%
\end{proof}
%%%(End of proof of Lemma 4.15)%%%%%%%
%%%%
%%%%
%%%%
%%%%
%%%%
%%%(Lemma 4.16)%%
\begin{lmm}\label{lmm: 1-connected}
%%%%%%%%%%%%%%%%%%
If $n\geq 2$, the space
$\po^{D,\Sigma}_n$
is $(2n\rmin (\Sigma)-5)$-connected.
\end{lmm}
%%%%(Proof of Lemma 5.11)%%
\begin{proof}
%%%%
If $d_{min}< n$, $\po^{D,\Sigma}_n=\P^D$ is contractible and the assertion is clear and suppose that $d_{min}\geq n$.
Consider the spectral sequence
%%(5.25)%%
\begin{equation}\label{eq:SSSSS}
%%%%%%%%%%
\big\{E^{t}_{k,s}, d^{t}:E^{t}_{k,s}\to 
E^{t}_{k+t,s+t-1}
\big\}
\ 
\Rightarrow H_{s-k}(\po^{D,\Sigma}_n;\Z).
%%%%%%%%%%
\end{equation}
%%%%%%%%%%%%%%%%%%
Then by using Lemma \ref{lmm: Ed}, we  easily see that
$E^1_{k,s}=0$ if one of the following three conditions (a), (b) and (c) holds:
%%
%%
%%%%%%%%%%%%%%%%
\begin{enumerate}
%%%%%%%%%%%%%%%%%
%%(a)%%
\item[(a)]
%%%%%%%
$k<0$, or $k>\lfloor \frac{d_{min}}{n}\rfloor +2$, or
$k=0$ with $s\not= 0$.
%%%(b)%%
\item[(b)]
%%%%%%%%
If $1\leq k\leq \lfloor \frac{d_{min}}{n}\rfloor$,
$s-k\leq (2n\rmin (\Sigma)-3)k-1.$
%%%(c)%%
\item[(c)]
%%%%%%%%
If $k=\lfloor \frac{d_{min}}{n}\rfloor+1$, 
$s-(\lfloor \frac{d_{min}}{n}\rfloor+1)
\leq 
(2n\rmin (\Sigma)-3)\lfloor \frac{d_{min}}{n}\rfloor-2.$
%%%%%%%%%%%%
\end{enumerate}
%%%%%%%%%%%%
%%
%%
%%
Hence, when $(k,s)\not= (0,0)$, we see that
$E^1_{k,s}=0$ for any $(k,s)$ if the condition $s-k\leq 2n\rmin (\Sigma)-2$ is satisfied.
Thus, by the spectral sequence (\ref{eq:SSSSS}), we show that
%%%(4.29)%%
\begin{equation}\label{eq: H=0}
%%%%%%%
H_i(\po^{D,\Sigma}_n;\Z)=0
\quad
\mbox{ for any }1\leq i\leq 2n\rmin (\Sigma)-5.
\end{equation}
%%%
So it suffices to show that the space
$\po^{D,\Sigma}_n$ is simply connected.
Note that an element of $\pi_1(\po^{D,\Sigma}_n)$ can be
represented by the $m$-tuple
$(\eta_1,\cdots ,\eta_m)$ of strings of $m$-different colors
such that each $\eta_k$ $(1\leq k\leq m)$
is a string with total multiplicity $d_k$ as in the case of strings of the classical braid
group $\mbox{Br}_d=\pi_1(C_d(\C))$ \cite{Hansen}.
However, when all string of $m$-different colors moves continuously, the following case 
$(*)$ is not 
allowed to occur in this representation:
%%%
\begin{enumerate}
\item[$(*)$]All strings of $m$-different colors with multiplicity $\geq n$ pass through a single point.
\end{enumerate}
%%%
By using this string representation
one can show that any strings can intersect, pass through one another (except the case $(*)$),
and thus change the order as in \cite[\S Appendix]{GKY1}.
Thus one can show that $a\cdot b=b\cdot a$ for any $a,b\in \pi_1(\po^{D,\Sigma}_n)$.
Hence,
$\pi_1(\po^{D,\Sigma}_n)$  is an abelian group.
\par
On the other hand,
since $n\geq 2$ and $\rmin (\Sigma)\geq 2$,
$2n\rmin (\Sigma)-5\geq 8-5=3>1$.
Hence, $H_1(\po^{D,\Sigma}_n;\Z)=0$ by (\ref{eq: H=0}).
Thus by the Hurewicz theorem, we see that
there is an isomorphism
$\pi_1(\po^{D,\Sigma}_n)\cong H_1(\po^{D,\Sigma}_n;\Z)=0$.
%%%
\end{proof}
%%(End of proof of Lemma 4.16)%%%
%%

Now it is ready to prove the unstability result for $\po^{D,\Sigma}_n$.

%%%(Lemma 4.17)%%
\begin{lmm}\label{lmm: E2}
%%%%%%%%%%%%%%%%%
If $0\leq k\leq \lfloor \frac{d_{\rm min}}{n}\rfloor$, 
$\theta^1_{k,s}:E^{1}_{k,s}\stackrel{\cong}{\longrightarrow} \ ^{\p}E^{1}_{k,s}$ is
an isomorphism for any $s$.
%%%%%%%%%%%%%
\end{lmm}
%%%%%%%%%%%%%%%%
\begin{proof}
Since $(X^{\Delta}_k,Y^{\Delta}_k)=(\SZ_k,\mathcal{X}^{D+\textbf{\textit{e}}_i}_k)$ 
for $k\leq \lfloor \frac{\dmin}{n}\rfloor$,
the assertion follows from Lemma \ref{lmm: E1}.
\end{proof}
%%%%%%%%%%%%%%
%%%
%%%
%%%
%%%(Stability Theorem of stabilization maps)%%%%
%%%(Stability Theorem of stabilization maps)%%%%
%%%(Stability Theorem of stabilization maps)%%%%
%%%
%%%%%(Theorem 4.18: Theorem III)%%
\begin{thm}\label{thm: III}
%%%%%%%%%%%%%%%%%%%
Let $n\geq 2$. Then
for each $1\leq i\leq r$,
the stabilization map
$$
s_{D,D+\textbf{\textit{e}}_i}:
\po^{D,\Sigma}_n\to 
\po^{D+\textbf{\textit{e}}_i,\Sigma}_n
$$
is a homotopy equivalence through dimension
$d(D;\Sigma ,n)$, where
$d(D;\Sigma ,n)$ denotes the integer given by (\ref{eq: dDSigma}).
%$d(D;\Sigma ,n)=(2n\rmin (\Sigma)-3)\lfloor \frac{d_{min}}{n}\rfloor -2.$
%%%%%%%%%%%%%%%%%
\end{thm}
%%%%%%%%%%%%%%%%%%
\begin{proof}
%%%%%%%%%%%%%%%%%
%%%(Proof of Theorem 4.18)%%%%%
%%
%%
We write $\rmin =\rmin (\Sigma)$
and
$d_{\rm min}^{\p}=\lfloor \frac{d_{\rm min}}{n}\rfloor$
as in the proof of Lemma \ref{lmm: Ed}.
If $n\geq 2$, the spaces
$\po^{D,\Sigma}_n$ and 
$\po^{D+\textbf{\textit{e}}_i,\Sigma}_n$ are simply connected
by Lemma \ref{lmm: 1-connected}.
Thus it suffices to prove that the map
$s_{D,D+\textbf{\textit{e}}_i}$ is a homology equivalence
through dimension $d(D;\Sigma,n)$.
\par
Let us consider the homomorphism
$\theta_{k,s}^t:E^{t}_{k,s}\to \ ^{\p}E^{t}_{k,s}$
of truncated spectral sequences given in (\ref{equ: theta2}).
%%%%
By using the commutative diagram (\ref{eq: diagram}) and the comparison theorem for spectral sequences, 
it suffices to prove that the positive integer $d(D;\Sigma ,n)$ 
has the 
following property:
%%%%%%%%%%%%%%%%%%
\begin{enumerate}
%%%%%%%%%%%%%%%%%%
\item[$(\dagger)$]
%%%%%%%%%%%%%%%%%
$\theta^{\infty}_{k,s}$
is  an isomorphism for all $(k,s)$ such that $s-k\leq d(D;\Sigma ,n)$.
%%%%%%%%%%%%%%%%
\end{enumerate}
%%
%%
%%(The case r<0 or r \leq d+2)%%%%%
By Lemma \ref{lmm: Ed}, 
$E^1_{k,s}=\ ^{\p}E^1_{k,s}=0$ if
$k<0$, or if $k\geq \dmin^{\p} +2$, or if $k=\dmin^{\p} +1$ with 
$s\leq (2n\rmin -2)\dmin^{\p} -1$.
Since $\{(2n\rmin -2)\dmin^{\p} -1\}-(\dmin^{\p}+1)=(2n\rmin -3)\dmin^{\p} -2
=d(D;\Sigma ,n)$,
we  see that:
%()%%
%%%%%%%
%%%%%%%
\begin{enumerate}
%%%%%%%
\item[$(\dagger)_1$]
if $k< 0$ or $k\geq \dmin^{\p} +1$,
$\theta^{\infty}_{k,s}$ is an isomorphism for all $(k,s)$ such that
$s-k\leq d(D;\Sigma ,n)$.
\end{enumerate}
\par
Next,  assume that $0\leq k\leq \dmin^{\p}$, and investigate the condition that
$\theta^{\infty}_{k,s}$  is an isomorphism.
%%%
%Then
Note that the groups $E^1_{k_1,s_1}$ and $^{\p}E^1_{k_1,s_1}$ are not known for
%%(S_1)%%%
$(u,v)\in\mathcal{S}_1=
\{(\dmin^{\p}+1,s)\in\Z^2:s\geq (2n\rmin -2)\dmin^{\p} \}$.
%%%%
%%%%
By considering the differentials $d^1$'s of
$E^1_{k,s}$ and $^{\p}E^1_{k,s}$,
%$d^1:\E^1_{k,s}\to \E^{1}_{k+1,s}$
%and
%$d^1:\Ed^1_{k,s}\to \Ed^{1}_{k+1,s}$
and applying Lemma \ref{lmm: E2}, we see that
$\theta^2_{k,s}$ is an isomorphism if
$(k,s)\notin \mathcal{S}_1 \cup \mathcal{S}_2$, where
%%%
%%(S_2)%%%
$$
\mathcal{S}_2=
\{(u,v)\in\Z^2:(u+1,v)\in \mathcal{S}_1\}
=\{(\dmin^{\p} ,v)\in \Z^2:v\geq (2n\rmin -2)\dmin^{\p}\}.
$$
%%%%
%%%
%%%%
A similar argument  shows that
$\theta^3_{k,s}$ is an isomorphism if
%%(S_3)%%
$(k,s)\notin \bigcup_{t=1}^3\mathcal{S}_t$, where
$\mathcal{S}_3=\{(u,v)\in\Z^2:(u+2,v+1)\in \mathcal{S}_1\cup
\mathcal{S}_2\}.$
%%%%%
%\par
%%%%
Continuing in the same fashion,
considering the differentials
$d^t$'s on $E^t_{k,s}$ and $^{\p}E^t_{k,s}$
and applying the inductive hypothesis,
%Lemma \ref{lmm: E2}, 
we  see that $\theta^{\infty}_{k,s}$ is an isomorphism
if $\dis (k,s)\notin \mathcal{S}:=\bigcup_{t\geq 1}\mathcal{S}_t
=\bigcup_{t\geq 1}A_t$,
where  $A_t$ denotes the set
%%%%(Def. of A_t)%%%%%%%
$$
A_t=
\left\{
\begin{array}{c|l}
 &\mbox{ There are positive integers }l_1,\cdots ,l_t
\mbox{ such that},
\\
(u,v)\in \Z^2 &\  1\leq l_1<l_2<\cdots <l_t,\ 
u+\sum_{j=1}^tl_j=\dmin^{\p} +1,
\\
& \ v+\sum_{j=1}^t(l_j-1)\geq (2n\rmin -2)\dmin^{\p}
\end{array}
\right\}.
$$
%%%%%%%% 
Note that 
if this set was empty for every $t$, then, of course, the conclusion of 
Theorem \ref{thm: III} would hold in all dimensions (this is known to be false in general). 
%\par
If $\dis A_t\not= \emptyset$, it is easy to see that
%%%
\begin{align*}
a(t)&=
\min \{s-k:(k,s)\in A_t\}=
(2n\rmin -2)\dmin^{\p} -(\dmin^{\p} +1)+t
\\
&=
(2n\rmin -3)\dmin^{\p}+t-1
=d(D;\Sigma ,n)+t+1.
\end{align*}
%%%
Hence, we obtain that
$\min \{a(t):t\geq 1,A_t\not=\emptyset\}=d(D;\Sigma ,n)+2.$
%%%
 Since $\theta^{\infty}_{k,s}$ is an isomorphism
for any $(k,s)\notin \bigcup_{t\geq 1}A_t$ for each $0\leq k\leq \dmin^{\p}$,
we have the following:
%%(*2)%%%
\begin{enumerate}
%%%%%%%%
\item[$(\dagger)_2$]
If $0\leq k\leq \dmin^{\p}$,
$\theta^{\infty}_{k,s}$ is  an isomorphism for any $(k,s)$ such that
$s-k\leq  d(D;\Sigma ,n)+1.$
\end{enumerate}
%%%%%%
Then, by $(\dagger)_1$ and $(\dagger)_2$, we know that
$\theta^{\infty}_{k,s}:E^{\infty}_{k,s}\stackrel{\cong}{\rightarrow}
\ ^{\p}E^{\infty}_{k,s}$ 
is an isomorphism for any $(k,s)$
if $s-k\leq d(D;\Sigma ,n)$. 
Hence, by $(\dagger)$  we have the desired
assertion and this completes the proof of Theorem \ref{thm: III}.
%%%%
%%%%
\end{proof}
%%(End of proof of Theorem 5.13)%%
%%
%%
%%
%%(Coroolary of stability Theorem)%%
%%(Coroolary of stability Theorem)%%
%%(Coroolary of stability Theorem)%%
%%%%%%%%%%%%%%%%%%%%%%%%%%%%%%%%%%%%
%%%%%(Corollary 4.19: Corollary III)%%
%%%%%%%%%%%%%%%%%%%%%%%%%%%%%%%%%%%%
\begin{crl}\label{crl: III}
%%%%%%%%%%%%%%%%%%%
Let $n\geq 2$.
Then
for each $\textbf{\textit{a}}\in (\Z_{\geq 0})^r$ with
$\textbf{\textit{a}}\not= {\bf 0}_r$,
the stabilization map
$$
s_{D,D+\textbf{\textit{a}}}:
\po^{D,\Sigma}_n\to 
\po^{D+\textbf{\textit{a}},\Sigma}_n
$$
is a homotopy equivalence through dimension
$d(D;\Sigma ,n)$. 
%where
%$d(D;\Sigma ,n)$ denotes the integer given by
%$d(D;\Sigma ,n)=(2n\rmin (\Sigma)-3)\lfloor \frac{d_{min}}{n}\rfloor -2.$
%%%%%%%%%%%%%%%%%
\end{crl}
%%%%%%%%%%%%%%%%%%
\begin{proof}
%%%%%%%
The assertion easily follows from
(\ref{eq: stab-compo}) and Theorem \ref{thm: III}.
%%%
\end{proof}
%%(End of proof of Corollary 4.19)%%%
%%%
%%
%%(End of SECTION 4)%%%%%
%%%%%%%%%%%%%%
%%%
%%%
%%%
%%%
%%%
%%%
%%%
%%%
%%%
%%%
%%%
%%%
%%%
%%%

%%%%%%%(SECTION 5: Scanning maps)%%%%%%%
%%%%%%%(SECTION 5: Scanning maps)%%%%%%%
%%%%%%%%%%%%%%%%%%%%%%%%%
\section{Scanning maps}\label{section: scanning maps}
%%%%%%%%%%%%%%%%%%%%%%%%%
In this section we consider configuration spaces and the scanning map.

%%%%
%%(Definition 5.1)%%%
\begin{dfn}
%%%%%%%%%%%%%%%%%%%%%
{\rm
 For a positive integer $d\geq 1$ and a based space $X$, let
 $\SP^d(X)$ denote {\it the $d$-th symmetric product} of $X$ defined as
the orbit space 
%%(5.1)%%
\begin{equation}
\SP^d(X)=X^d/S_d,
\end{equation}
where the symmetric group $S_d$ of $d$ letters acts on the $d$-fold
 product $X^d$ in the natural manner.
 \qed
 }
 %%(End of Definition 5.1)%%%
\end{dfn}
%%%%%%%%%%%%%%%%%%%%%
%%%
%%(Remark 5.2)%%
\begin{rmk}
%%%%%%%%%%
{\rm
%%(i)%%
(i)
%%%%%%
An element $\eta\in \SP^d(X)$ may be identified with a formal linear
combination
%%(5.2)%%%
\begin{equation}
%%%
\eta =\sum_{k=1}^sn_k\alpha_k,
\end{equation}
%%%
 where
$\alpha_1,\cdots ,\alpha_s$ are distinct points in $X$
and $n_1,\cdots ,n_s$ are positive integers such that
$\sum_{k=1}^sn_k=d$.
In this situation we shall refer to $\eta$ as configuration 
(or divisor) of points, the points $\alpha_k\in X$
having {\it a multiplicity} $d_k$.
\par
%%(ii)%%
(ii)
%%%%%%%
For example, when $X=\C$, there is a natural homeomorphism
%%(5.3)%%
\begin{equation}\label{eq: psi}
%%%
\begin{CD}
\P^d @>\psi_d>\cong> \SP^d(\C) 
\\
f(z)=\prod_{k=1}^s(z-\alpha_k)^{n_k} @>>>
\eta =\sum_{k=1}^sn_k\alpha_k
\end{CD}
\end{equation}
%%%%
where
$n_k\in \N$ with $\sum_{k=1}^sn_k=d$.
\qed
}
%%%%%%%%%%%%
\end{rmk}
%%%(End of Remark 5.2)%%%

%%(Definition 5.3)%%%
\begin{dfn}
%%%%%%%%%%%%%%%%%%%%%
{\rm
(i)
When $A\subset X$ is a closed subspace, define an equivalence relation
\lq\lq$\sim$\rq\rq \ on $\SP^d(X)$ by
%%(5.4)%%
\begin{equation}
%%%%%%%%
\eta_1\sim \eta_2\quad
\mbox{if }\quad
\eta_1 \cap (X\setminus A)=\eta_2 \cap (X\setminus A)
\quad
\mbox{for }\eta_1,\eta_2\in \SP^d(X).
\end{equation}
%%%
Define $\SP^d(X,A)$ as the quotient space
%%(5.5)%%
\begin{equation}
%%%%
\SP^d(X,A)=\SP^d(X)/\sim .
\end{equation}
%%%%%%%%
Note that the points in $A$ are ignored in $\SP^d(X,A)$.
\par
%%(ii)%%
(ii)
%%%%%%%
If $A\not=\emptyset$, we have a natural inclusion 
$\SP^d(X,A)\subset \SP^{d+1}(X,A)$
given by adding a point in $A$, and we can define
$\SP^{\infty}(X,A)$ as the union
%%(5.6)%%
\begin{equation}
\SP^{\infty}(X,A)=\bigcup_{d=1}^{\infty}\SP^d(X,A).
\end{equation}
%%(ii)%%%%%
\par
%%%%%%%%%
%%(iii)%%
(iii)
%%%%%%%%
For each $D=(d_1,\cdots ,d_r)\in (\Z_{\geq 1})^r$, let
$E_n^{D,\Sigma}(X)$ denote the subspace of $\SP^d(X)^{rn}$
given by
%%(5.7)%%
\begin{equation}\label{eq: ESD}
%%%%%%%%%%%%%%%%%%%%%%%%%
E_n^{D,\Sigma}(X)=
\big\{(\xi_1,\cdots ,\xi_r)
\in
\prod_{i=1}^r\SP^{d_i}(X)^n
:\mbox{(\ref{eq: ESD}.1),  (\ref{eq: ESD}.2)}
\big\}, 
%%%%%%%%%
\end{equation}
%%%%%
where two conditions (\ref{eq: ESD}.1) and (\ref{eq: ESD}.2) are given by
\begin{enumerate}
\item[(\ref{eq: ESD}.1)]
For each $1\leq i\leq r$,
$\xi_i=(\xi_{i,1},\cdots ,\xi_{i,n})\in \SP^{d_i}(X)^n$
with
$\xi_{i,j}\in \SP^{d_i}(X)$.
\item[(\ref{eq: ESD}.2)]
$\bigcap_{(i,j)\in \sigma\times [n]}\xi_{i,j}=\emptyset$
for any $\sigma \in I(\KS)$.
%%%%%%
\end{enumerate}
%%%%%%%
\par
%%(iv)%%
(iv)
%%%%%%%
Let $A\subset X$ be a closed subspace and $A\not=\emptyset$.
We define an equivalence relation \lq\lq$\sim$\rq\rq \ on
the space $E^{D,\Sigma}_n(X)$ by
$$
(\xi_1,\cdots ,\xi_r)\sim
(\eta_1,\cdots ,\eta_r)
\quad
\mbox{if }\quad
\xi_i \cap (X\setminus A)=\eta_i \cap (X\setminus A)
\quad
\mbox{for each }1\leq j\leq r.
$$
Let 
$E^{D,\Sigma}_n(X,A)$ be the quotient space
%%(5.8)%%
\begin{equation}
%%%%%%%%%
E^{D,\Sigma}_n(X,A)=E^{D,\Sigma}_n(X)/\sim .
%%%%%%%%%
\end{equation}
%%%%
Adding points in $A$ gives a natural inclusion
$
E^{D,\Sigma}_{n}(X,A)\subset 
E_n^{D+\textbf{\textit{e}}_i,\Sigma}(X,A)
$ for
each $1\leq i\leq r$.
So,  one can define the space
$E^{\Sigma}_{n}(X,A)$ as the union} 
%%(5.9)%%
\begin{equation}
%%%
E^{\Sigma}_{n}(X,A)=\bigcup_{D\in \N^r}
E^{D,\Sigma}_{n}(X,A).
\qquad\qquad
\qed
%%%%%%%
\end{equation}
%%%%%%%%%%%
%%%%
\end{dfn}
%%(End of Definition 5.3)%%%
%%
%%
%%
%%
%%
%%
%%(Remark 5.4)%%
\begin{rmk}\label{rmk: ESigma}
%%%%%%%%%%%%%%%%%
{\rm
(i)
For each $D=(d_1,\cdots ,d_r)\in \N^r$,
there is a natural homeomorphism
%%%(5.10)%%%
\begin{equation}\label{eq: identify}
%%%%%%%%%
\begin{CD}
\po^{d,\Sigma}_n @>\Psi_D>\cong>  E^{D,\Sigma}_n(\C)
\\
(f_1(z),\cdots ,f_r(z)) @>>> (\Psi _{d_1}(f_1(z)),\cdots ,\Psi_{d_r} (f_r(z)))
\end{CD}
\end{equation}
%%%%%%
where  $\Psi_d(f(z))\in \SP^d(\C)^n$ denotes the $n$-tuple of configuration given by
%%(5.11)%%
\begin{equation}
%%%
\Psi_d(f(z))=
(\psi_d(f(z)),\psi_d(f(z)+f^{\p}(z)),\cdots
,\psi_d(f(z)+f^{(n-1)}(z)))
%%%
\end{equation}
%%%
for $f(z)\in \P^d$,
where $\psi_d$ is the map defined in (\ref{eq: psi}).
%%%%%
\par
%%(ii)%%
(ii)
%%%%%%%
In general, $E^{D,\Sigma}_n(\C)$ is path-connected.
In fact, for any two points $\xi_0$ and $\xi_1$ in $E^{D,\Sigma}_n(\C)$ one can construct
a path $\omega :[0,1]\to E^{D,\Sigma}_n(\C)$ such that
$\omega (i)=\xi_i$ for $i\in\{0,1\}$ 
by the method explained in  \cite[\S Appendix]{GKY1}. Hence,
the space $\po^{d,\Sigma}_n$ is also path connected. 
\qed
}
\end{rmk}
%%%(End of Remark 5.4)%%

%%%(Definition 5.5)%%%
\begin{dfn}\label{dfn: stabilization etc}
%%%%%%%%%%%%%%%%%%%%%%
{\rm
Let $\varphi_D:\C\stackrel{\cong}{\longrightarrow}
U_D$ be any fixed homeomorphism, and
we choose any
mutually distinct fixed $r$ points 
$x_1,\cdots ,x_r\in \C\setminus U_D$ 
as in Definition \ref{dfn: stab}.
\par
%%%(i)%%
(i)
%%%%%%%
Let $d$ be a positive integer and let
$\eta=\sum_{k=1}^sn_ky_k\in \SP^d(\C)$
be any element such that $\{y_k\}_{k=1}^s\in C_s(\C)$
and $n_k\in \Z_{\geq 1}$ with $\sum_{k=1}^sn_k=d$.
In this situation let $\tilde{\varphi}_d(\eta)\in \SP^d(U_D)$ denote the configuration given by
%%(5.12)%%
\begin{equation}
\tilde{\varphi}_D(\eta)=\sum_{k=1}^sn_k\varphi_D(y_k).
\end{equation}
%%%%%%%%
\par
%%(ii)%%
(ii)
%%%%%%%%
When $\eta =(\eta_1,\cdots ,\eta_n)\in \SP^{d}(\C)^n$ with
$\eta_i\in\SP^d(\C)$, let
$\Phi_D(\eta)\in \SP^d(\C)^n$
denote the $n$-tuple of configurations given by
%%(5.13)%%
\begin{equation}
%%%
\Phi_D(\eta)=(\tilde{\varphi}_D(\eta_1),\cdots ,\tilde{\varphi}_D(\eta_n)).
%%%
\end{equation}
%%%
\par
%%(iii)%%
(iii)
%%%%%%%%
For each
$\textbf{\textit{a}}=(a_1,\cdots ,a_r)\not= {\bf 0}_r\in (\Z_{\geq 0})^r$,
define the stabilization map
%%(5.14)%%
\begin{equation}
%%%
\hat{s}_{D,D+\textbf{\textit{a}}}:
E^{D,\Sigma}_n(\C)\to 
E^{D+\textbf{\textit{a}}}_n(\C)
%%%%%%%%%%%
\end{equation}
%%%%%%%%%%%
by
%%%(5.15)%%
\begin{equation}
%%%
\hat{s}_{D,D+\textbf{\textit{a}}}
(\xi_1,\cdots ,\xi_r)=
(\Phi_D(\xi_1)+a_1\overline{x_1},\cdots ,\Phi_D(\xi_r)+a_r\overline{x_r})
\end{equation}
%%%
for $(\xi_1,\cdots ,\xi_r)\in E^{D,\Sigma}_n(U_D)$
with $\xi_i=(\xi_{i,1},\cdots ,\xi_{i,n})\in
\SP^{d_i}(\C)^n$,
where we set
%%(5.16)%%
\begin{equation}
%%%%%%%%%%%
\Phi_D(\xi_i)+a_i\overline{x_i}=
(\tilde{\varphi}_D(\xi_{i,1})+a_ix_i,\cdots ,\tilde{\varphi}_D(\xi_{i,n})+a_ix_i).
%%%%%%%%%%%%
\end{equation}
%%%
%%
%%
%%
It is easy to see that the diagram
%%(5.17)%%
\begin{equation}
%%%%%%%%%%
\begin{CD}
\po^{D,\Sigma}_n @>s_{D,D+\textbf{\textit{a}}}>> \po^{D+\textbf{\textit{a}}}_n
\\
@V{\Psi_D}V{\cong}V @V{\Psi_{D+\textbf{\textit{a}}}}V{\cong}V
\\
E^{D,\Sigma}_n (\C)@>\hat{s}_{D,D+\textbf{\textit{a}}}>> E^{D+\textbf{\textit{a}}}_n(\C)
\end{CD}
%%%%%%%%%
\end{equation}
%%%%%%%%%
is commutative.
\qed
%%%%%%
}
%%%%%%
\end{dfn}
%%(End of Definition 5.5)%%%%%

%%%%%%%%%%%%%%%%%%%%%
Now we are ready to define the scanning map.
%%
%%(Definition 5.6)%%
\begin{dfn}
%%%%%%%%%%%%%%%%
{\rm
Let $\epsilon_0>0$ be any fixed sufficiently small number and let
$U=\{w\in \C:\vert w\vert <1\}$.
For each $w\in \C$, let $U_w=\{x\in \C:\vert x-w\vert <\epsilon_0\}$.
%\par
Then
for an element $\eta =(\eta_1,\cdots ,\eta_r)\in E^{D,\Sigma}_n(\C),$
define a map
%%()%%
$
sc_D(\eta):\C \to 
E^{\Sigma}_{n}(\overline{U},\partial \overline{U})
$
by 
$$
w\mapsto
\eta \cap \overline{U}_w=
(\eta_1\cap \overline{U}_w,\cdots ,\eta_r\cap \overline{U}_w)
\in E^{\Sigma}_n(\overline{U}_w,\partial \overline{U}_w)
\cong
E_{n}^{\Sigma}(\overline{U},\partial \overline{U})
$$
for $w\in \C$, where we identify
$(\overline{U}_w,\partial \overline{U}_w)$
with $(\overline{U},\partial \overline{U})$
in the canonical way.
Since $\dis \lim_{w\to\infty}sc(\eta)(w)=(\emptyset,\cdots ,\emptyset),$
it naturally extends to a map
%%%(5.18)%%
\begin{equation}\label{equ: sc}
%%%%%%%%%
sc(\eta):S^2=\C\cup\infty \to E^{\Sigma}_{n}(\overline{U},\partial \overline{U})
%\simeq E^{\Sigma}_n(S^2,\infty)
\end{equation}
%%%%%%%%%%
with $sc(\eta)(\infty)=(\emptyset ,\cdots ,\emptyset).$
%%
%%%%%
Now we choose the point $\infty$ and the empty configuration
$(\emptyset,\cdots ,\emptyset)$ as the base-points of 
$S^2=\C \cup \infty$ and 
$E^{\Sigma}_n(\overline{U},\partial \overline{U})$, respectively. 
Then the map $sc(\eta)$ is a base-point preserving map, and
we obtain a map 
$$
sc:E^{D,\Sigma}_n(\C)\to
\Omega^2 E^{\Sigma}_n(\overline{U},\partial \overline{U}).
$$
However,
since $E^{D,\Sigma}_n(\C)$ is connected, the image of the map $sc$ is contained some path-component of 
$\Omega^2E^{\Sigma}_n(\overline{U},\partial \overline{U})$,
which we denote by
$\Omega^2_DE^{\Sigma}_n(\overline{U},\partial \overline{U}).$ 
Thus we have the map
%%%%
%%%(5.19)%%
\begin{equation}\label{equ: scanning}
sc_D:E^{D,\Sigma}_n(\C)\to
\Omega^2_D E^{\Sigma}_n(\overline{U},\partial \overline{U}).
\end{equation}
%%%
%%Since
Since we can identify $\po^{D,\Sigma}_n=E^{D,\Sigma}_n(\C)$
as in (\ref{eq: identify}),
we obtain 
the map
%%%
%%(5.20)%%%%%
\begin{equation}\label{equ: scanning map}
%%%%%%%%%%%
sc_D:\po^{D,\Sigma}_n
\to
\Omega^2_D E^{\Sigma}_n(\overline{U},\partial \overline{U}).
\end{equation}
%%%
\par We refer to this map (and others defined by the same method) as 
\lq\lq\ the scanning map\rq\rq. 
%%%
\par
Now let ${\bf 0}_r\not= \textbf{\textit{a}}\in (\Z_{\geq 0})^r$
be an $r$-tuple of integers.
Then
it is easy to see that there is a commutative diagram
%%(5.21)%%
\begin{equation}\label{CD1}
%%%%%%%%
\begin{CD}
\po^{D,\Sigma}_n @>sc_D>>
\Omega^2_D E^{\Sigma}_n(\overline{U},\partial \overline{U})
\\
@V{s_{D,D+\textbf{\textit{a}}}}VV   @VV{\simeq}V
\\
\po^{D+\textbf{\textit{a}},\Sigma}_n@>sc_{D+\textbf{\textit{a}}}>>
\Omega^2_{D+\textbf{\textit{a}}}
 E^{\Sigma}_n(\overline{U},\partial \overline{U})
\end{CD}
\end{equation}
%%%
\par
%%%
%%%
Let 
$\dis \po^{D+\infty,\Sigma}_n
=
\lim_{\textbf{\textit{a}}\to\infty}
\po^{D+\textbf{\textit{a}},\Sigma}_n
$
denote the colimit constructed from the stabilization maps
$\{s_{D,D+\textbf{\textit{a}}}:\textbf{\textit{a}}\in (\Z_{\geq 0})^r\}$,
where the notation 
$\textbf{\textit{a}}=(a_1,\cdots ,a_r)\to\infty$ 
means that
$\min \{a_k:1\leq k\leq r\}\to \infty$.
%%
%%%%%%%
\par
%%%%%%
Then by using (\ref{CD1}) we obtain 
{\it the stabilized scanning map}
%%(5.22)%%
\begin{equation}
S:
\po^{D+\infty}_n
=\lim_{\textbf{\textit{a}}\to\infty}\po^{D+\textbf{\textit{a}},\Sigma}_n
\to
\Omega^2_0 E^{\Sigma}_n(\overline{U},\partial \overline{U}),
\end{equation}
%%%
where  $\dis S=\lim_{\textbf{\textit{a}}\to\infty}sc_{D+\textbf{\textit{a}}}$ and
$\Omega^2_0 X$
denotes the path component of
$\Omega^2 X$
which contains the constant map.
\qed
}
\end{dfn}
%%%%(End of Definition 5.6)%%%
%%
%%
%%
%%
%%%(Scanning maps)%%%%%%%%%
%%%%(Theorem 5.7: scanning map)%%
\begin{thm}\label{thm: scanning map}
%%%%%%%%%%%%%%%%%%%%%
The stabilized scanning map
$$
S:
\po^{D+\infty}_n
%=\lim_{\textbf{\textit{a}}\to\infty}\po^{D+\textbf{\textit{a}},\Sigma}_n
\stackrel{\simeq}{\longrightarrow}
\Omega^2_0 E^{\Sigma}_n(\overline{U},\partial \overline{U})
$$
is a homotopy equivalence.
\qed
\end{thm}
%%%%%%%%%%%%%%%%%%
\begin{proof}
%%%%%%%%%%%
The assertion can be proved
by using Segal's scanning  method given in
\cite[Prop. 4.4]{Gu2} (cf. \cite{Gu1}) and \cite{GKY2}.
%%%
\end{proof}
%%%%%(End of proof of Theorem 5.7)%%%

Next we investigate about the homotopy type of the space
$E^{\Sigma}_n(\overline{U},\partial \overline{U})$.

%%(Definition 5.8)%%
\begin{dfn}\label{dfn: veeK}
%%%%
{\rm
Let $(X,*)$ be a based space, let $I$ be 
a collection of some subsets of $[N]=\{1,2,\cdots ,N\}$, and
let $\Sigma$ be a fan in $\R^m$.
%%
%%%
\begin{enumerate}
\item[(i)]
Let $\vee^IX$ denote the subspace of $X^N$ defined by
%%(5.23)%%
\begin{equation}
%%%
\vee^IX=\{(x_1,\cdots ,x_N)\in X^N:\ (\dagger)_I\},
\mbox{ where}
\end{equation}
%%%%
\begin{enumerate}
\item[]
$(\dagger)_I$\ \ 
For each $\sigma\in I$, there is some $j\in \sigma$ such that $x_j=*$.
\end{enumerate}
\item[(ii)]
Recall the set
%%(5.24)%%
\begin{equation} 
[r]\times [n]
=\{(i,j)\in \N^2:1\leq i\leq r,1\leq j\leq n\}
\end{equation}
%%%
of $rn$ points and let $I(\Sigma,n)$ denote the collection of subsets in $[r]\times [n]$ 
defined by
%%%(5.25)%%
\begin{equation}
%%%
I(\Sigma,n)=\{\sigma\times [n]:\sigma\in I(\KS)\}.
\end{equation}
%%%%
\item[(iii)]
Similarly let $\KS (n)$ denote the simplicial complex on the index set $[r]\times [n]$
defined by
%%%(5.26)%%
\begin{equation}
%%%%%%
\KS (n)=\{\tau \subset [r]\times [n]:\sigma\times [n] \not\subset \tau\
\mbox{ for any }\sigma\in I(\KS)\}.\quad
\qed
%%%
\end{equation}
%%%
\end{enumerate}
}
\end{dfn}
%%%%%%%%%%%%%(End of Definition 5.8)%%
%%

%%(Lemma 5.9)%%
\begin{lmm}[cf. \cite{KY6}, Lemma 6.3]\label{lmm: veeK}
%%%%%%%%%%%%%
Let $K$ be a simplicial complex on the index set $[N]$ and let
$(X,*)$ be a based space.
\begin{enumerate}
\item[$\I$]
$I(K)=\{\sigma\subset [N]:\sigma\not\subset \tau
\mbox{ for any }\tau\in K\}$.
%\item[$\II$]
%$K=\{\tau\subset [N]:\sigma\not\subset \tau
%\mbox{ for any }\tau\in I(K)\}$.
\item[$\II$]
$\mathcal{Z}_K(X,*)=\vee^{I(K)}X$.
\end{enumerate}
\end{lmm}
%%%%(Proof of Lemma 5.9)%%%%
\begin{proof}
%%%%%%%%%%%%%%%%%%%%%%%%%%%%
%%%
The assertion (i)  easily follows from the definition of simplicial complexes and the
assertion (ii) follows from \cite[Lemma 4.2]{KY9}.
%%%%
\end{proof}
%%%%(End of proof of Lemma 5.9)%%%

%%%(Lemma 5.10)%%
\begin{lmm}\label{lmm: E-infty}
%%%%%%%%%%%%%%%%
There is a homotopy equivalence
$$
r_{\Sigma}:
E^{\Sigma}_n(\overline{U},\partial \overline{U})
\stackrel{\simeq}{\longrightarrow}
DJ(\KS (n)).
%\mathcal{Z}_{\KS}(\CP^{\infty},*).
$$
\end{lmm} 
%%%%%%%%%%%%%%%%
\begin{proof}
%%%%%%%%%%%%%%%
For each $\epsilon >0$, let $U(\epsilon)=\{w\in \C: \vert w\vert <\epsilon\}\subset \C\cup\infty =S^2.$
The proof is analogous to that of
\cite[Prop. 3.1]{Se},  \cite[Lemma 7.10]{Ka1}
and \cite[Lemma 4.3]{KY9}.
Note that the space
$E^{\Sigma}_n(\overline{U},\partial \overline{U})$ is 
homeomorphic to the space
%%(5.27)%%
\begin{equation}
%%%%
E^{\Sigma}_{n}(S^2,\infty)
=\{(\eta_1,\cdots ,\eta_r)\in (\SP^{\infty}(S^2,\infty)^n)^r:
(*)_n\},
\quad
\mbox{ where}
\end{equation}
%%%%
%%%%%%%%%
\begin{enumerate}
\item[$(*)_n$]
%%%%
$\bigcap_{(i,j)\in \sigma\times [n]}\eta_{i,j}=\emptyset$
for any $\sigma\in I(\KS)$, where
$\eta_i=(\eta_{i,1},\cdots ,\eta_{i,n})\in \SP^{\infty}(S^2,\infty)^n$
with $\eta_{i,j}\in \SP^{\infty}(S^2,\infty)$ for each $1\leq i\leq r$.
\end{enumerate}
%%%
For each $\epsilon >0$, %let $U(\epsilon)=\{w\in \C: \vert w\vert <\epsilon\}$ and
let $E_{\epsilon}^{\Sigma}$ denote the open subset
of $E^{\Sigma}_{n}(S^2,\infty)$
consisting of all $r$-tuples
$(\eta_1,\cdots ,\eta_r)\in
E^{\Sigma}_{n}(S^2,\infty)$
such that, 
for any $\sigma \in I(\mathcal{K}_{\Sigma})$ there
exists some $(i,j)\in\sigma\times [n]$  satisfying the condition 
 $\eta_{i,j}\cap \overline{U(\epsilon)}=\emptyset.$
%%%
\par
Then the radial expansion defines a deformation retraction
%%(5.28)%%
\begin{equation}
%%%
r_{\epsilon}:E^{\Sigma}_{\epsilon}\stackrel{\simeq}{\rightarrow}
\vee^{I(\Sigma ,n)}\SP^{\infty}(S^2,\infty)
\end{equation}
%%%%%
(in this case, if $\eta_{i,j}\cap \overline{U(\epsilon)}=\emptyset$
and 
$(i,j)\in \sigma\times [n]$ (for any $\sigma\in I(K)$),
then the configuration $\eta_{i,j}$ gets retracted to $\infty$).
Since $E^{\Sigma}_{n}(S^2,\infty)=
\bigcup_{\epsilon>0}E^{\Sigma}_{\epsilon}$
and
there is a homeomorphism
%%%% 
$\SP^{\infty}(S^2,\infty)\cong
\CP^{\infty}$, 
there is a deformation retraction
%%(5.29)%%
\begin{equation}
%%%
E^{\Sigma}_n(S^2,\infty)
\stackrel{\simeq}{\rightarrow}
\vee^{I(\Sigma,n)}\SP^{\infty}(S^2,\infty)
\cong \vee^{I(\Sigma,n)}\CP^{\infty}.
\end{equation}
%%%%%%%%%%
%%%
Since 
$I(\Sigma,n)=\{\tau \subset [r]\times [n]:\tau\notin \KS (n)\}$,
by Lemma \ref{lmm: veeK},
we can identify
$\vee^{I(\Sigma,n)}\CP^{\infty}=
\mathcal{Z}_{\KS (n)}(\CP^{\infty},*)=
DJ(\KS (n))$.
Thus we obtain the desired homotopy equivalence.
%%%%%%%%%
\end{proof}
%%%%(End of Proof of Lemma 5.10)%%

%%(Remark 5.11)%%
\begin{rmk}\label{rmk: toric-variety-remark}
%%%%%
{\rm
%%%%%%%%%
For each $(i,j)\in [r]\times [n]$,
let $\n_{i,j}\in \Z^{mn}$
denote the lattice vector
defined by
%%(5.30)%%
\begin{equation}
\n_{i,j}=(\textbf{\textit{a}}_1,\cdots ,\textbf{\textit{a}}_n),
\mbox{ where we set }
\textbf{\textit{a}}_k=
\begin{cases}
\n_i & (k=j)
\\
{\bf 0}_m & (k\not= j)
\end{cases}
\end{equation}
%%%%%%%%%%%
and define the fan $F(\Sigma,n)$ in $\R^{mn}$ by
%%(5.31)%%
\begin{equation}
F(\Sigma ,n)=\{c_{\tau}:\tau\in \KS (n)\},
\end{equation}
%%%%%
where $c_{\tau}$ denotes the cone in $\R^{mn}$
generated by $\{ \n_{i,j}:(i,j)\in \tau\}$.
\par
Then
one can show that
there is a homeomorphism
%%(5.32)%%
\begin{equation}\label{eq: moment-angle}
\mathcal{Z}_{\KS}(\C^n,(\C^n)^*)\cong \mathcal{Z}_{\KS (n)}(\C,\C^*),
\end{equation}
%%%%%and that
and that 
 $\XS (n)$ is a toric variety associated to the fan $F(\Sigma,n)$.
Since the proof is tedious and we do not need this fact,
we omit the details.
 \qed
}
%%%
\end{rmk}
%%%%%(End of Remark 5.29)%%%

%%%(End of SECTION 5)%%%%%%
%%%
%%%
%%%
%%%
%%%
%%%

%%%%%%%%%%%%%%%%%%%%%%
%%%%(SECTION 6: Stable result)%%%%
%%%%%%%%%%%%%%%%%%%%%%
%%%%(SECTION 6: Stable result)%%%%
\section{The stable result}\label{section: stability}
%%%%%%%%%%%%%%%%%%%%%%%%%%%%%%

In this section we give the proof of the following stability result
(Theorem \ref{thm: IV}) by using the stabilized scanning map
and Theorem \ref{thm: scanning map}.
%%
%%
%%
%%(Definition 6.1)%%%
\begin{dfn}
%%%%
{\rm
Let $D=(d_1,\cdots ,d_r)\in \N^r$ and
$\textit{\textbf{a}}=(a_1,\cdots ,a_r)\in \N^r$ be two $r$-tuples of 
positive  integers 
such that
%%(6.1)%%
\begin{equation}
%%%%%%%
\sum_{k=1}^rd_k\textbf{\textit{n}}_k=\sum_{k=1}^r a_k\textbf{\textit{n}}_k={\bf 0}_m,
%%%%%%%%%
\end{equation}
%%%
and consider the following homotopy commutative diagram
%%%%(6.2)%%
\begin{equation}
%%%%%%%%%%
\begin{CD}
\po^{D,\Sigma}_n@>i_D>> \Omega^2_D\XS (n)@>>\simeq> \Omega_0^2\XS (n) 
\\
@V{s_D}VV @V{\simeq}VV \Vert @.%@V{\simeq}VV
\\
\po^{D+\textit{\textbf{a}},\Sigma}_n 
@>i_{D+\textit{\textbf{a}}}>> 
\Omega^2_{D+\textit{\textbf{a}}}\XS (n) @>>\simeq> \Omega^2_0\XS (n)
\end{CD}
%%%%%%%
\end{equation}
%%%%%%%
Then
by identifying
$\dis \po^{D+\infty,\Sigma}_n=\lim_{k\to\infty}\po^{D+k\textbf{\textit{a}},\Sigma}_n$,
we obtain the map
%%
%%
%%
%%(6.3)%%%%%%%%%
\begin{equation}\label{eq: inclusion stab}
%%%%%%%%%%%%%%%%
i_{D+\infty}=\lim_{k\to\infty}i_{D+k\textit{\textbf{a}}}:
\po^{D+\infty,\Sigma}_n=\lim_{k\to\infty}\po^{D+k\textbf{\textit{a}},\Sigma}_n
 \to \Omega^2_0\XS (n).
%%%%%%%%%%%
\end{equation}
%%%
}
\end{dfn}
%%(End of definition 6.1)%%%

The purpose of this section is to prove the following result.
%%%
%%%(Stable Theorem)%%%%
%%%(Theorem 6.2)%%
\begin{thm}\label{thm: IV}
%%%%%%%%%%%%%%%%%%%%%%%%%
The map
$i_{D+\infty}:
\po^{D+\infty,\Sigma}_n
 \stackrel{\simeq}{\longrightarrow}
 \Omega^2_0\XS (n)$
is a homotopy equivalence.
\end{thm}
%%%%
To prove the above result (Theorem \ref{thm: IV})
we recall the following definitions and results.

%%%(Definition 6.3)%%%
\begin{dfn}
%%%%%%%%%%%%%%%%%%%%%%
{\rm
%%(i)%%
(i)
%%%%%%
For an open set $X\subset \C$, let $F^{\Sigma}_n(X)$ denote the space
of $r$-tuples $(f_1(z),\cdots ,f_r(z))\in \C [z]^r$
of (not necessarily monic) polynomials
satisfying the following  condition:
%%%(6.3.1)%%%
\begin{enumerate}
%%%%%%%%%%%%%
\item[(\ref{eq: inclusion stab}.1)]
For any $\sigma =\{i_1,\cdots ,i_s\}\in I(\mathcal{K}_{\Sigma})$,
the polynomials $f_{i_1}(z),\cdots ,f_{i_s}(z)$ have no common root 
of multiplicity $\geq n$ in $X$.
%%%%%%%%%%%%%%%
\end{enumerate}
%%%%%%%%%%%%%%%
%%
%%
%%
%%
Define the map
%%()%%
%\begin{equation}
%%%%
$i_{n}^{\Sigma}:X\to 
\mathcal{Z}_{\KS}(\C^n,(\C^n)^*)$
%\quad
%\mbox{ by }
%\end{equation}
%%%%%%
by
%%(6.4)%%
\begin{equation}
%%%%
i_{n}^{\Sigma}(f)(\alpha)=
(F_n(f_1)(\alpha),F_n(f_2)(\alpha ),\cdots ,
F_n(f_r)(\alpha))
\end{equation}
%%%%%%
for
$(f,\alpha)
=((f_1(z),\cdots ,f_r(z)),\alpha)\in F^{\Sigma}_n(X)\times X.$
%and represents a map
%$X\to \XS$.
%%
\par
%%%(ii)%%
(ii) 
%%%%%%%%
Let $U=\{w\in \C:\vert w\vert <1\}$ and let
%%(6.5)%%
\begin{equation}
%%%%%%%%
ev_0:F^{\Sigma}_n(U)\to \mathcal{Z}_{\KS}(\C^n,(\C^n)^*)
\end{equation}
%%%%% 
denote the map given by evaluation at $0$, i.e.
%%(6.6)%%
\begin{equation}
%%%%%%%%
ev_0(f)=(F_n(f_1)(0),F_n(f_2)(0),\cdots ,F_n(f_r)(0))
%%%%%%%%%
\end{equation}
%%%%%%%%
for
$f=(f_1(z),\cdots ,f_r(z))\in F^{\Sigma}_n(X)$.
\par
%%%(iii)%%
(iii)
%%%%%%%%%
Let $\tilde{F}^{\Sigma}_n(U)\subset F^{\Sigma}_n(U)$ denote the subspace
%consisting 
of all
$(f_1(z),\cdots ,f_r(z))\in F^{\Sigma}_n(X)$ such that no $f_i(z)$ is identically
zero, and let
%%(6.7)%%%%
\begin{equation}
%%%%%%%%%%%
ev:\tilde{F}^{\Sigma}_n(U)\to 
\mathcal{Z}_{\KS}(\C^n,(\C^n)^*)
%%%%%%%%%%
\end{equation}
%%%%%%%%%%
be
the map given by the restriction
$ev=ev_0\vert \tilde{F}^{\Sigma}_n(U)$.
\par
(iv)
We denote by
%%(6.8)%%
\begin{equation} 
u:\tilde{F}^{\Sigma}_n(U)/\T^r_{\C}\to E^{\Sigma}_n(\overline{U}, \partial \overline{U})
\end{equation}
%%%%%
the natural map which
assigns to an $r$-tuple  
$[f_1(z),\cdots ,f_r(z)]\in \tilde{F}^{\Sigma}_n(U)/\T^r_{\C}$ of polynomials
the $r$-tuple of configurations in $\C^n$
represented by their roots of
$F_n(f_1)(z),\cdots ,F_n(f_r)(z)$
which lie in $U$.
\qed
%It may be shown by the argument given in \cite[Proposition 4.8]{Se} that the map $u$ is a homotopy equivalence.
%%%
}
%%%%%%%%
\end{dfn}
%%(End of Definition 6.3)%%%%%
%%%
%%(Lemma 6.4)%%
\begin{lmm}\label{lmm: ev}
%%%%%%%%%%%%%%
%$\I$
The map
$ev:\tilde{F}^{\Sigma}_n(U)\stackrel{\simeq}{\longrightarrow}
\mathcal{Z}_{\KS}(\C^n,(\C^n)^*)$
is a homotopy equivalence.
%\par
%$\II$
%The map
%$u:\tilde{F}^{\Sigma}_n(U)/\T^r_{\C}
%\stackrel{\simeq}{\longrightarrow} 
%E^{\Sigma}_n(\overline{U}, \partial \overline{U})
%$
%is a homotopy equivalence.
%%%
\end{lmm}
%%%(Proof of Lemma 6.4)%%%%%
\begin{proof}
%%%%%
For each 
$\textbf{\textit{b}}=(b_0,b_1,\cdots ,b_{n-1})\in \C^n$, let
$f_{\textbf{\textit{b}}}(z)\in \C [z]$ denote the polynomial
defined by
%%()%%
%\begin{equation}
%%%
$
f_{\textbf{\textit{b}}}(z)=b_0+
\sum_{k=1}^{n-1}\frac{(b_k-b_0)z^k}{k!},
$
%%%
%\end{equation}
%%%
and define the map $i_0:
\mathcal{Z}_{\KS}(\C^n,(\C^n)^*)
\to F^{\Sigma}_n(U)$ by 
$
i_0(\textbf{\textit{b}}_1,\cdots ,\textbf{\textit{b}}_r)
=(f_{\textbf{\textit{b}}_1}(z),\cdots ,
f_{\textbf{\textit{b}}_r}(z))
$
for $(\textbf{\textit{b}}_1,\cdots ,\textbf{\textit{b}}_r)
\in \mathcal{Z}_{\KS}(\C^n,(\C^n)^*).$
Since the degree of the each polynomial
$f_{\textbf{\textit{b}}_k}(z)$ is at most $n-1$,
it has no root of multiplicity $\geq n$ and the map 
$i_0$ is well-defined.
Clearly $ev_0\circ i_0=\mbox{id}$.
\par
On the other hand,
let $\Phi: F^{\Sigma}_n(U)\times [0,1]\to
F^{\Sigma}_n(U)$ be
the homotopy given by
$
\Phi ((f_1(z),\cdots ,f_r(z)),t)=(f_{1}(tz),\cdots ,f_{r}(tz)).
$
This gives a homotopy between
$i_0\circ ev_0$ and the identity map, and this proves that $ev_0$ is a homotopy equivalence.
Since $F^{\Sigma}_n(U)$ is an infinite dimensional manifold and
$\tilde{F}^{\Sigma}_n(U)$ is a closed subspace of $F^{\Sigma}_n(U)$ of infinite codimension, by
using \cite[Theorem 2]{EK}, one can show that
the inclusion
$\tilde{F}^{\Sigma}_n(U)\to F^{\Sigma}_n(U)$ is a homotopy equivalence.
Hence $ev$ is also a homotopy equivalence.
%%
%\par
%(ii)
%This easily follows from \cite[(iii) of Lemma 8.4]{KY11}.
%%%
\end{proof}
%%%%%%%%%(End of proof of Lemma 6.4)%%%%
%%%
%%%
%%%
Now it is ready to prove Theorem \ref{thm: IV}.

%%%%%%%%%%%%%%%%%%%%%%%%%%
%%%(Proof of stability result)%%
%%%%%%%%%%%%%%%%%%%%%%%%%%
%%%(Proof of Theorem 6.2)%%
\begin{proof}[Proof of Theorem \ref{thm: IV}]
%%%%%%%%%%%%%%%%%%%%%%%%%%
Let $U=\{w\in\C: \vert w\vert <1\}$ as before and
note that the group $\T^r_{\C}$ acts freely on the space
$\tilde{F}^{\Sigma}_n(X)$
by coordinate multiplication for $X=U$ or $\C$.
Let $\tilde{F}^{\Sigma}_n(X)/\T^r_{\C}$ denote the corresponding orbit space.
%%
%Let $u:\tilde{F}^{\Sigma}_n(U)/\T^r_{\C}\to E^{\Sigma}_n(\overline{U}, \partial \overline{U})$
%denote the natural map which
%assigns to an $r$-tuple  
%$[f_1(z),\cdots ,f_r(z)]\in \tilde{F}^{\Sigma}_n(U)/\T^r_{\C}$ of polynomials
%the $rn$-tuple of their roots of
%$F_n(f_1)(z),\cdots ,F_n(f_r)(z)$
%which lie in $U$.
%It may be shown by the argument given in \cite[Proposition 4.8]{Se} that the map $u$ is a homotopy equivalence.
Note that $u:\tilde{F}^{\Sigma}_n(U)/\T^r_{\C}
\stackrel{\simeq}{\longrightarrow} 
E^{\Sigma}_n(\overline{U}, \partial \overline{U})
$
is a homotopy equivalence.
Indeed, this  follows from \cite[(iii) of Lemma 8.4]{KY11}.
Now
let  
$scan: \tilde{F}^{\Sigma}_n(\C)\to 
\Map (\C,\tilde{F}^{\Sigma}_n(U))$ denote the map
given by 
$$
scan (f_1(z),\cdots ,f_r(z))(w)=(f_1(z+w),\cdots ,f_r(z+w))
$$
for $w\in\C$, and
 consider the diagram
$$
\begin{CD}
\tilde{F}^{\Sigma}_n(U) @>ev>\simeq>
\mathcal{Z}_{\KS}(\C^n,(\C^n)^*)
\\
@V{p}VV @.
\\
\tilde{F}(U)^{\Sigma}_n/\T^r_{\C} @>u>\simeq> 
E^{\Sigma}_n(\overline{U},\partial \overline{U})
\end{CD}
$$
where
$p:\tilde{F}^{\Sigma}_n(U)\to \tilde{F}^{\Sigma}_n(U)/\T^r_{\C}$
denotes the natural projection map.
Note that $p$ is a $\T^r_{\C}$-principal bundle projection. 
%%%
%%%
Consider the diagram below
$$
\begin{CD}
\tilde{F}^{\Sigma}_n(\C) @>scan>> 
\Map (\C, \tilde{F}^{\Sigma}_n(U)) 
@>ev_{\#}>\simeq> \Map (\C,\mathcal{Z}_{\KS}(\C^n,(\C^n)^*))
\\
@V{p}VV @V{p_{\#}}VV @.
\\
\tilde{F}^{\Sigma}_n(\C)/\T^r_{\C} @>scan>>
\Map (\C, \tilde{F}^{\Sigma}_n(U)/\T^r_{\C}) 
@>u_{\#}>{\simeq}>
\Map (\C, E^{\Sigma}_n(\overline{U},\partial \overline{U}))
\end{CD}
$$
induced from the above diagram.
Observe that $\Map (\C,\cdot )$ can be replaced by $\Map^* (S^2,\cdot)$
by extending  from $\C$ to $S^2=\C\cup \infty$
(as base point preserving maps).
%%%
%%%
Thus by setting
$$
\begin{cases}
j_D^{\p}:\po^{D,\Sigma}_n \stackrel{\subset}{\longrightarrow}
 \tilde{F}^{\Sigma}_n(\C) 
\stackrel{scan}{\longrightarrow}
\Map_D^*(S^2,\tilde{F}^{\Sigma}_n(U))
=\Omega^2_D\tilde{F}^{\Sigma}_n(U)
\\
j_D^{\p\p}:E^{D,\Sigma}_n(\C)
\stackrel{\subset}{\longrightarrow}
\tilde{F}^{\Sigma}_n(\C)/\T^r_{\C} \stackrel{scan}{\longrightarrow}
\Map_D^*(S^2,\tilde{F}^{\Sigma}_n(U)/\T^r_{\C})
=\Omega^2_D(\tilde{F}^{\Sigma}_n(U)/\T^r_{\C})
\end{cases}
$$
we obtain the following commutative diagram, where
the suffix $D$ denotes the appropriate path component:
%%%%(The main diagram)%%
$$
\begin{CD}
%%%%%%
\po^{D,\Sigma}_n
@>j_D^{\p}>> \Omega^2_D\tilde{F}^{\Sigma}_n(U) 
@>\Omega^2 ev>\simeq> \Omega^2_D
\mathcal{Z}_{\KS}(\C^n,(\C^n)^*)
@>\Omega^2q_{\Sigma}>\simeq> \Omega^2_D \XS (n)
\\
@V{\cong}VV 
 @V{\Omega^2p}V{\simeq}V @. @.
\\
E^{\Sigma}_D(\C)
@>j_D^{\p\p}>>
\Omega^2_D (\tilde{F}^{\Sigma}_n(U)/\T^r_{\C}) 
@>\Omega^2 u>{\simeq}>
\Omega^2_DE^{\Sigma}_n(\overline{U},\partial \overline{U})
@.
\end{CD}
$$
Note that the maps
$\Omega^2q_{\Sigma}$, $ev$, $\Omega^2p$ and $u$ are homotopy equivalences.
Moreover, from the definitions of the maps, one can see that the following two equalities hold (up to homotopy equivalence):
%%%%%%
%%(6.9)%%
\begin{equation}
%%%%%%%%%%%
\Omega^2q_{\Sigma}\circ \Omega^2ev\circ  j_D^{\p}=i_D,
\quad 
\Omega^2u\circ j_D^{\p\p}=sc_D.
%%%%%%%%%%
\end{equation}
%%%
Hence, the maps $i_D$ and $sc_D$ are homotopic up to homotopy equivalences.
Thus, if we replace $D$ by $D+k\textit{\textbf{a}}$
and let $k\to\infty$
then,
by using Theorem \ref{thm: scanning map},
we see that the map $i_{D+\infty}$ is a homotopy equivalence.
%%%%%%%%%%%%%%%%%%%%%%%%
\end{proof}
%%(End of Proof of Theorem 6.2)%%%
%%%%%%%
%%
%%(End of SECTION 6)%%%
%%%
%%%
%%%
%%%
%%%
%%%%
%%%%%%%%%%%%%%%%%%%%%%%%%%%
%%(Proof of main results)%%%
%%%(SECTION 7)%%%%%%%%%%%%%%%%%%%%
\section{Proofs of the main results}\label{section: proofs}
%%%%%%%%%%%%%%%%%%%%%%%%%%%

In this section we prove Theorem \ref{thm: I} and Corollary \ref{crl: II}.
For this purpose,  from now on we always assume that
 $\XS$ is a simply connected smooth toric variety such that 
the condition $($\ref{equ: homogenous}.1$)$ 
is satisfied.
%%
%%
%%%%
Now we can prove the main results.
%%
%%
%%
%%%%%%%%%%%%%%%%%%%%%%%%%%%%%%%%
%%%(Proof of Theorem 2.10)%%
\begin{proof}[Proof of Theorem \ref{thm: I}]
%%%%%%%%%%%%%%%%%%%%%%%%%%%%%%%%
%%(i)%%
The assertion (i) follows from Corollary \ref{crl: III} and
Theorem
\ref{thm: IV}.
It remains to show (ii) and
suppose that $\sum_{k=1}^rd_k\textbf{\textit{n}}_k\not= {\bf 0}_m$.
%%.
\par
By the assumption
$($\ref{equ: homogenous}.1$)$,
%%%%
there is an $r$-tuple $D_*=(d_1^*,\cdots ,d_r^*)\in \N^r$ 
%of positive integers 
such that
$\sum_{k=1}^rd_k^*\textit{\textbf{n}}_k={\bf 0}_m.$
Then if we choose a sufficiently large positive integer $n_0$,
the following equality holds:
%%(7.1)%%
\begin{equation}
%%%%%%%%
\textbf{\textit{a}}=n_0D_*-D
=(n_0d_1^*-d_1,\cdots ,n_0d_r^*-d_r)\in \N^r.
%%%%%%%%%%%
\end{equation}
%%%%
Since the $r$-tuple $n_0D_*=D+\textbf{\textit{a}}\in \N^r$ satisfies the condition
(\ref{equ: homogenous}.1), the map
$i_{D+\textbf{\textit{a}}}$ is well-defined.
Then one can
define the map
%%(7.2)%%
\begin{equation}\label{eq: mapD}
%%%%%%%%
j_D:\po^{D,\Sigma}_n\to
\Omega^2\mathcal{Z}_{\KS}(D^{2n},S^{2n-1})
%%%%%%%%%
\end{equation}
%%%
by the composite $j_D=(i_{D+\textbf{\textit{a}}})\circ 
(s_{D,D+\textbf{\textit{a}}})$,
%%%()%%
{\small
\begin{equation*}
\po^{D,\Sigma}_n
\stackrel{s_{D,D+\textbf{\textit{a}}}}{\longrightarrow}
\po^{D+\textbf{\textit{a}},\Sigma}_n
%=\po^{n_0D_*,\Sigma}_n
\stackrel{i_{D+\textbf{\textit{a}}}}{\longrightarrow}
\Omega^2_{D+\textbf{\textit{a}}}\XS (n)
\simeq
\Omega^2_0\XS (n)\simeq
\Omega^2\mathcal{Z}_{\KS}(D^{2n},S^{2n-1}).
%%%%%%
\end{equation*}
}
%%%
\par\noindent{}Note that the two maps $s_{D,D+\textbf{\textit{a}}}$
and $i_{D+\textbf{\textit{a}}}$ are homotopy equivalences through
dimensions $d(D;\Sigma,n)$ and $d(D+\textbf{\textit{a}};\Sigma,n)$
(by Corollary \ref{crl: III} and Theorem \ref{thm: I}).
Since $d(D;\Sigma,n)\leq d(D+\textbf{\textit{a}};\Sigma,n)$, the map $j_D$ is a homotopy equivalence through dimension
$d(D;\Sigma,n)$.
%%%
\end{proof}
%%(End of proof of Theorem 2.10)%%%
%%
%%
%%
%%
%%
%%
%%(Proof of Corollary 2.11)%%%%%%%%
\begin{proof}[Proof of Corollary \ref{crl: II}.]
%%%%%%%%%%%%%%%%%%%%%%%%%%%%%%%%%%%
Let $\XS$ be a compact smooth toric variety such that
$\Sigma (1)=\{\mbox{Cone}(\textit{\textbf{n}}_k):1\leq k\leq r\}$, 
where $\{\textit{\textbf{n}}_k\}_{k=1}^r$ are  primitive generators
as in Definition \ref{dfn: fan}.
Since $\XS$ is a compact, by (ii) of Lemma \ref{lmm: toric} we easily see that
the condition (\ref{equ: homogenous}.1) 
is satisfied for $\XS$.
Since $\Sigma_1\subsetneqq \Sigma$, by using
Lemma \ref{lmm: toric}
we  see that $X_{\Sigma_1}$ is a non-compact smooth toric subvariety of $\XS$.
Moreover, since
 $\Sigma (1)\subset \Sigma_1\subsetneqq \Sigma$, we see that
$\Sigma_1(1)=\Sigma (1)$.
Hence,  
the condition (\ref{equ: homogenous}.1) holds for $X_{\Sigma_1}$, too. 
%Note that the condition  (\ref{equ: homogenous}.2) also holds for ${X}_{\Sigma_1}$
%by the assumption.
Thus,
the assertion follows from Theorem \ref{thm: I}.
%%%
\end{proof}
%%(End of proof of Corollary 2.12)%%%%
%%
%%
%%
%%(End of section 7)%%%
%%
%%()%%%%%%
\par\vspace{1mm}\par
\noindent{\bf Acknowledgements. }
%%%%%%%%
The authors should like to take this opportunity to thank
Professors Martin Guest and  
 Masahiro Ohno 
for his many valuable  insights and suggestions concerning toric varieties
and scanning maps.
%%%
%\par
The second author was supported by 
JSPS KAKENHI Grant Number 18K03295.
This work was also supported by the Research Institute of Mathematical
Sciences, a Joint Usage/Research Center located in Kyoto University.

%%%%%%%
%\par\vspace{1mm}\par
%\par\noindent{\bf Acknowledgements. }
%%%%%%%%%
%The second author was supported by 
%JSPS KAKENHI Grant Number  18K03295. 
%This work was also supported by the Research Institute for Mathematical Sciences, 
%a Joint Usage/Research Center located in Kyoto University.

%% The Appendices part is started with the command \appendix;
%% appendix sections are then done as normal sections
%% \appendix

%% \section{}
%% \label{}

%% References
%%
%% Following citation commands can be used in the body text:
%% Usage of \cite is as follows:
%%   \cite{key}          ==>>  [#]
%%   \cite[chap. 2]{key} ==>>  [#, chap. 2]
%%   \citet{key}         ==>>  Author [#]

%% References with bibTeX database:

%\bibliographystyle{model1-num-names}
%\bibliography{<your-bib-database>}

%% Authors are advised to submit their bibtex database files. They are
%% requested to list a bibtex style file in the manuscript if they do
%% not want to use model1-num-names.bst.

%% References without bibTeX database:

%%%%%%%%%%%%%%%%%%%

\end{document}